\documentclass[reqno,12pt]{amsart}
\usepackage{amssymb}
\usepackage{amsfonts}
\usepackage{amsmath}
\usepackage{mathrsfs}
\usepackage{diagbox}
\usepackage{color,soul}
\usepackage{amsthm, graphicx,empheq}

\usepackage{framed}
\usepackage[backref,colorlinks,linkcolor=red,anchorcolor=green,citecolor=blue]{hyperref}
\usepackage{cleveref}

\newtheorem{lemma}{Lemma}[section]
\newtheorem{theorem}{Theorem}[section]

\theoremstyle{definition}
\newtheorem{definition}{Definition}[section]
\theoremstyle{remark}
\newtheorem{remark}{Remark}[section]
\theoremstyle{example}

\numberwithin{figure}{section}
\numberwithin{equation}{section}
\numberwithin{table}{section}

\def\dl{\displaystyle}

\newcommand{\p}{\partial}
\newcommand{\norm}[1]{\left\Vert#1\right\Vert}
\newcommand{\dd}{\mathrm{d}}

\newcommand{\di}{\mathrm{div}}
\newcommand{\grad}{\mathrm{grad}}

\makeatletter

\newcommand{\Rmnum}[1]{\expandafter\@slowromancap\romannumeral#1@}
\makeatother

\allowdisplaybreaks
\begin{document}
\title[Stationary Subsonic  Euler Flows  with
Mass-Additions]{Stability of Stationary Subsonic Compressible Euler Flows with
Mass-Additions in Two-Dimensional Straight Ducts}

\author{Junlei Gao}
\author{Hairong Yuan}
\address{
Department of Mathematics, Zhejiang Normal University, Jinhua 321004, China} \email{gaojunlei123math@zjnu.edu.cn}

\address{
School of Mathematical Sciences and Shanghai Key Laboratory of Pure Mathematics and Mathematical Practice, East China Normal University, Shanghai
200241, China} \email{hryuan@math.ecnu.edu.cn}



\date{\today}

\begin{abstract}
We show existence, uniqueness and stability for a family of stationary subsonic compressible Euler flows with mass-additions in two-dimensional rectilinear ducts, subjected to suitable time-independent multi-dimensional boundary conditions at the entrances and exits.The stationary subsonic Euler equations consist a quasi-linear system of elliptic-hyperbolic composite-mixed type, while addition-of-mass destructs the usual methods based upon conservation of mass and Lagrangian coordinates to separate the elliptical and hyperbolic modes of the system. We establish a new decomposition and nonlinear iteration scheme to overcome this major difficulty. It reveals that mass-additions introduce very strong interactions in the elliptic and hyperbolic modes, and lead to a class of second-order elliptic equations with multiple integral nonlocal terms. The linearized problem is solved by studying algebraicand analytical properties of infinite weakly coupled boundary-value problems of ordinary differential equations, each with multiple nonlocal terms, after applications of Fourier analysis methods.
\end{abstract}

\maketitle
\tableofcontents
\section{Introduction}\label{sec1}
We  are concerned with subsonic compressible Euler flows in a two-dimensional straight duct with constant cross-sections, while mass is added with a rate proportional to the density of the flow.  An example of such case is ejecting of petrol to the flow in the combustion chamber of a subsonic ramjet \cite{price}. 
To simplify the mathematical analysis, we ignore the chemical reactions, as well as changes of momentum and total energy, and consider a model problem presented by the following compressible Euler equations with mass-additions (more assumptions will be specified in the sequel):
\begin{eqnarray} \label{Euler equation1} \displaystyle
\begin{cases}
\p_t \rho+\di (\rho \mathbf{u})=\mathfrak{M}(\mathbf{x}, t), \\
\p_t (\rho \mathbf{u})+\di (\rho \mathbf{u}\otimes \mathbf{u})+\grad\, p=\mathbf{0}, \\
\p_t(\rho(e+\frac1{2}|\mathbf{u}|^2)) +\di ((\rho(e+\frac1{2}
|\mathbf{u}|^2)+p)\mathbf{u})=0.
\end{cases}
\end{eqnarray}
Here `$\di$' and `$\grad$' denote correspondingly the divergence operator and gradient operator about the spatial position ${\bf{x}}$ in the Euclidean plane $\mathbb{R}^2$, and $t \in \mathbb{R}$ represents time. 
The unknowns $\rho,\, \mathbf{u},\, p,\,e$ are density of mass, velocity field, pressure and internal energy of the fluid, respectively. The function $\mathfrak{M}(\mathbf{x}, t)=\lambda \rho m(\mathbf{x})$ shows the change of mass per unit volume in unit time, with $\lambda$ a number and $m(\mathbf{x})$ a bounded positive function. Thus $\lambda m(\mathbf{x})$ denotes the mass of, say, petrol, added to unit mass of air in the duct in unit time. Notice that we also allow $\mathfrak{M} < 0$ (i.e., $\lambda<0$), which means the mass is absorbed from the flow.

The pressure $p$ is usually given by an equation of state $p=p(\rho,s)$, where $s$ is the entropy. The local sound speed is defined as $c(\rho,s)=\sqrt{\frac{\p}{\p \rho}p(\rho,s)}$. The non-dimensional number $M=|\mathbf{u}|/c$ is called Mach number. If $M>1$\,(\,$M<1$) at a space-time point, the flow is said supersonic (subsonic) there. To be specific, in this paper, we consider polytropic gases, i.e., the equation of state is $p=A(s)\rho^\gamma$, where $\gamma > 1$ is the adiabatic exponent,  $A(s)=\kappa_0\exp(s/c_\nu)$, and $\kappa_0$,  $c_\nu$ are positive constants. We then have
\begin{eqnarray*}
e=c_\nu T=\frac{1}{\gamma-1}\frac{p}{\rho}, \ \ p=R\rho T,\ \  c=\sqrt{\frac{\gamma p}{\rho}}, \ \ c_\nu=\frac{R}{\gamma-1},\ \ T=\frac{c^2}{\gamma R},
\end{eqnarray*}
where $T>0$ is temperature of the flow, and $R>0$ is the universal gas constant.

The flow is called steady or stationary, if it is time-independent. Then \eqref{Euler equation1} reads
\begin{eqnarray}
\varphi&\triangleq&\di (\rho \mathbf{u}\otimes \mathbf{u})+\grad\, p=\mathbf{0},\label{4101}\\
\varphi_1&\triangleq&\di (\rho \mathbf{u})-\lambda \rho m(\mathbf{x})=0,\label{4102}\\
\varphi_2&\triangleq&\di (\rho E \mathbf{u})=0,\label{4103}
\end{eqnarray}
where
\begin{equation}\label{eq14addnew}
E\triangleq\frac{1}{2}|\mathbf{u}|^2 + \frac{c^2}{\gamma - 1}
\end{equation}
is the total enthalpy, which is also called as Bernoulli constant.
We only consider the two-space-dimensional case, namely $\mathbf{x}=(x,y)\in\mathbb{R}^2$, hence the velocity $\mathbf{u}=(u,v)^{\top}$. In the following, we usually call  $u$ ($v$, resp.) the normal (tangential, resp.) velocity of the flow, since the duct with length $l$ is given by
\begin{align}\label{eq16}
\Omega=\{(x,y):~ 0<x<l,~0<y<\pi\},
\end{align}
and the flow is assumed to move from the entrance $\Sigma_0=\{(0,y):~0<y<\pi\}$ to the exit $\Sigma_l=\{(l,y):~0<y<\pi\}$, thus $u$ is the velocity component normal to each cross-section $\Sigma_x=\{(x,y):~ 0<y<\pi\}$ of the duct, while $v$ is the component tangent to it.

We now specify the boundary conditions subjected to the equations \eqref{4101}-\eqref{4103} and formulate the main problem  to study in this paper.
On the entry $\Sigma_{0}$, it is required that
\begin{eqnarray}\label{4202}
E(0,y)=E_0(y),\quad s(0,y)=s_0(y), \quad v(0,y)=v_0(y),
\end{eqnarray}
and $E_0(y), s_0(y), v_0(y)$ are given functions on $[0,\pi]$.
The pressure is prescribed at the exit:
\begin{eqnarray}\label{4203}
p=p_l(y)\qquad\text{on}~ \Sigma_{l}.
\end{eqnarray}
The solid lateral walls of the duct are denoted by $W_{-}=\{(x,0):~0<x<l\}$ and $W_{+}=\{(x,\pi):~0<x<l\}$,  on which the natural slip condition is subjected:
\begin{eqnarray}\label{4204}
v \equiv 0 \qquad\text{on}~ W_{\pm}.
\end{eqnarray}

\begin{framed}
\underline{Problem  \uppercase\expandafter{\romannumeral1}}: Is it well-posed for classical solutions representing subsonic flows to the Euler equations \eqref{4101}-\eqref{4103}, with boundary conditions \eqref{4202}-\eqref{4204} in the domain $\Omega$?
\end{framed}

A classical strategy to Problem \uppercase\expandafter{\romannumeral1} is as follows. For given mass-addition $\lambda m_b(x)$, we construct a family of unidimensional subsonic classical solutions $U_b$, which are uniform on each cross-section $\Sigma_x$ but vary along $x$. They are called {\em{background solutions}}. Notice that for different $\lambda\in\mathbb{R}$, the background solutions $U_b$ are different, even if they might have the same subsonic state at the entrance $\Sigma_0$. There is also a critical length $l_*$ determined by $U_b(0), m_b(x)$ and $\lambda$ (cf. \eqref{4706}, \eqref{4707}), so that the flow is subsonic, vacuum-free, and has no stagnation point, if the duct is shorter than $l_*$. The details could be found in Section \ref{sec41}. 
Then we solve Problem \uppercase\expandafter{\romannumeral1} with boundary conditions being suitable small perturbations of a background solution.

\begin{theorem}[Main theorem]\label{ThmSY}
For given uniform subsonic state $U_b(0)$ at the entrance $\Sigma_0$ and positive $C^\infty$ function $m_b(x)$ on $\{x\ge0\}$, there is a countable set $\mathcal{S}\subset\mathbb{R}$. Let  $U_b$ be a background solution determined by $U_b(0)$, $m_b(x)$ and $\lambda\in\mathbb{R}\setminus\mathcal{S}$, with $l_*$ the corresponding critical length. Suppose that the length of the duct $l$ is less than $l_*$.

Then for any $\alpha\in (0, 1)$,  there exist positive constants $\epsilon_0$ and $C$ depending only on $U_b$ and  $l,\,\alpha$, such that if the perturbations of boundary conditions and mass-addition are small:
\begin{eqnarray}\label{4206}
&&\|E_0(y) - E_b(0)\|_{C^{3,\alpha}(\Sigma_0)} + \|s_0(y) - s_b(0)\|_{C^{3,\alpha}(\Sigma_0)}+\|v_0(y)\|_{C^{3,\alpha}(\Sigma_0)} \nonumber\\
&&\quad \quad + \|p_l(y) - p_b(l)\|_{C^{3,\alpha}(\Sigma_l)} + \|m(\mathbf{x}) - m_b(x)\|_{C^{3,\alpha}(\bar{\Omega})} \leq \epsilon \leq \epsilon_0,
\end{eqnarray}
and the symmetry conditions hold:
\begin{eqnarray}\label{4207}
&\frac{\dd E_0(y)}{\dd y}=\frac{\dd s_0(y)}{\dd y} = \frac{\dd^2 v_0(y)}{\dd y^2}=\frac{\dd p_l(y)}{\dd y}=\frac{\dd^3 p_l(y)}{\dd y^3}=0, \quad &\text{for}\quad y=0,\,\pi;\nonumber\\
&\quad \p_y m(\mathbf{x})=\p^2_y m(\mathbf{x}) \equiv 0,\quad &\text{on} \quad W_- \cup W_+,
\end{eqnarray}
then the answer to Problem $\,\uppercase\expandafter{\romannumeral1}$ is affirmative:
  there is a unique subsonic flow $U$ with the following properties:
\begin{itemize}
\item[i)] $U$ satisfies the equations \eqref{4101}-\eqref{4103} and boundary conditions \eqref{4202}-\eqref{4204} point-wisely;
\item[ii)] symmetry:
\begin{eqnarray}\label{4208}
\p_y E = \p_y s = \p_y p= \p^3_y p = \p^2_y v\equiv 0\quad \text{on} \quad W_- \cup W_+;
\end{eqnarray}
\item[iii)] stability:
\begin{eqnarray}\label{4209}
\|E - E_b\|_{C^{3,\alpha}(\bar{\Omega})} + \|s - s_b\|_{C^{3,\alpha}(\bar{\Omega})} + \|p - p_b\|_{C^{3,\alpha}(\bar{\Omega})} + \|v\|_{C^{3,\alpha}(\bar{\Omega})} \leq C\epsilon.
\end{eqnarray}
\end{itemize}
\end{theorem}

\begin{remark}
To avoid singularity at the corners $(\Sigma_0\cup\Sigma_l)\cap(W_- \cup W_+),$ we require that the boundary conditions and mass-addition satisfy the symmetry conditions \eqref{4207}. They are compatibility conditions for  classical solutions, since, for example,
the momentum equation $\eqref{4101}$ and slip condition $\eqref{4204}$ yield
\begin{eqnarray}\label{4205}
\p_yp = 0\quad \text{on} \quad  W_-\cup W_+.
\end{eqnarray}
It is not clear presently how to obtain classical solutions which are only continuously differentiable. \qed
\end{remark}

The major obstruction to prove Theorem  \ref{ThmSY} is the elliptic-hyperbolic composite-mixed type of the steady Euler system  \eqref{4101}-\eqref{4103} for subsonic flows, which means it has a real eigenvalue with multiplicity two and a pair of conjugate complex eigenvalues, see \cite[p.1355]{Yuan-2006}. For the two-space-dimensional case, thanks to the conservation of mass,  Lagrangian coordinates or stream function formulation could be introduced. Then by characteristic decompositions, one may separate the elliptic modes indicated by the complex eigenvalues from the hyperbolic or transport modes associated with the real eigenvalue, which is  linearly degenerate. This procedure is established by S. Chen \cite{Chen-2006} in 2006 and widely used to study subsonic flows and transonic shocks in ducts, provided conservation of mass is valid, see for example, \cite{chen-deng-xiang2012,duduan2016,fang-gao-2021,Fang-Xin-2020,Gao-Liu-Yuan-2020,Liu-2008,liuyuan2008,Xie-Xin-2007,Xie-Xin-2010,Yuan-2006} and references therein.  Obviously it fails for the present case of mass-additions. The other approach independent of conservation of mass and works for three-space-dimensional case is proposed in by Chen and Yuan in \cite{Chen-Yuan-2008}. A general decomposition lemma and iteration scheme were established later by Liu, Xu, and Yuan \cite{Liu-Xu-Yuan-2016}, and then adopted in \cite{Yuan-Zhao-2021,Yuan-Zhao-2020} to study subsonic flows and transonic shocks in ducts with frictions taken into account.  This approach could be used to study mass-additions considered in this paper. However, a disadvantage is that the  obtained physical quantities do not have the same regularity: pressure is one order smoother than other quantities such as  density, entropy, and velocity. That is why we pursue a new decomposition lemma and iteration scheme in this paper, to show that each unknown of the flow field actually has the same regularity. This is the main contribution of the work. We find that the boundary value problem for the Euler system formulated in Problem I could be decoupled to Cauchy problem for two transport equations of entropy and total enthalpy, a mixed boundary value problem for a second-order (elliptic) equation of pressure, and two-point boundary value problems for ordinary differential equations of tangential velocity on each cross-section of the duct.  For the last problem,  it is fascinating to see that the solvability conditions, necessary for the linearized problems, are fulfilled automatically for the nonlinear problem. This reveals a beautiful intrinsic structure hiding in the Euler system.

The other discovery is the very strong interactions induced by mass-addition in the elliptic and hyperbolic modes, which can be seen from the derivation of a second-order elliptic equation with up to six integral-type nonlocal terms, to design a contractive nonlinear iteration scheme, cf. Remark \ref{rmkp2}. It is more challenging than the nonlocal problems appeared in the previous studies of subsonic flows and transonic shocks in ducts with geometric effects, or friction, or heat-exchanges  \cite{Gao-Liu-Yuan-2020,liuyuan2008,Yuan-Zhao-2021,Yuan-Zhao-2020}. This difficulty arises because we are studying  perturbations of {\it large solutions} of the Euler equations.  We adapt the framework of \cite{Liu-Xu-Yuan-2016} to solve the linearized nonlocal elliptic problems, which combines Fourier analysis, linear algebra, as well as theory of real analytic functions. To guarantee uniqueness of solutions, it is found that a sufficient condition is to exclude certain background solutions determined by  numbers $\lambda$ in a countable set $\mathcal{S}\subset\mathbb{R}$ as indicated in Theorem \ref{ThmSY}, cf. Lemma \ref{Lem44}. To obtain a priori estimates,  we introduce a class of anisotropic H\"{o}lder spaces to solve the transport equations and ordinary differential equations with parameters, which incorporates the Schauder theory of elliptic equations, and might be of independent interests, see Sections \ref{sec53}-\ref{sec8}.  We wish these new observations would be helpful to study other problems for compressible Euler equations.

Apart from what mentioned above, we briefly review here some other previous studies on stationary subsonic compressible Euler equations. More references and details could be found in \cite{Gao-Liu-Yuan-2020,Yuan-Zhao-2020}. For the isentropic ir-rotational flows, Shiffman firstly studied the existence of subsonic flow by using the variational method \cite{Shiffman-1952}. Bers proved the existence and uniqueness of two-dimensional subsonic flows by using the complex analysis method \cite{Bers-1954}. Finn and Gilbarg established the existence and uniqueness for three-dimensional subsonic flows passing bodies \cite{Finn-Gilbarg-1957-1}. The result was generalized by Dong and Ou \cite{dong-1993}, and then later extended by Liu and Yuan to study subsonic potential flows in largely-open nozzles \cite{liuyuan2014}.  For the compressible Euler equations,  Xie and Xin \cite{Xie-Xin-2007} firstly studied isentropic non-swirl subsonic flows in two-dimensional infinite duct with variable boundaries,  and  proved that there is a critical mass flux below which a global subsonic flow exists in the duct.  See also  \cite{bae,chen-deng-xiang2012,duduan2016,Huang-Wang-Wang-2011,Liu-2008,Xie-Xin-2007,Xie-Xin-2010,Xie-Xin-2010-1} for further results.

The rest of the paper is organized as follows. In Section \ref{sec41}, we  construct background solutions, which are unidimensional subsonic flows with mass additions in the ducts.  In Section \ref{sec42}, we present the new  decomposition  of stationary compressible Euler equations aforementioned. 
In Section \ref{sec43}, the equations and boundary conditions  are linearized at a given background solution. Then Problem \Rmnum{1} could be rewritten as {Problem \uppercase\expandafter{\romannumeral2}} (see Section \ref{sec44new}). In Section \ref{sec44}, we study well-posedness and regularity of solutions to a mixed boundary value problem of second-order elliptic equations with multiple integral nonlocal terms. It is based on Galerkin method and Schauder estimates of elliptic equations. In Section \ref{sec53}, we define the $x$-directional anisotropy H\"{o}lder spaces and  prove some of their properties. The well-posedness of Cauchy problem for transport equations with nonhomogeneous terms in the anisotropy H\"{o}lder spaces is presented in Section \ref{sec54}.
The solvability of two-point boundary value problems of ordinary differential equation along each cross-section of the duct is studied in Section \ref{sec8}.
In Section \ref{sec45}, the stability of subsonic flows with mass-additions is proved by a new nonlinear iteration scheme and Schauder fixed-point theorem. By showing each component of the unique solution to the nonlinear problem shares the same order of regularity, we then complete the proof of the main theorem.

\section{Unidimensional steady subsonic flows with mass-additions}\label{sec41}
We construct the background solutions $U_b(x)=(p_b(x), s_b(x), E_b(x), v_b(x)\equiv0)$ aforementioned,  which are subsonic solutions to the Euler equations \eqref{4101}-\eqref{4103} that depend only on the $x$-direction. Hence they satisfy the following ordinary differential equations
\begin{equation*}
\frac{\dd }{\dd x}(\rho_b u_b)=\lambda \rho_b m_b(x),\quad
\frac{\dd }{\dd x}(\rho_b u_b^2+p_b)=0,\quad
\frac{\dd }{\dd x}\Big[\rho_b \big(\frac{1}{2}u_b^2+\frac{\gamma p_b}{(\gamma-1)\rho_b}\big)u_b\Big]=0.
\end{equation*}
For $C^1$ flows without vacuum,  from these we could solve that  
\begin{align}\label{4201}
&\frac{\dd u_b}{\dd x} = \frac{\lambda (\gamma+1)m_b(x)}{2}\frac{M^2_b}{1-M_b^2},\\
&\frac{\dd \rho_b}{\dd x} = \frac{\lambda}{2}\frac{\rho_bm_b(x)}{u_b}\frac{2- (\gamma +3)M_b^2}{1-M_b^2},\\
&\frac{\dd p_b}{\dd x} = - \frac{\lambda}{2}\rho_bm_b(x)u_b\frac{2+ (\gamma -1)M_b^2}{1-M_b^2},\\
&\frac{\dd M_b^2}{\dd x} = \frac{\lambda}{2}\frac{m_b(x)}{u_b}\frac{\gamma(\gamma-1)M_b^6+(3\gamma-1)M_b^4+2M_b^2}{1-M_b^2},\label{4201-4}\\
&\frac{\dd E_b}{\dd x} = -\lambda\frac{m_b(x)E_b}{u_b},\\
&\frac{\dd A(s_b)}{\dd x} =-\frac{\lambda \gamma m_b(x)\Big[1-\frac{\gamma-1}{2}M_b^2\Big]}{u_b}A(s_b),\label{4206add}
 \end{align}
where we require the Mach number $M_b(x) = {u_b(x)}/{c_b(x)}$ is positive and not equal to $1$.

Let the Mach number at the entry $\Sigma_0$ satisfy $M_b(0) > 0$. Then from \eqref{4201-4} and \eqref{4201},
we obtain
\begin{eqnarray*}
\frac{\gamma+1}{\gamma(\gamma-1)M_b^4+(3\gamma-1)M^2_b+ 2}\frac{\dd M^2_b}{\dd x}= \frac{1}{u_b}\frac{\dd u_b}{\dd x}.
\end{eqnarray*}
Direct integration shows that
\begin{eqnarray}\label{4701}
&&\frac{u_b}{u_b(0)}= \frac{\gamma M^2_b+1}{\gamma M^2_b(0)+1}\cdot\frac{(\gamma-1)M^2_b(0)+2}{(\gamma-1)M^2_b+2}.
\end{eqnarray}
Similarly, we have
\begin{equation}\label{4702}
\begin{cases}\dl
\frac{p_b}{p_b(0)}= \frac{\gamma M^{2}_b(0) + 1}{\gamma M_b^2 + 1},\\
\dl \frac{\rho_b}{\rho_b(0)}= \frac{M_b}{M_b(0)}\cdot\left(\frac{(\gamma-1) M^2_b + 2}{(\gamma-1) M^{2}_b(0) + 2}\right)^{2}\cdot\left(\frac{\gamma M^2_b(0)+1}{\gamma M^2_b+1}\right)^{3},\\
\dl \frac{E_b}{E_b(0)}= \frac{M_b}{M_b(0)}\cdot \frac{(\gamma-1) M^2_b + 2}{(\gamma-1) M^2_b(0) + 2}\cdot\left(\frac{\gamma M^{2}_b(0) + 1}{\gamma M_b^2 + 1}\right)^{2},\\
\dl \frac{A(s_b)}{A(s_b(0))}= \left(\frac{M_b(0)}{M_b}\right)^{\gamma}\cdot\left(\frac{\gamma M^{2}_b + 1}{\gamma M_b^2(0) + 1}\right)^{3\gamma-1}\cdot\left(\frac{(\gamma-1) M^{2}_b(0) + 2}{(\gamma-1) M_b^2 + 2}\right)^{2\gamma},
\end{cases}
\end{equation}
where $p_b(0),\,u_b(0),\,\rho_b(0),\,A(s_b(0)),\,E_b(0)$ are all positive constants. Substituting \eqref{4701} and \eqref{4702} into $\eqref{4201-4}$, then
\begin{eqnarray}\label{4703}
\frac{2(1 - M_b^{2})}{M_b^{2}[(\gamma-1) M_b^{2} + 2]^{2}} \frac{\dd M_b^{2}}{\dd x} = \frac{\lambda m_b(x)}{m_{0}},
\end{eqnarray}
where
\begin{eqnarray*}
m_{0} = \frac{(\gamma-1) M^{2}_b(0) + 2}{\gamma M^{2}_b(0) + 1 } > 0.
\end{eqnarray*}
Integrating $\eqref{4703}$ from $0$ to $x$, we get
\begin{eqnarray}\label{4705}
{F(M_b) = F(M_b(0)) +  \frac{\lambda}{m_{0}}\int ^{x}_{0}m_b(s)\,\dd s,}
\end{eqnarray}
\begin{figure}
\centering
	\includegraphics[width=0.8\linewidth]{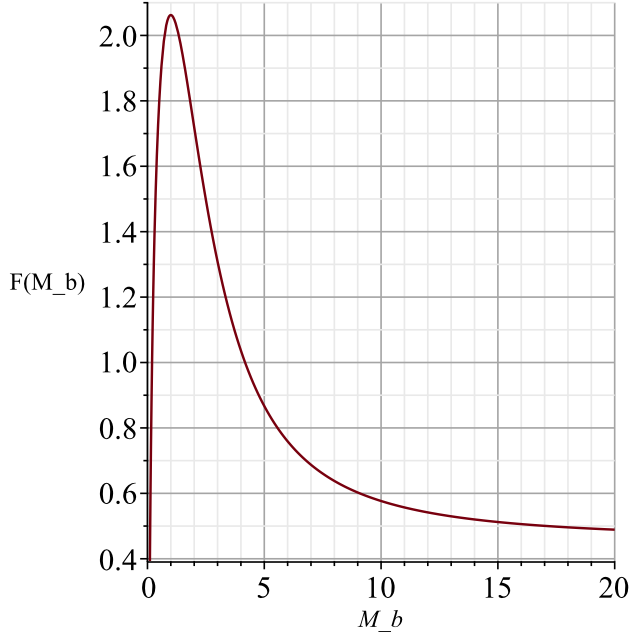}
\vspace{-8cm}
	\caption{Graph of the function $F(M_b)$.} \label{fig2}
\end{figure}
where
\begin{eqnarray*}
F(M_b) \triangleq \frac{\frac{\gamma+1}{\gamma-1}}{(\gamma-1) M^{2}_b+2}+\frac{1}{2}\ln\left[\frac{M^{2}_b}{(\gamma-1) M^{2}_b+2}\right],\end{eqnarray*}
whose graph is shown in Figure \ref{fig2}.
Noting that $F'(M_b) = \frac{4M_b(1 - M_b^{2})}{M_b^{2}[(\gamma-1) M_b^{2} + 2]^{2}}$, so for subsonic flow, $F'(M_b) > 0$; whereas for supersonic flow, $F'(M_b) < 0$. Therefore, no matter the flow is supersonic or subsonic, once the value $F(M_b)$ is known, $M_b$ can be determined immediately by the implicit function theorem. Consequently, from \eqref{4705}, once $x$ is given, $M_b(x)$ can be uniquely determined if we know it is larger or less than $1$.

One infers from \eqref{4201}-\eqref{4206add} the monotonicity of physical quantities $u_b,\,\rho_b,\,p_b,\,M_b,\,E_b$ and $A(s_b)$ along the positive $x$-direction, as shown in Table \ref{tab3} and Table \ref{tab4}. Notice that for the subsonic case $(0 < M_b < 1)$,  the opposite role played by the addition $(\lambda > 0)$ or absorbtion $(\lambda < 0)$ of mass in the flow.
\begin{tiny}
\begin{table}
\centering
\begin{tabular}{|c|c|c|c|c|c|c|}
\hline
\diagbox{\quad}{Monotonicity}{functions}&$u_b$&$\rho_b$&$\,p_b$&$M_b$&$E_b$\\ 
\hline
mass decreases~$(\lambda < 0)$&$\downarrow$&$M_b\in(0,\sqrt{\frac{2}{\gamma+3}})$, $\downarrow$; ~~$M_b\in(\sqrt{\frac{2}{\gamma+3}}, 1)$, $\uparrow$&$\uparrow$&$\downarrow$&$\uparrow$\\
\hline
mass increase~$(\lambda > 0)$&$\uparrow$&$M_b\in(0,\sqrt{\frac{2}{\gamma+3}})$, $\uparrow$;~~ $M_b\in(\sqrt{\frac{2}{\gamma+3}}, 1)$,  $\downarrow$&$\downarrow$&$\uparrow$&$\downarrow$\\
\hline
\end{tabular}
\medskip
\caption{Monotonicity of physical quantities along the positive $x$-direction in subsonic mass-changing flows.}\label{tab3}
\end{table}
\end{tiny}
\begin{tiny}
\begin{table}
\centering
\begin{tabular}{|c|c|c|c|c|c|c|}
\hline
\diagbox{\quad}{Monotonicity}{functions}&$A(s_b)$\\ 
\hline
mass decreases~$(\lambda < 0)$&$\gamma\geq 3$,$\uparrow$. If $\gamma < 3$, then $M_b\in(0,\sqrt{\frac{2}{\gamma-1}})$,$\uparrow$;~~$M_b\in(\sqrt{\frac{2}{\gamma-1}}, 1)$, $\downarrow$\\
\hline
mass increase~$(\lambda > 0)$&$\gamma\geq 3$,$\downarrow$. If $\gamma < 3$, then $M_b\in(0,\sqrt{\frac{2}{\gamma-1}})$,$\downarrow$;~~$M_b\in(\sqrt{\frac{2}{\gamma-1}}, 1)$,$\uparrow$\\
\hline
\end{tabular}
\medskip
\caption{Monotonicity of entropy along the positive $x$-direction in subsonic mass-changing flows.}\label{tab4}
\end{table}
\end{tiny}

For the case of mass-absorbtion $(\lambda < 0)$, we notice that there is a critical length of the duct, denoted as $l^-_{sub}$, such that at the exit,  the Mach number along the flow direction decreases to $0$. So is the density, see \eqref{4702}. It can be solved from $\eqref{4705}$:
\begin{eqnarray}\label{4706}
F(0) = F(M_b(0)) + \frac{\lambda}{m_{0}}\int ^{l^-_{sub}}_{0}m_b(s)\,\dd s.
\end{eqnarray}
For the case of mass-addition $(\lambda > 0)$, there is also a critical length of the duct, denoted by $l^+_{sub}$,  such that the flow is sonic at the exit, namely $M_b(l^+_{sub})=1$. We see that
\begin{eqnarray}\label{4707}
F(1) = F(M_b(0)) +  \frac{\lambda}{m_{0}}\int ^{l^+_{sub}}_{0}m_b(s)\,\dd s.
\end{eqnarray}
It is easy to check that if $m_b(x)$ is bounded and has a positive lower bound, then the critical length $l_*=l^{\mp}_{sub}$ are both finite and positive.

We now specify the background subsonic solutions. Let $m_b(x)$ have a positive lower bound. At the entry $\Sigma_0$ we prescribe $E_b(0),\,s_b(0),\,p_b(0)$ and $v_b\equiv 0$. Then by $\eqref{eq14addnew}$, we could solve $u_b(0)$ and hence $M_b(0)$. The critical length $l_*$  can be determined by $\eqref{4706}$ or $\eqref{4707}$ for a fixed $\lambda\in\mathbb{R}$. We choose a $l\leq l_*$. For any $x \in [0, l]$, using $\eqref{4705}$ and implicit function theorem, we get a unique $M_b(x)\in(0, 1)$. Furthermore, by substituting $M_b(x)$ into \eqref{4701} and \eqref{4702}, we obtain $u_b(x),\,\rho_b(x),\,p_b(x),\,E_b(x)$, and $A(s_b)(x)$. Then we have the following lemma.

\begin{lemma}[Subsonic background solution]\label{lem421}
For a given positive smooth function $m_b(x)$ which is bounded away from zero, and $\lambda\in\mathbb{R},\,l\in (0, l_*)$, (if $\lambda=0$, set $l_*=+\infty$), and given positive constants $E_b(0),\,s_b(0),\,p_b(0)$ at the entry so the flow is subsonic there, the equations $\eqref{4201}$ have a unique smooth subsonic solution $E_b(x),\,s_b(x),\,u_b(x)$ and $p_b(x)$. Furthermore, the solution is a real analytic function of the parameter $\lambda$.
\end{lemma}
We remark that the last observation, namely  real analyticity of background solution on the parameter $\lambda$, is of crucial importance for our analysis in the sequel, cf. Lemma \ref{Lem44}.

\section{A decomposition lemma}\label{sec42}


In this section, we present and prove the new decomposition lemma of Euler equations with mass-additions for the two-space-dimensional case. For a vector field $\mathbf{u}$ and a function $f$, we set ${\mathrm{D}}_{\mathbf{u}}f = \mathbf{u}\cdot \grad f$. The integral curves of the velocity field $\mathbf{u}$ are called streamlines.

\begin{lemma}\label{Lem41}
For polytropic gases in the duct $\Omega$, suppose that $\,p,\,\rho,\,\mathbf{u},\,m\in C^2(\Omega)\cap C^1(\bar{\Omega})$ and $\rho>0,\,u > 0$ in $\bar{\Omega}$. Let $k(x)\in C^1([0, l])$ be an axillary function.  If they satisfy the following problems (where $\lambda\in \mathbb{R}$, and $\gamma>1$ is the adiabatic exponent,  $M$ is the  Mach number;  moreover, $L^m(\cdot),\,L_n(\cdot)$ are  smooth functions satisfying $L^m(0) = 0,\,L_n(0) = 0$ for $m = 1,\,2,\,3,\,n=1,\,2$):
\begin{eqnarray}
&&\begin{cases}\label{4301}
{\mathrm{D}}_{\mathbf{u}} E+\lambda m(\mathbf{x})E=0&\quad \text{\rm in}\quad \Omega,\\
E=E_0(y)&\quad \text{\rm on}\quad \Sigma_0;\\
\end{cases}\\
&&\begin{cases}\label{4302}
{\mathrm{D}}_{\mathbf{u}} A(s)+\lambda \gamma m(\mathbf{x})\Big[1-\frac{\gamma-1}{2}M^2\Big]A(s)=0&\quad \text{\rm in}\quad \Omega,\\
A(s)=A(s_0)(y)&\quad \text{\rm on}\quad \Sigma_0;
\end{cases}\\
&&\begin{cases}\label{4304}
\p_y v =  - \frac{{\mathrm{D}}_{\mathbf{u}}p}{\gamma p} + \lambda m(\mathbf{x})\Big[1+\frac{(\gamma-1)M^2}{2}\Big]
+ \frac{1}{u}\Big[v\p_y u + \frac{1}{\rho}\p_xp\Big] - k(x)&\quad \text{\rm in}\quad \Omega,\\
v=v_0(y)&\quad \text{\rm on}\quad \Sigma_0,\\
v=0&\quad \text{\rm on}\quad W_-\cup W_+;
\end{cases}\\
&&\begin{cases}\label{4303}
{\mathrm{D}}_{\mathbf{u}}\Big[\frac{{\mathrm{D}}_{\mathbf{u}} p}{\gamma p}\Big]-\di \Big[\frac{\grad\,
	p}{\rho}\Big]- \Big[(\p_x u)^2+2\p_y u\p_x v+(\p_y v)^2\Big]- \frac{\lambda(\gamma-1)}{2}{\mathrm{D}}_\mathbf{u}\Big[m(\mathbf{x})M^2\Big]\\
\quad -\lambda \di \Big[m(\mathbf{x})\mathbf{u}\Big]+ L^1({\mathrm{D}}_{\mathbf{u}} E+\lambda m(\mathbf{x})E)\\
\qquad+L^2\left({\mathrm{D}}_{\mathbf{u}} A(s)+\lambda \gamma m(\mathbf{x})\Big[1-\frac{\gamma-1}{2}M^2\Big]A(s)\right)\\
\quad\quad+L^3\left(\p_y v+\frac{{\mathrm{D}}_{\mathbf{u}}p}{\gamma p} - \lambda m(\mathbf{x})\Big[1+\frac{(\gamma-1)M^2}{2}\Big] \right.
\left.- \frac{1}{u}\Big[v\p_y u + \frac{1}{\rho}\p_xp\Big] + k(x)\right) = 0\quad \text{\rm in}\ \Omega, \\
\frac{{\mathrm{D}}_{\mathbf{u}}p}{\gamma p} + \p_y v - \frac{1}{u}\Big[v\p_y u + \frac{1}{\rho}\p_x p\Big] - \lambda m(\mathbf{x})\Big[1+\frac{(\gamma-1)M^2}{2}\Big]+ L_1({\mathrm{D}}_{\mathbf{u}} E+\lambda m(\mathbf{x})E)\\ \quad+L_2\left({\mathrm{D}}_{\mathbf{u}} A(s)+\lambda \gamma m(\mathbf{x})\Big[1-\frac{\gamma-1}{2}M^2\Big]A(s)\right) = 0\quad \text{\rm on}\quad \Sigma_0,\\
p=p_l(y)\quad\quad\quad\quad \text{\rm on}\quad \Sigma_l,\\
\p_yp= 0\quad\quad\quad\quad\ \  \text{\rm on}\quad W_-\cup W_+,
\end{cases}
\end{eqnarray}
 then $k(x)\equiv0$, and $(p,\, \rho,\, \mathbf{u})$ is a solution to the boundary value problem \eqref{4101}-\eqref{4103}, \eqref{4202}-\eqref{4204} in {Problem \uppercase\expandafter{\romannumeral1}}. On the contrary, for given $\lambda\in\mathbb{R}$ and $m(\mathbf{x})\in C^2(\Omega)\cap C^1(\bar{\Omega})$, any solution $\,p,\,\rho,\,\mathbf{u}\in C^2(\Omega)\cap C^1(\bar{\Omega})$ with $\rho>0,~ u>0$ to the Problem \Rmnum{1} is a solution to the coupled problems \eqref{4301}-\eqref{4303} (in which $k(x)\equiv0$).
\end{lemma}

By $\eqref{4301}$ and $\eqref{4302}$, we infer that  for $\lambda= 0$ (no mass-addition),   the well-known fact that total enthalpy $E$ and entropy $s$ remain unchanged along the streamlines of $C^1$ flows.

We now prove Lemma \ref{Lem41}. Suppose that $\rho,\, p,\, \mathbf{u}$ with indicated regularity solve the Euler equations and fulfill the boundary conditions \eqref{4202}-\eqref{4204}.
As it is assumed that the flow contains no vacuum ($\rho>0$),  the conservation of mass $\varphi_1$ can be written as
\begin{eqnarray}\label{4305}
\varphi_1 = {\mathrm{D}}_{\mathbf{u}} \rho+\rho\, \di\, \mathbf{u}-\lambda \rho m(\mathbf{x}).
\end{eqnarray}
By the energy  equation, we obtain the following transport equation of total enthalpy
\begin{eqnarray}\label{4306}
\varphi_3\triangleq\frac{1}{\rho}({\varphi_2 - E\varphi_1}) = {\mathrm{D}}_{\mathbf{u}} E+\lambda m(\mathbf{x})E.
\end{eqnarray}
This implies \eqref{4301}.

By the identity $\di (\mathbf{u}\otimes \mathbf{v}) = (\di\, \mathbf{v})\mathbf{u} + {\mathrm{D}}_{\mathbf{v}} \mathbf{u}$, where $\mathbf{u}$ and $\mathbf{v}$ are vector fields, the momentum  equation becomes
\begin{eqnarray}\label{4307}
\varphi_0=(\varphi_0^x,\varphi_0^y)^{\top}\triangleq\frac{1}{\rho}({\varphi-\varphi_1 \mathbf{u}}) = {\mathrm{D}}_{\mathbf{u}}
\mathbf{u}+\frac{\grad\, p}{\rho}+\lambda m(\mathbf{x})\mathbf{u}.
\end{eqnarray}
Taking scalar product with  $\mathbf{u}$ yields $$\frac{1}{2} {\mathrm{D}}_{\mathbf{u}}\big(|\mathbf{u}|^2\big) = \varphi_0\cdot \mathbf{u} - \frac 1\rho {\mathrm{D}}_{\mathbf{u}}p-\lambda m(\mathbf{x})|\mathbf{u}|^2.$$  It follows  that
\begin{eqnarray}\label{4308}
\varphi_4\triangleq\frac{\gamma-1}{\rho^{\gamma-1}}\big(\varphi_3-\varphi_0\cdot
\mathbf{u}\big) = {\mathrm{D}}_{\mathbf{u}}A(s)+\lambda \gamma m(\mathbf{x})\Big[1-\frac{\gamma-1}{2}M^2\Big]A(s),
\end{eqnarray}
from which \eqref{4302} follows.

Next,  we derive the second-order  equation for pressure  in \eqref{4303}.
By $\eqref{4305}$ and $\eqref{4308}$, we  have
\begin{eqnarray}\label{4309}
\bar{\varphi}_1\triangleq\frac{\varphi_1}{\rho} + \frac{\varphi_4}{\gamma A(s)} = \frac{{\mathrm{D}}_{\mathbf{u}}p}{\gamma p} + \di\, \mathbf{u}-\frac{\lambda(\gamma-1)m(\mathbf{x})M^2}{2}.
\end{eqnarray}
Then $\eqref{4307}$ and $\eqref{4309}$ yield
\begin{eqnarray}\label{4310}
&&{\mathrm{D}}_{\mathbf{u}}\bar{\varphi}_1-\di \varphi_0={\mathrm{D}}_{\mathbf{u}}\Big[\frac{{\mathrm{D}}_{\mathbf{u}} p}{\gamma p}\Big]-\di \Big[\frac{\grad\,
	p}{\rho}\Big]- [\di ({\mathrm{D}}_{\mathbf{u}}\mathbf{u}) - {\mathrm{D}}_{\mathbf{u}}(\di\, \mathbf{u})]\nonumber\\
&&\quad \quad -\frac{\lambda(\gamma-1)}{2}{\mathrm{D}}_{\mathbf{u}}\Big[m(\mathbf{x})M^2\Big]- \lambda \di \Big[m(\mathbf{x})\mathbf{u}\Big]{\mathrm{D}}_{\mathbf{u}}\Big[\frac{{\mathrm{D}}_{\mathbf{u}} p}{\gamma p}\Big] - \di \Big[\frac{\grad\,
	p}{\rho}\Big]\nonumber\\
&&\quad = -\Big[(\p_x u)^2+2\p_x v\p_y u+(\p_y v)^2\Big]-\frac{\lambda(\gamma-1)}{2}{\mathrm{D}}_{\mathbf{u}}\Big[m(\mathbf{x})M^2\Big]\nonumber\\
&&\quad \quad  -\lambda \di \Big[m(\mathbf{x})\mathbf{u}\Big].
\end{eqnarray}
By the momentum equation in the $x$-direction, we have
\begin{eqnarray*}
\p_x u = - \frac{1}{u}\Big[v\p_y u + \frac{1}{\rho}\p_x p\Big]-\lambda m(\mathbf{x}) + \frac{1}{u}\varphi^x_0.
\end{eqnarray*}
Substituting it into the mass equation $\bar{\varphi}_1$, we have \eqref{4304} and the second equation in \eqref{4303}:
\begin{eqnarray}\label{4311}
\bar{\varphi}_1 - \frac{1}{u}\varphi^x_0 = \p_y v  - \frac{1}{u}\Big[v\p_y u + \frac{1}{\rho}\p_x p\Big]-\lambda m(\mathbf{x})\Big[1+\frac{(\gamma-1)M^2}{2}\Big] + \frac{{\mathrm{D}}_{\mathbf{u}}p}{\gamma p}.
\end{eqnarray}
Notice that $(\ref{4303})_4$ follows from \eqref{4205}.

We now prove the more difficult converse conclusion. The boundary conditions $\eqref{4301}_2$, $\eqref{4302}_2$, $\eqref{4303}_3$ and $\eqref{4304}_2$ are  \eqref{4202}, \eqref{4203}, and  $\eqref{4204}$ stated in {Problem \uppercase\expandafter{\romannumeral1}}. From $\eqref{4301}_1$ and $\eqref{4306}$, we infer that
\begin{eqnarray}\label{4312}
\varphi_3 = \frac{1}{\rho}({\varphi_2 - E\varphi_1}) = {\mathrm{D}}_\mathbf{u} E+\lambda m(\mathbf{x})E = 0.
\end{eqnarray}
Similarly, from $\eqref{4302}_1$ and $\eqref{4308}$, one has
\begin{eqnarray}\label{4313}
\varphi_4 = \frac{\gamma - 1}{\rho^{\gamma - 1}}\big(\varphi_3 - \varphi_0\cdot \mathbf{u}\big) = {\mathrm{D}}_\mathbf{u} A(s)+\lambda \gamma m(\mathbf{x})\Big[1-\frac{\gamma-1}{2}M^2\Big]A(s)=0.
\end{eqnarray}
Hence  $\varphi_0\cdot \mathbf{u} = 0$, or recalling the assumption $u>0$,
\begin{eqnarray}\label{4314}
\varphi^x_0 = - \frac{v}{u}\varphi_0^y.
\end{eqnarray}
Furthermore, by $\eqref{4310}$ and $\eqref{4301}_1$, $\eqref{4302}_1$, $\eqref{4303}_1$ and $\eqref{4304}_1$,
\begin{eqnarray}\label{4315}
{\mathrm{D}}_{\mathbf{u}}\bar{\varphi}_1-\di \varphi_0 = {\mathrm{D}}_{\mathbf{u}}\bar{\varphi}_1-\di \varphi_0 + L^1(\varphi_4) + L^2(\varphi_3) = -L^3(\bar{\varphi}_1 - \frac{1}{u}\varphi^x_0 + k(x)).
\end{eqnarray}
For any $x\in [0, l]$, we integrate $\eqref{4304}_1$ on $\Sigma_{x}$,
\begin{eqnarray*}
\int_{\Sigma_{x}}\p_y v\,\dd y = \int_{\Sigma_{x}}\left\{ - \frac{{\mathrm{D}}_{\mathbf{u}}p}{\gamma p} + \lambda m(\mathbf{x})\Big[1+\frac{(\gamma-1)M^2}{2}\Big] + \frac{1}{u}\Big[v\p_y u + \frac{1}{\rho}\p_x p\Big] - k(x)\right\}\dd y.
\end{eqnarray*}
By the slip condition $\eqref{4204}$, the left-hand-side equals zero, so
\begin{eqnarray}\label{4316}
k(x) = \frac{1}{\pi}\int_{0}^{\pi}  \left\{- \frac{{\mathrm{D}}_{\mathbf{u}}p}{\gamma p} + \lambda m(\mathbf{x})\Big[1+\frac{(\gamma-1)M^2}{2}\Big] + \frac{1}{u}\Big[v\p_y u + \frac{1}{\rho}\p_x p\Big]\right\}\,\dd y.
\end{eqnarray}
We assert that $k(x) \equiv 0$.

In fact, by $\eqref{4304}_1$ and $\eqref{4311}$,
\begin{equation}\label{4317}
\bar{\varphi}_1 - \frac{1}{u}\varphi^x_0 + k(x) = 0 \quad \text{in}\quad \Omega.
\end{equation}
From the boundary condition $\eqref{4303}_2$ on the inlet, namely $(\bar{\varphi}_1 - \frac{1}{u}\varphi^x_0)|_{\Sigma_0} = 0$, we have $k(0) = 0$. Furthermore, substituting $\eqref{4314}$ and $\eqref{4317}$ into $\eqref{4315}$ yields \begin{eqnarray*}
&&0 = {\mathrm{D}}_{\mathbf{u}}\bar{\varphi}_1 - \di \varphi_0 = {\mathrm{D}}_{\mathbf{u}}\Big[\frac{1}{u}\varphi^x_0 - k(x)\Big] - \p_y\varphi^y_0 + \p_x\Big[\frac{v}{u}\varphi_0^y\Big]\nonumber\\
&&\quad =u\p_x\Big[\frac{1}{u}\varphi^x_0 - k(x)\Big] + v\p_y\Big[\frac{1}{u}\varphi^x_0 - k(x)\Big]- \p_y\varphi^y_0 + \p_x\Big[\frac{v}{u}\varphi_0^y\Big].
\end{eqnarray*}
Some simplification shows that
\begin{eqnarray}\label{4318}
\p_y\varphi^y_0 +\frac{uv}{u^2+v^2} \left[{\mathrm{D}}_{\mathbf{u}}\Big[\frac{1}{u}\Big]+\p_y\Big[\frac{v}{u}\Big]\right]\varphi_0^y =  - \frac{u^3}{u^2+v^2}k'(x).
\end{eqnarray}
By $\eqref{4303}_4$ and $\eqref{4304}_3$, the boundary conditions of equations $\eqref{4318}$ are
\begin{eqnarray}\label{4319}
\varphi_0^y = 0\quad \text{on}\quad \Sigma_{x}\cap (W_-\cup W_+).
\end{eqnarray}
For given $x\in[0,l]$,  the two-point boundary value problem \eqref{4318}-\eqref{4319} on the cross-section $\Sigma_{x}$ is solved by
\begin{eqnarray}\label{4320}
\varphi_0^y(x,y) = -k'(x)\int_{0}^{y}\beta(x,s)\mathrm{e}^{\int_{s}^{y}\alpha(x,t)\,\dd t}\,\dd s,
\end{eqnarray}
where $\alpha(x,y)=\frac{uv}{u^2+v^2} \left[{\mathrm{D}}_{\mathbf{u}}\Big[\frac{1}{u}\Big]+\p_y\Big[\frac{v}{u}\Big]\right]$, $\beta(x,y)=\frac{u^3}{u^2+v^2}>0$. Taking $y=\pi$, we get
\begin{eqnarray*}
\varphi_0^y(x,\pi) = -k'(x)\int_{0}^{\pi}\beta(x,s)\mathrm{e}^{\int_{s}^{\pi}\alpha(x,t)\,\dd t}\,\dd s=0.
\end{eqnarray*}
So $k'(x) \equiv 0$, thus $k(x) \equiv k(0) = 0$.

Therefore,  \eqref{4320} shows  $\varphi_0^y(x,y) \equiv 0$, and by \eqref{4314}, $\varphi_0^x(x,y) \equiv 0$. This demonstrates $\varphi = 0$ from $\eqref{4307}$, namely the conservation of momentum in the Euler equations.  So \eqref{4317} implies that  $\bar{\varphi}_1 = 0$.  Combining this with \eqref{4309} and \eqref{4313}, we see  $\varphi_1 = 0$, i.e., the equation of mass holds.  Finally, we infer the conservation of energy $\varphi_2 = 0$ from $\eqref{4312}$. The proof of Lemma \ref{Lem41} is completed.

\begin{remark}
A similar decomposition holds for the three-space-dimensional stationary compressible Euler equations, for which there is an additional transport equation for the normal component of vorticity, and the ordinary differential equations of tangential velocity component on each cross-section is replaced by a nonhomogeneous Cauchy-Riemann system of tangential velocities.\qed
\end{remark}

\section{A nonlinear boundary-value problem for an elliptic-hyperbolic coupled system}\label{sec43}

In this section we apply Lemma \ref{Lem41} to formulate a Problem \Rmnum{2}, which is equivalent to Problem \Rmnum{1}. The main efforts are to calculate explicitly a nonlinear elliptic-hyperbolic coupled system,  by linearization around a background solution, so that the  leading parts are  linear, and the other parts involving lower-order derivatives are small in suitable metrics, although they might be nonlinear and coupled to each other. The idea is simple and the computations are straightforward. We present the detailed expressions for readers convenience to check the correctness of nonlinear iteration scheme constructed in Section \ref{sec45}.

\subsection{Boundary conditions and equation for pressure}\label{sec31}
We calculate the second equation in \eqref{4303} to obtain a Robin condition for the pressure on the inlet $\Sigma_0$. To  this end, multiplying both sides of $\eqref{4303}_2$ by $\rho u$,  we get \begin{eqnarray}\label{4401}
&&\rho u\left[\frac{{\mathrm{D}}_\mathbf{u}p}{\gamma p} + \p_y v - \frac{1}{u}\Big(v\p_y u + \frac{1}{\rho}\p_x p\Big) - \lambda m(\mathbf{x})\Big(1+\frac{(\gamma-1)M^2}{2}\Big)\right]\nonumber\\
&&\quad  = \Big[\frac{u^2}{c^2} - 1\Big]\p_xp  - \frac{\lambda \gamma m(\mathbf{x})u \Big[1+\frac{(\gamma-1)M^2}{2}\Big] }{c^2}p\nonumber\\
&&\quad \quad +\rho u\p_y v+uv\Big[\frac{1}{u^2}+\frac{1}{c^2}\Big]\p_y p - \frac{\rho v}{u}\p_y E +  \frac{v\rho^{\gamma}}{(\gamma-1)u}\p_y A(s) + \frac{\rho v^2}{u}\p_y v = 0.\nonumber\\
\end{eqnarray}
For subsonic flow, ${u^2}/{c^2} < 1$, we divide both sides of $\eqref{4401}$ by $\Big[\frac{u^2}{c^2} - 1\Big]$. Then
\begin{eqnarray}\label{4402}
 \p_x p - \frac{\lambda \gamma m(\mathbf{x})u \Big[1+\frac{(\gamma-1)M^2}{2}\Big] }{u^2 - c^2}p=G_1(U)+G_2(U),
\end{eqnarray}
where
\begin{equation}\label{4403}
\begin{cases}
G_1 \triangleq - \frac{uv}{\frac{u^2}{c^2}-1}\Big(\frac{1}{u^2}+\frac{1}{c^2}\Big)\p_y p,&\\
G_2\triangleq - \frac{1}{\frac{u^2}{c^2} - 1}\left[\rho u\p_y v - \frac{\rho v}{u}\p_y E +  \frac{v\rho^{\gamma}}{(\gamma-1)u}\p_y A(s) + \frac{\rho v^2}{u}\p_y v\right].&
\end{cases}
\end{equation}

Next we specify the equation of pressure $p$ in $\eqref{4303}_1$. 
It is calculated term by term as follows.

The first term:
\begin{eqnarray*}
&&{\mathrm{D}}_{\mathbf{u}}\Big[\frac{{\mathrm{D}}_{\mathbf{u}}p}{\gamma p}\Big] = {\mathrm{D}}_{\mathbf{u}}\Big[\frac{u\p_xp+v\p_yp}{\gamma p}\Big] = u\p_x\Big[\frac{u\p_xp+v\p_yp}{\gamma p}\Big]+v\p_y\Big[\frac{u\p_xp+v\p_yp}{\gamma p}\Big]\nonumber\\
&&\quad = \frac{u\p_xu\p_xp+u^2\p_x^2p}{\gamma p} - \frac{u^2}{\gamma p^2}(\p_xp)^2+\frac{u\p_xv\p_yp+uv\p_{xy}^2p}{\gamma p} - \frac{uv}{\gamma p^2}\p_xp\p_yp\nonumber\\
&&\quad \quad +\frac{v\p_yu\p_xp+uv\p_{xy}^2p}{\gamma p} - \frac{uv}{\gamma p^2}\p_xp\p_yp+\frac{v\p_yv\p_yp+v^2\p_y^2p}{\gamma p} - \frac{v^2}{\gamma p^2}(\p_yp)^2\nonumber\\
&&\quad =\frac{u\p_xu}{\gamma p}\p_xp+\frac{u^2}{\gamma p}\p_x^2p+\frac{v^2}{\gamma p}\p_y^2p+2\frac{uv}{\gamma p}\p_{xy}^2p - \frac{u^2}{\gamma p^2}(\p_xp)^2+C_1(U),
\end{eqnarray*}
where $\p_{xy}^2p=\frac{\p^2 p}{\p x\p y},$ and
\begin{eqnarray*}
C_1(U) = \frac{u\p_xv\p_yp+v\p_yu\p_xp+v\p_yv\p_yp}{\gamma p} - 2\frac{uv}{\gamma p^2}\p_xp\p_yp-\frac{v^2}{\gamma p^2}(\p_yp)^2
\end{eqnarray*}
is an expression of $U$ and its first-order derivatives ${\mathrm{D}}U$.

The second term:
\begin{eqnarray*}
\di\,\Big[\frac{\grad\, p}{\rho}\Big] =\frac{\p^2_xp}{\rho} + \frac{\p^2_yp}{\rho}- \frac{1}{\rho^2}\p_x\rho\p_xp + C_2(U),
\end{eqnarray*}
where
$C_2(U) = - \frac{1}{\rho^2}\p_y\rho\p_yp$
contains only $U$ and ${\mathrm{D}}U$.

The third term:
\begin{eqnarray*}
(\p_xu)^2+2\p_yu\p_xv + (\p_yv)^2 = (\p_xu)^2+C_3(U),
\end{eqnarray*}
where
$C_3(U) =2\p_yu\p_xv + (\p_yv)^2$
also contains only $U$ and ${\mathrm{D}}U$.

The fourth term:
\begin{align*}
{\mathrm{D}}_\mathbf{u}\Big[m(\mathbf{x})M^2\Big]&=u\p_x\Big[m(\mathbf{x})M^2\Big]
+v\p_y\Big[m(\mathbf{x})M^2\Big]\\
&=uM^2\p_xm(\mathbf{x})+2uMm(\mathbf{x})\p_xM + C_4(U),
\end{align*}
where
$C_4(U) = v\p_y\Big[m(\mathbf{x})M^2\Big]$
contains only $U,\,m(\mathbf{x})$, and their first-order derivatives ${\mathrm{D}}U,$ $ {\mathrm{D}}m(\mathbf{x})$.

The fifth term:
\begin{eqnarray*}
&&\di \Big[m(\mathbf{x})\mathbf{u}\Big] = \p_x\Big[m(\mathbf{x})u\Big]+\p_y\Big[m(\mathbf{x})v\Big]\nonumber\\
&&\quad =u\p_x m(\mathbf{x})+m(\mathbf{x})\p_x u+v\p_y m(\mathbf{x})+m(\mathbf{x})\p_y v\nonumber\\
&&\quad =u\p_x m(\mathbf{x})+m(\mathbf{x})\p_xu+m(\mathbf{x})\p_y v+C_5(U)
\end{eqnarray*}
where
$C_5(U) = v\p_y m(\mathbf{x}).$

From these, we get
\begin{eqnarray}\label{4404}
&&{\mathrm{D}}_\mathbf{u}\Big[\frac{{\mathrm{D}}_\mathbf{u} p}{\gamma p}\Big]-\di \Big[\frac{\grad\,
	p}{\rho}\Big]- \Big[(\p_x u)^2 +2\p_yu\p_xv + (\p_y v)^2\Big] \nonumber\\
&&\qquad\qquad\qquad\qquad\qquad- \frac{\lambda(\gamma-1)}{2}{\mathrm{D}}_\mathbf{u}\Big[m(\mathbf{x})M^2\Big] - \lambda \di \Big[m(\mathbf{x})\mathbf{u}\Big]\nonumber\\
&&\quad =\frac{u^2-c^2}{\gamma p}\p_x^2p+\frac{v^2-c^2}{\gamma p}\p_y^2p+2\frac{uv}{\gamma p}\p_{xy}^2p-\frac{u^2}{\gamma p^2}(\p_x p)^2+\frac{u\p_x u+\frac{c^2}{\rho}\p_x\rho}{\gamma p}\p_x p -(\p_x u)^2\nonumber\\
&&\quad \quad \qquad\qquad-\frac{\lambda(\gamma-1)M^2+2\lambda}{2}u\p_x m(\mathbf{x})-\frac{\lambda(\gamma-1)u m(\mathbf{x})}{2}\p_xM^2 \nonumber\\
&&\qquad\qquad\qquad\qquad-\lambda m(\mathbf{x})\p_x u-\lambda m(\mathbf{x})\p_y v + F_1(U),
\end{eqnarray}
where
\begin{eqnarray}\label{4405}
F_1(U) = C_1(U)-C_2(U)-C_3(U)-\frac{\lambda(\gamma-1)}{2}C_4(U)-\lambda C_5(U).
\end{eqnarray}

The next step is using equations $\eqref{4301}_1$,\,$\eqref{4302}_1$, and $\eqref{4304}_1$, to replace normal non-small quantities in \eqref{4404}, such as  $u\p_xu$, \,$(\p_xu)^2$, \,$\p_x\rho$, and $\p_yv$, by tangential small quantities. In fact, we have
\begin{equation}\label{4406}
\begin{cases}
\p_xu= -\frac{1}{\rho u}\p_xp-\lambda m(\mathbf{x}) - \frac{v}{u^2}\p_y E +\frac{1}{\gamma - 1}\frac{v\rho^{\gamma - 1}}{u^2}\p_y A(s)-\frac{v}{u}\p_xv -\frac{\lambda m(\mathbf{x})v^2}{u^2},&\\
\p_x\rho=\frac{1}{c^2}\p_xp+\frac{\lambda m(\mathbf{x})\rho}{u}\Big[1-\frac{\gamma - 1}{2}M^2\Big]+\frac{v \rho^\gamma}{u c^2}\p_y A(s),&\\
\p_y v=-\frac{{\mathrm{D}}_{\mathbf{u}}p}{\gamma p} + \lambda m(\mathbf{x})\Big[1+\frac{(\gamma-1)M^2}{2}\Big] + \frac{1}{u}\Big[v\p_y u + \frac{1}{\rho}\p_xp\Big]-k(x).&\\
\end{cases}
\end{equation}
Notice that $\eqref{4406}_3$ follows directly from $\eqref{4304}_1$.
To show $\eqref{4406}_1$, from the definition of total enthalpy $$E=\frac{1}{2}u^2+\frac{1}{2}v^2+\frac{\gamma}{\gamma-1}p^{1-\frac{1}{\gamma}}A(s)^{\frac{1}{\gamma}},$$ there holds $u^2=2E-v^2-\frac{2\gamma}{\gamma-1}p^{1-\frac{1}{\gamma}}A(s)^{\frac{1}{\gamma}}$. Taking derivatives with respect to $x$ yields $2u\p_xu=2\p_xE-2v\p_xv-\frac{2}{\rho}\p_xp-\frac{2}{\gamma-1}\rho^{\gamma-1}\p_xA(s)$. Then
one obtains $\eqref{4406}_1$ after replacing $\p_xE$ and $\p_xA(s)$ using equations $\eqref{4301}_1$, $\eqref{4302}_1$.
For $\eqref{4406}_2$, noticing that $\rho=p^{\frac{1}{\gamma}}A(s)^{-\frac{1}{\gamma}}$, thanks to the state equation $p=A(s)\rho^{\gamma}$, differentiate the two sides  with respect to $x$, one has $\p_x\rho=\frac{1}{c^2}\p_xp-\frac{\rho^{\gamma}}{c^2}\p_xA(s)$. Then $\eqref{4406}_2$ follows by replacing $\p_xA(s)$ using $\eqref{4302}_1$.
The alerted reader may recognize that we introduce the functions $L^m\,(m=1,\,2,\,3\,)$ in $\eqref{4304}_1$ to make the above procedure of replacing $\p_xE$ and $\p_xA(s)$ etc. precisely.

Thanks to $\eqref{4406}$, direct computation shows
\begin{eqnarray}\label{4407}
&&\frac{u\p_x u+\frac{c^2}{\rho}\p_x\rho}{\gamma p}\p_x p-(\p_x u)^2-\lambda m(\mathbf{x})\p_x u-\lambda m(\mathbf{x})\p_y v-\frac{\lambda(\gamma-1)u m(\mathbf{x})}{2}\p_xM^2\nonumber\\
&&\quad =-\frac{1}{\rho^2}\frac{1}{2E - \frac{2}{\gamma - 1}c^2}(\p_x p)^2+\Big[\frac{(\gamma-1)^2}{2}\frac{u^4}{c^4}+\frac{\gamma-1}{2}\frac{u^2}{c^2}-1\Big]\frac{\lambda m(\mathbf{x})c^2}{\gamma p u}\p_x p\nonumber\\
&&\quad \quad +\lambda^2m^2(\mathbf{x})\Big[\frac{(\gamma-1)^2}{4}\frac{u^4}{c^4}-1\Big]+\lambda m(\mathbf{x})k(x)
+F_2(U),
\end{eqnarray}
where
\begin{eqnarray}\label{4408}
&&F_2(U)=\frac{1}{\rho^2}\frac{v^2}{[2E - \frac{2}{\gamma - 1}c^2][2E - \frac{2}{\gamma - 1}c^2 - v^2]}(\p_x p)^2-\frac{(\gamma-1)\lambda m(\mathbf{x})v^2}{2\gamma pu}\p_x p\nonumber\\
&&\quad \quad -\frac{(\gamma-1)\lambda^2 m^2(\mathbf{x})v^2}{2c^2}+\frac{v \rho^{\gamma-1}}{\gamma pu}\p_xp\p_y A(s)+\frac{\lambda m(\mathbf{x})v\p_y p}{\gamma p}\nonumber\\
&&\quad \quad +\frac{\Big[- \frac{v}{u}\p_y E +\frac{1}{\gamma - 1}\frac{v\rho^{\gamma - 1}}{u}\p_y A(s)-v\p_xv -\frac{\lambda m(\mathbf{x})v^2}{u}\Big]^2}{u^2}\nonumber\\
&&\quad \quad -\frac{2\Big[-\frac{1}{\rho}\p_x p-\lambda m(\mathbf{x})u\Big]\Big[-\frac{v}{u}\p_y E +\frac{1}{\gamma - 1}\frac{v\rho^{\gamma - 1}}{u}\p_y A(s)-v\p_xv -\frac{\lambda m(\mathbf{x})v^2}{u}\Big]}{u^2}\nonumber\\
&&\quad \quad +\frac{-\frac{v}{u}\p_y E +\frac{1}{\gamma - 1}\frac{v\rho^{\gamma - 1}}{u}\p_y A(s)-v\p_xv -\frac{\lambda m(\mathbf{x})v^2}{u}}{\gamma p}\p_x p\nonumber\\
&&\quad \quad -\lambda m(\mathbf{x})\frac{- \frac{v}{u}\p_y E +\frac{v\rho^{\gamma - 1}}{(\gamma - 1)u}\p_y A(s)-v\p_x v -\frac{\lambda m(\mathbf{x})v^2}{u}-v\p_y u}{u}\nonumber\\
&&\quad \quad -\frac{\lambda(\gamma-1)u m(\mathbf{x})}{2}\frac{2c^2\Big[- \frac{v}{u}\p_y E +\frac{v\rho^{\gamma - 1}}{(\gamma - 1)u}\p_y A(s)-v\p_xv -\frac{\lambda m(\mathbf{x})v^2}{u}\Big]}{c^4}\nonumber\\
&&\quad \quad +\frac{\lambda(\gamma-1)u m(\mathbf{x})}{2}\frac{u^2\Big[\frac{\gamma-1}{2}\frac{\lambda m(\mathbf{x})v^2}{u}-\frac{v\rho^{\gamma-1}}{u}\p_y A(s)\Big]}{c^4}.
\end{eqnarray}
Substituting $\eqref{4407}$ into $\eqref{4404}$, one has
\begin{eqnarray}\label{4409}
&&{\mathrm{D}}_\mathbf{u}\Big[\frac{{\mathrm{D}}_\mathbf{u} p}{\gamma p}\Big]-\di \Big[\frac{\grad\,
	p}{\rho}\Big]-\Big[(\p_x u)^2+2\p_y u\p_x v+(\p_y v)^2\Big]\nonumber\\
&&\quad \quad -\frac{\lambda(\gamma-1)}{2}{\mathrm{D}}_\mathbf{u}\Big[m(\mathbf{x})M^2\Big]-\lambda \di \Big[m(\mathbf{x})\mathbf{u}\Big]\nonumber\\
&&\quad = \frac{1 }{\gamma p}\left[2E - \frac{\gamma + 1}{\gamma - 1}c^2\right] \p^2_x p - \frac{2}{\gamma p^2}\left[E-\frac{c^2}{\gamma-1} + \frac{c^4}{4\gamma}\frac{1}{E- \frac{c^2}{\gamma - 1}}\right](\p_x p)^2\nonumber\\
&&\quad \quad +\Big[\frac{(\gamma-1)^2}{2}\frac{u^4}{c^4}+\frac{\gamma-1}{2}\frac{u^2}{c^2}-1\Big]\frac{\lambda m(\mathbf{x})c^2}{\gamma p u}\p_x p-\frac{\lambda(\gamma-1)M^2+2\lambda}{2}u\p_x m(\mathbf{x})\nonumber\\
&&\quad \quad +\lambda^2m^2(\mathbf{x})\Big[\frac{(\gamma-1)^2}{4}\frac{u^4}{c^4}-1\Big]+\frac{v^2-c^2}{\gamma p}\p_y^2p+2\frac{uv}{\gamma p}\p_{xy}^2p+\lambda m(\mathbf{x})k(x)\nonumber\\
&&\quad \quad + F_1(U) + F_2(U).
\end{eqnarray}
Multiplying both sides of $\eqref{4409}$ by $\gamma p$, we have
\begin{eqnarray}\label{4410}
&&L_p(U)\triangleq\left[2E - \frac{\gamma + 1}{\gamma - 1}c^2\right] \p^2_x p-c^2\p^2_y p+\Big[v^2\p^2_y p + 2uv\p^2_{xy}p\Big] \nonumber\\
&&\quad \quad\quad -\frac{2 }{ p}\left[E-\frac{c^2}{\gamma-1} + \frac{c^4}{4\gamma}\frac{1}{E- \frac{c^2}{\gamma - 1}}\right](\p_x p)^2+\Big[\frac{(\gamma-1)^2}{2}\frac{u^4}{c^4}+\frac{\gamma-1}{2}\frac{u^2}{c^2}-1\Big]\nonumber\\
&&\quad \quad\quad \times\frac{\lambda m(\mathbf{x})c^2}{u}\p_x p+\gamma \lambda^2m^2(\mathbf{x})\Big[\frac{(\gamma-1)^2}{4}\frac{u^4}{c^4}-1\Big]p-\frac{\lambda(\gamma-1)M^2+2\lambda}{2}\nonumber\\
&&\quad \quad\quad \times\gamma pu\p_x m(\mathbf{x})+\lambda m(\mathbf{x})\gamma pk(x)\nonumber\\
&&\quad =F_3(U)\triangleq-\gamma p(F_1+F_2).
\end{eqnarray}

\subsection{Linearization}
In this subsection, we separate the linear main parts from the nonlinear small parts in \eqref{4301}, \eqref{4302}, \eqref{4304}, \eqref{4410}, and \eqref{4402}, since we are dealing with a small perturbation of a special subsonic background solution $U_b$ as assumed in Theorem \ref{ThmSY}.
For convenience, we use $O(x)$ to denote a quantity satisfying $|O(x)| < C|x|$, with $C$ a generic positive constant depending only on the background solution ${U_b}$ and length $l$ of the duct.

\begin{definition}[Higher-order term]\label{Def41}
An expression of $\hat{U}=U-U_b,\,m(\mathbf{x})$ and their derivatives is called
a higher-order term, if it contains either
\begin{itemize}
\item $m(\mathbf{x}) - m_b(x)$ and its first-order derivatives;
\item  products of $\hat{U}$;
\item  the product of $\hat{U}$ and its derivatives ${\mathrm{D}}\hat{U}$.
\end{itemize}
\end{definition}
We notice that $G_1,\,F_1,\,F_2$ defined in \eqref{4403}, \eqref{4405}, and \eqref{4408} are all high-order terms. However, $G_2$ in \eqref{4403} is not, due to the first term in it.

\subsubsection{Linearize equation of total enthalpy $E$}
The background solution $E_b$ is constructed in Section \ref{sec41},  where it is shown that $\frac{\dd E_b}{\dd x} = -\lambda\frac{m_b(x)E_b}{u_b}$. By straightforward calculations,
\begin{eqnarray}\label{4411}
&&{\mathrm{D}}_{\mathbf{u}} E+\lambda m(\mathbf{x})E={\mathrm{D}}_{\mathbf{u}}\hat{E}+\frac{u\lambda m_b(x)}{u_b}\Big[\frac{1}{2}-\frac{1}{(\gamma-1)M^2_b}\Big]\hat{E}+\frac{u\lambda m_b(x)E_b}{\rho_bu^3_b}\hat{p}\nonumber\\
&&\quad \quad\quad \quad\quad \quad\quad \quad\quad \quad +\frac{u\lambda m_b(x)E_b\rho_b^{\gamma-1}}{(\gamma-1)u^3_b}\widehat{A(s)}+F_{E},
\end{eqnarray}
where
\begin{eqnarray*}
&&F_{E} = \frac{u\lambda E}{u_b}(m(\mathbf{x})-m_b(x))-\frac{u\lambda m_b(x)E_b}{u_b}(u-u_b)\Big[\frac{1}{u} - \frac{1 }{u_b}\Big]\nonumber\\
&&\quad \quad +\lambda u(m(\mathbf{x})E-m_b(x)E_b)\Big[\frac{1}{u} - \frac{1 }{u_b}\Big]+\frac{u\lambda m_b(x)E_b(u-u_b)}{(u+u_b)u^3_b}\nonumber\\
&&\quad \quad \times\Big[\hat{E}-\frac{1}{\rho_b}\hat{p}-\frac{\rho_b^{\gamma-1}}{\gamma-1 }\widehat{A(s)}\Big]+\frac{u\lambda m_b(x)E_bv^2}{(u+u_b)u^2_b} + O(|\hat{p}|^2+|\widehat{A(s)}|^2).\nonumber
\end{eqnarray*}
In the calculations, we used the following identity:
\begin{eqnarray}\label{4412}
&&u-u_b = \frac{1}{u_b}\Big[\hat{E}-\frac{1}{\rho_b}\hat{p}-\frac{\rho_b^{\gamma-1}}{\gamma-1 }\widehat{A(s)}\Big]+\frac{u_b-u}{(u+u_b)u_b}\Big[\hat{E}-\frac{1}{\rho_b}\hat{p}-\frac{\rho_b^{\gamma-1}}{\gamma-1 }\widehat{A(s)}\Big]
\nonumber\\
&&\quad \quad -\frac{v^2}{u+u_b} + O(|\hat{p}|^2+|\widehat{A(s)}|^2).
\end{eqnarray}
To see this, recalling that $u^2=2E-v^2-\frac{2\gamma}{\gamma-1}p^{1-\frac{1}{\gamma}}A(s)^{\frac{1}{\gamma}}$, then
\begin{eqnarray*}
u-u_b=\frac{1}{u+u_b}(u^2-u^2_b)=\frac{1}{u+u_b}\left[2E-v^2-\frac{2\gamma}{\gamma-1}p^{1-\frac{1}{\gamma}}A(s)^{\frac{1}
{\gamma}}-2E_b+\frac{2\gamma}{\gamma-1}p^{1-\frac{1}{\gamma}}_bA(s_b)^{\frac{1}{\gamma}}\right],
\end{eqnarray*}
and $\eqref{4412}$ is a consequence of  Taylor expansion up to second order at the background state $U_b$ with respect to the variables $E,\,A(s)$ and $p$.

Dividing both sides of $\eqref{4411}$ by $u$, one has
\begin{eqnarray}\label{4413}
{\mathrm{D}}'_{\mathbf{u}}\hat{E} \triangleq \p_x\hat{E} + \frac{v}{u}\p_y\hat{E}=\lambda a_1(\lambda,x)\hat{E} + \lambda a_2(\lambda,x)\widehat{A(s)} + \lambda a_3(\lambda,x)\hat{p}+ \lambda F_{E}',
\end{eqnarray}
where
\begin{equation}\label{4414}
\begin{cases}
a_1(\lambda,x)\triangleq\frac{m_b(x)}{u_b}\Big[\frac{1}{2}-\frac{1}{(\gamma-1)M^2_b}\Big]
,\\
a_2(\lambda,x)\triangleq\frac{m_b(x)E_b\rho_b^{\gamma-1}}{(\gamma-1)u^3_b} > 0,\\
a_3(\lambda,x)\triangleq\frac{m_b(x)E_b}{\rho_bu^3_b} > 0,
\end{cases}
\end{equation}
and
\begin{eqnarray}\label{4415}
&&F_{E}'=\frac{E}{u_b}(m(\mathbf{x})-m_b(x))-\frac{m_b(x)E_b}{u_b}(u-u_b)\Big[\frac{1}{u} - \frac{1}{u_b}\Big]+(m(\mathbf{x})E-m_b(x)E_b)\Big[\frac{1}{u} - \frac{1 }{u_b}\Big]\nonumber\\
&&\quad \quad +\frac{m_b(x)E_b(u-u_b)}{(u+u_b)u^3_b}\Big[\hat{E}-\frac{1}{\rho_b}\hat{p}-\frac{\rho_b^{\gamma-1}}{\gamma-1 }\widehat{A(s)}\Big]+\frac{m_b(x)E_bv^2}{(u+u_b)u_b} + O(|\hat{p}|^2+|\widehat{A(s)}|^2),
\end{eqnarray}
which are composed of  high-order terms and do not contain any derivative of $\hat{U}$.

\subsubsection{Linearize equation of entropy}
We linearize $\eqref{4302}_1$ at the background solution $A(s_b)$ in a similar way, where $\frac{\dd A(s_b)}{\dd x} = -\frac{\lambda \gamma m_b(x)A(s_b)}{u_b}\Big[1-\frac{\gamma-1}{2}M_b^2\Big]$. It follows that
\begin{eqnarray*}
&&{\mathrm{D}}_{\mathbf{u}} A(s)+\lambda \gamma m(\mathbf{x})A(s)\Big[1-\frac{\gamma-1}{2}M^2\Big]\nonumber\\
&&\quad = {\mathrm{D}}_{\mathbf{u}}\widehat{A(s)}+\frac{u}{u^3_b}\lambda \gamma m_b(x)A(s_b)\Big[\frac{3(\gamma-1)}{2}M_b^2-1\Big]\hat{E}\nonumber\\
&&\quad \quad +\frac{u\lambda  m_b(x)}{u_b}\Big[\gamma+\frac{1}{2}-\frac{(\gamma-1)^2}{2}M_b^2+\frac{1}{(\gamma-1)M_b^2}\Big]\widehat{A(s)}\nonumber\\
&&\quad \quad +\frac{u\lambda \gamma m_b(x)A(s_b)}{u^3_b\rho_b}\Big[1+\frac{\gamma-1}{2}M_b^2+\frac{(\gamma-1)^2}{2}M_b^4\Big]\hat{p}+\lambda F_{s},
\end{eqnarray*}
with
\begin{align*}
&F_{s} = \frac{u \gamma (\gamma-1)m_b(x)A(s_b)}{2u_b}\frac{u_b^2}{c_b^2}[c^2-c_b^2]\left[\frac{1}{c^2}-\frac{1}{c_b^2}\right]+u \gamma\left[m(\mathbf{x})A(s)\Big[1-\frac{\gamma-1}{2}M^2\Big]\right.\nonumber\\
&\left.\quad \quad -m_b(x)A(s_b)\Big[1-\frac{\gamma-1}{2}M_b^2\Big]\right]\Big[\frac{1}{u}-\frac{1}{u_b}\Big]+\frac{\gamma A(s)u}{u_b}\Big[1-\frac{\gamma-1}{2}M^2\Big]\nonumber\\
&\quad \quad \times[m(\mathbf{x})- m_b(x)] -\frac{\gamma m_b(x)A(s_b)u}{u_b} \Big[1-\frac{\gamma-1}{2}M_b^2\Big](u-u_b)\Big[\frac{1}{u}-\frac{1}{u_b}\Big]\nonumber\\
&\quad \quad -\frac{\gamma (\gamma-1)m_b(x)u }{2u_b}\widehat{A(s)}[M^2-M_b^2]-\frac{\gamma (\gamma-1)m_b(x)A(s_b)u }{2u_b}(|\mathbf{u}|^2-u_b^2)\nonumber\\
&\quad \quad \times\Big[\frac{1}{c^2}-\frac{1}{c_b^2}\Big]-\frac{\gamma m_b(x)A(s_b)uv^2}{u^2_b[u+u_b]}\Big[1-\frac{\gamma-1}{2}M_b^2\Big]\nonumber\\
&\quad \quad +\frac{\gamma m_b(x)A(s_b)u(u-u_b)} {u^2_b[(u+u_b)u_b]}\Big[1-\frac{\gamma-1}{2}M_b^2\Big]\Big[\hat{E}-\frac{1}{\rho_b}\hat{p}-\frac{\rho_b^{\gamma-1}}{\gamma-1 }\widehat{A(s)}\Big]\nonumber\\
&\quad \quad + O(|\hat{p}|^2 + |\widehat{A(s)}|^2).
\end{align*}
In the above calculations, we used $\eqref{4412}$ and the identity
\begin{eqnarray}\label{4416}
c^2 - c_b^2 = \frac{\gamma - 1}{\rho_b}\hat{p}+\rho_b^{\gamma - 1}\widehat{A(s)}+ O(|\hat{p}|^2 + |\widehat{A(s)}|^2)
\end{eqnarray}
based upon the definition
$c^2=\gamma p^{1-\frac{1}{\gamma}}A(s)^{\frac{1}{\gamma}}$ and Taylor expansion.
Therefore we have
\begin{eqnarray}\label{4417}
&&{\mathrm{D}}'_{\mathbf{u}}\widehat{A(s)} \triangleq \p_x\widehat{A(s)} + \frac{v}{u}\p_y\widehat{A(s)}=\lambda  b_1(\lambda,x) \hat{E} + \lambda  b_2(\lambda,x)\widehat{A(s)}\nonumber\\
&&\quad \quad \quad \quad\quad \quad+ \lambda b_3(\lambda,x)\hat{p} + \lambda F_{s}',
\end{eqnarray}
where
\begin{equation}\label{4418}
\begin{cases}\dl
b_1(\lambda,x)\triangleq\frac{\gamma m_b(x)A(s_b)}{u^3_b}\Big[1+\frac{\gamma-1}{2}M_b^2\Big]<0
,\\
\dl b_2(\lambda,x)\triangleq \frac{m_b(x)}{u_b}\Big[\gamma+\frac{1}{2}-\frac{(\gamma-1)^2}{2}M_b^2+\frac{1}{(\gamma-1)M_b^2}\Big] ,\\
\dl b_3(\lambda,x)\triangleq\frac{\gamma m_b(x)A(s_b)}{u^3_b\rho_b}\Big[1+\frac{\gamma-1}{2}M_b^2+\frac{(\gamma-1)^2}{2}M_b^4\Big]> 0,
\end{cases}
\end{equation}
and
\begin{align}\label{4419}
&F_{s}' =  \frac{\gamma (\gamma-1)m_b(x)A(s_b)}{2u_b}\frac{u_b^2}{c_b^2}[c^2-c_b^2]\Big[\frac{1}{c^2}-\frac{1}{c_b^2}\Big]+\gamma\left[m(\mathbf{x})A(s)\Big[1-\frac{\gamma-1}{2}M^2\Big]\right.\nonumber\\
&\left.\quad \quad - m_b(x)A(s_b)(1-\frac{\gamma-1}{2}M_b^2)\right]\Big[\frac{1}{u}-\frac{1}{u_b}\Big]+\frac{\gamma A(s)}{u_b}\Big[1-\frac{\gamma-1}{2}M^2\Big][m(\mathbf{x})- m_b(x)]\nonumber\\
&\quad \quad -\frac{\gamma m_b(x)A(s_b)}{u_b} \Big[1-\frac{\gamma-1}{2}M_b^2\Big](u-u_b)\Big[\frac{1}{u}-\frac{1}{u_b}\Big]-\frac{\gamma (\gamma-1)m_b(x)}{2u_b}\widehat{A(s)}[M^2-M_b^2]\nonumber\\
&\quad \quad -\frac{\gamma (\gamma-1)m_b(x)A(s_b)}{2u_b}(|\mathbf{u}|^2-u_b^2)\Big[\frac{1}{c^2}-\frac{1}{c_b^2}\Big]-\frac{\gamma m_b(x)A(s_b)v^2}{u^2_b[u+u_b]}\Big[1-\frac{\gamma-1}{2}M_b^2\Big]\nonumber\\
&\quad \quad +\frac{\gamma m_b(x)A(s_b)(u-u_b)}{u^2_b[(u+u_b)u_b]}\Big[1-\frac{\gamma-1}{2}M_b^2\Big]\Big[\hat{E}-\frac{1}{\rho_b}\hat{p}-\frac{\rho_b^{\gamma-1}}{\gamma-1 }\widehat{A(s)}\Big]\nonumber\\
&\quad \quad + O(|\hat{p}|^2 + |\widehat{A(s)}|^2).
\end{align}
Notice that $F_{s}'$ is also made of  higher order terms and does not contain any derivative of $\hat{U}$.

\subsubsection{Linearize equation of tangential velocity $v$}
Next, we linearize the equation $\eqref{4304}_1$. Using $\eqref{4412}$ and $\eqref{4416}$, one obtains
\begin{eqnarray}\label{4420}
&&\p_y v = -\frac{{\mathrm{D}}_{\mathbf{u}}p}{\gamma p} + \lambda m(\mathbf{x})\Big[1+\frac{(\gamma-1)M^2}{2}\Big] + \frac{1}{u}\Big[v\p_y u + \frac{1}{\rho}\p_xp\Big] - k(x)\nonumber\\
&&\quad = c_1(\lambda,x)\p_x\hat{p}+c_2(\lambda,x)\hat{p}+c_3(\lambda,x)\hat{E}+c_4(\lambda,x)\widehat{A(s)}+F_{v}-k(x),
\end{eqnarray}
where
\begin{equation}\label{4421}
\begin{cases}
c_1(\lambda,x)\triangleq \frac{1-M^2_b}{\rho_b u_b},\\
c_2(\lambda,x)\triangleq -\frac{\lambda m_b(x)}{2\rho_bc^2_b}\frac{(\gamma-1)^2M^8_b+(\gamma-1)M^6_b+(5\gamma^2-12\gamma+10)M^4_b+(7\gamma-11)M^2_b+2}{M_b^2(1-M^2_b)},\\
c_3(\lambda,x)\triangleq \frac{\lambda m_b(x)}{2u^2_b}\frac{3(\gamma-1)M^4_b+(\gamma+3)M^2_b+2}{1-M^2_b},\\
c_4(\lambda,x)\triangleq -\frac{\lambda m_b(x)\rho_b^{\gamma-1}}{2c^2_b}\frac{(\gamma-1)M^8_b+(\gamma-1)M^6_b+(2\gamma-6)M^4_b+(\gamma+5)M^2_b+2}{M_b^2(1-M^2_b)},
\end{cases}
\end{equation}
and
\begin{eqnarray}\label{4422}
&&F_{v} = \Big[\frac{1-\frac{u^2}{c^2}}{\rho u}-\frac{1-\frac{u^2_b}{c^2}}{\rho_b u_b}\Big]\p_x\hat{p}-\frac{\dd p_b}{\dd x}\Big[\frac{u^2}{c^2}-\frac{u^2_b}{c^2_b}\Big]\Big[\frac{1}{\rho u}-\frac{1}{\rho_b u_b}\Big]\nonumber\\
&&\quad \quad + [1-M^2_b]\frac{\dd p_b}{\dd x}\Big[\frac{1}{\rho}-\frac{1}{\rho_b}\Big]\Big[\frac{1}{u}-\frac{1}{u_b}\Big]-\frac{1}{\rho_b u_b}\frac{\dd p_b}{\dd x}[u^2-u^2_b]\Big[\frac{1}{c^2}-\frac{1}{c^2_b}\Big]\nonumber\\
&&\quad \quad - \frac{1}{\rho_b u_b}\frac{\dd p_b}{\dd x}[u-u_b]\Big[\frac{1}{u}-\frac{1}{u_b}\Big]-\frac{1}{\rho_b u_b}\frac{\dd p_b}{\dd x}[\rho-\rho_b]\Big[\frac{1}{\rho}-\frac{1}{\rho_b}\Big]\nonumber\\
&&\quad \quad + \frac{u_b}{\rho_bc^2_b}\Big[\frac{1}{c^2}-\frac{1}{c^2_b}\Big][c^2-c^2_b]+\Big[1+\frac{(\gamma-1)\lambda }{2}\frac{u^2}{c^2}\Big][\lambda m(\mathbf{x})-\lambda m_b(x)]\nonumber\\
&&\quad \quad + \frac{(\gamma-1)\lambda m(\mathbf{x})}{2}\frac{v^2}{c^2}+\frac{v}{u}\p_y u-\frac{v\p_y p}{\gamma p}-\frac{(\gamma-1)\lambda m_b(x)M^2_b}{2}\Big[\frac{1}{c^2}-\frac{1}{c^2_b}\Big]\nonumber\\
&&\quad \quad \times [c^2-c^2_b]-\frac{\lambda m_b(x)v^2}{2u_b(u+u_b)}\Big[2+(\gamma-1)M^2_b\Big]- \frac{\lambda m_b(x)v^2}{2c^2_b}\frac{3+(\gamma-1)M^2_b}{1-M^2_b}\nonumber\\
&&\quad \quad +\frac{\lambda m_b(x)(u_b-u)}{2u^2_b(u+u_b)}\Big[2+(\gamma-1)M^2_b\Big]\Big[\hat{E}-\frac{1}{\rho_b}\hat{p}-\frac{\rho_b^{\gamma-1}}{\gamma-1 }\widehat{A(s)}\Big] \nonumber\\
&&\quad \quad + \frac{(\gamma-1)\lambda m_b(x)}{2}\Big[\frac{1}{c^2}-\frac{1}{c^2_b}\Big][u^2-u^2_b]+ O(|\hat{p}|^2+|\widehat{A(s)}|^2).
\end{eqnarray}
By $\eqref{4316}$, one has
\begin{eqnarray}\label{4423}
&&k(x) = \frac{1}{\pi}\int_{0}^{\pi}\big[c_1(\lambda,x)\p_x\hat{p}+c_2(\lambda,x)\hat{p}+c_3(\lambda,x)\hat{E}+c_4(\lambda,x)\widehat{A(s)}+F_{v}\big]\,\dd y\nonumber\\
&&\quad =\frac{c_1(\lambda,x)}{\pi}\int_{0}^{\pi}\p_x\hat{p}\,\dd y+\frac{c_2(\lambda,x)}{\pi}\int_{0}^{\pi}\hat{p}\,\dd y+\frac{c_3(\lambda,x)}{\pi}\int_{0}^{\pi}\hat{E}\,\dd y\nonumber\\
&&\quad \quad + \frac{c_4(\lambda,x)}{\pi}\int_{0}^{\pi}\widehat{A(s)}\,\dd y+\frac{1}{\pi}\int_{0}^{\pi}F_{v}\,\dd y.
\end{eqnarray}
So $\eqref{4420}$ can be written as
\begin{eqnarray}\label{4424}
&&\p_y v = c_1(\lambda,x)\p_x\hat{p}+c_2(\lambda,x)\hat{p}+c_3(\lambda,x)\hat{E}+c_4(\lambda,x)\widehat{A(s)}-\frac{c_1(\lambda,x)}{\pi}\int_{0}^{\pi}\p_x\hat{p}\,\dd y\nonumber\\
&&\quad \quad - \frac{c_2(\lambda,x)}{\pi}\int_{0}^{\pi}\hat{p}\,\dd y-\frac{c_3(\lambda,x)}{\pi}\int_{0}^{\pi}\hat{E}\,\dd y-\frac{c_4(\lambda,x)}{\pi}\int_{0}^{\pi}\widehat{A(s)}\,\dd y+F_{v}',
\end{eqnarray}
with
\begin{eqnarray}\label{4425}
F_{v}'=F_{v}-\frac{1}{\pi}\int_{0}^{\pi}F_{v}\,\dd y.
\end{eqnarray}
Notice that $F_{v}'$ is consisted of higher-order terms and does not contain any derivative of $u,\,v$ with respect to $x$.

\subsubsection{Linearize equation of pressure $p$}
Observing that $L_p(U_b) = 0$, we will linearize $\eqref{4410}$ to obtain the equation of perturbed pressure in the form $\mathcal {L}(\hat{U}) = L_p(U) - L_p(U_b) + F_4(U)$, where $F_4(U)$ are higher-order terms. We  carry out the calculations term by term.

The first term:
\begin{eqnarray}\label{4426}
&&\Big[2E- \frac{\gamma + 1}{\gamma - 1}c^2\Big]\p^2_x p -\Big[2E_b- \frac{\gamma + 1}{\gamma - 1}c^2_b\Big]\frac{\dd^2 p_b}{\dd x^2}\nonumber\\
&&\quad =\Big[2\hat{E}- \frac{\gamma + 1}{\gamma - 1}(c^2 - c^2_b)\Big]\p_x^2\hat{p}+(u^2_b - c^2_b)\p_x^2\hat{p} + \Big[2\hat{E}- \frac{\gamma + 1}{\gamma - 1}(c^2 - c^2_b)\Big]\frac{\dd^2 p_b}{\dd x^2}.
\end{eqnarray}

The second term:
\begin{eqnarray}\label{4427}
c^2\p^2_y p =  [c^2 -c^2_b]\p^2_y \hat{p} + c^2_b\p^2_y \hat{p}.
\end{eqnarray}

The third term:
\begin{eqnarray}\label{4428}
&&\frac{2}{p}\Big[E- \frac{c^2}{\gamma - 1}\Big](\p_x p)^2 - \frac{2}{p_b}\Big[E_b- \frac{c^2_b}{\gamma - 1}\Big]\Big[\frac{\dd p_b}{\dd x}\Big]^2\nonumber\\
&&\quad = \frac{2}{p_b}\Big[\frac{\dd p_b}{\dd x}\Big]^2\hat{E}-\frac{2}{(\gamma - 1)p_b}\Big[\frac{\dd p_b}{\dd x}\Big]^2[c^2 - c^2_b]- \frac{u^2_b}{p^2_b}\Big[\frac{\dd p_b}{\dd x}\Big]^2\hat{p}\nonumber\\
&&\quad \quad + \frac{2u^2_b}{p_b}\frac{\dd p_b}{\dd x}\p_x\hat{p} + D_1(U),
\end{eqnarray}
where
\begin{eqnarray*}
&&D_1(U) =  \p_x(p + p_b)\p_x\hat{p}\Big[\frac{|\mathbf{u}|^2}{p} - \frac{u^2_b}{p_b}\Big]+ \Big[\frac{\dd p_b}{\dd x}\Big]^2(|\mathbf{u}|^2 - u^2_b)\Big[\frac{1}{p} - \frac{1}{p_b}\Big]\nonumber\\
&&\quad \quad - \Big[\frac{\dd p_b}{\dd x}\Big]^2\frac{\hat{p}u^2_b}{p_b}\Big[\frac{1}{p} - \frac{1}{p_b}\Big]
\end{eqnarray*}
involves only $\hat{U}$ and its first-order derivatives.

The fourth term:
\begin{eqnarray}\label{4429}
&&\frac{2}{p}\Big[\frac{c^4}{4\gamma }\frac{1}{E- \frac{c^2}{\gamma - 1}}\Big](\p_x p)^2 - \frac{2}{p_b}\Big[\frac{c^4_b}{4\gamma }\frac{1}{E_b- \frac{c_b^2}{\gamma - 1}}\Big]\Big[\frac{\dd p_b}{\dd x}\Big]^2 \nonumber\\
&&\quad = \frac{1}{\gamma}\left[\frac{c^4}{p|\mathbf{u}|^2}(\p_x p)^2 -  \frac{c_b^4}{p_b u_b^2}\Big(\frac{\dd p_b}{\dd x}\Big)^2\right] = - \frac{2}{\gamma p_b M_b^4}\Big[\frac{\dd p_b}{\dd x}\Big]^2\hat{E}\nonumber\\
&&\quad \quad +\frac{2}{\gamma(\gamma - 1) p_b M_b^4}[1+(\gamma - 1)M_b^2]\Big[\frac{\dd p_b}{\dd x}\Big]^2[c^2 - c_b^2]
- \frac{c_b^4}{\gamma p_b^2u_b^2}\Big[\frac{\dd p_b}{\dd x}\Big]^2\hat{p}\nonumber\\
&&\quad \quad + \frac{2c_b^4}{\gamma p_b u_b^2}\frac{\dd p_b}{\dd x}\p_x \hat{p} + D_2(U),
\end{eqnarray}
where
\begin{eqnarray*}
&&D_2(U) = \frac{\p_x(p + p_b)\p_x\hat{p}}{\gamma}\Big[\frac{c^4}{p|\mathbf{u}|^2} - \frac{c_b^4}{p_b u_b^2}\Big] + \frac{1}{\gamma}\frac{\dd p_b}{\dd x}\Big[\frac{1}{|\mathbf{u}|^2} - \frac{1}{u_b^2}\Big]\Big[\frac{c^4}{p} - \frac{c_b^4}{p_b}\Big]\nonumber\\
&&\quad \quad -\Big[\frac{\dd p_b}{\dd x}\Big]^2\frac{2c_b^4}{\gamma p_b u_b^2}\Big[\hat{E} - \frac{1}{\gamma -1}(c^2 - c_b^2)\Big]\Big[\frac{1}{|\mathbf{u}|^2} - \frac{1}{u_b^2}\Big] + \frac{c_b^4}{\gamma p_b u_b^2}(\p_x \hat{p})^2\nonumber\\
&&\quad \quad -\Big[\frac{\dd p_b}{\dd x}\Big]^2\frac{\hat{p}c_b^4}{\gamma p_b u_b^2}\Big[\frac{1}{p} - \frac{1}{p_b}\Big]+\frac{(c^2 - c_b^2)}{\gamma u_b^2}\Big[\frac{\dd p_b}{\dd x}\Big]^2\Big[\frac{c^2}{p} - \frac{c_b^2}{p_b}\Big]\nonumber\\
&&\quad \quad +\Big[\frac{\dd p_b}{\dd x}\Big]^2\frac{c_b^2}{\gamma u_b^2}[c^2 - c_b^2]\Big[\frac{1}{p} - \frac{1}{p_b}\Big].
\end{eqnarray*}
It contains only $\hat{U}$ and first-order derivatives of $\hat{U}$.

The fifth term:
\begin{eqnarray}\label{4430}
&&\frac{\lambda(\gamma-1)^2}{2} \frac{m(\mathbf{x})u^3}{c^2}\p_x p-\frac{\lambda(\gamma-1)^2}{2} \frac{m_b(x)u^3_b}{c^2_b}\frac{\dd p_b}{\dd x}\nonumber\\
&&\quad =\frac{\lambda(\gamma-1)^2}{2} \frac{m_b(x)u^3_b}{c^2_b}\p_x\hat{p}+\frac{3\lambda(\gamma-1)^2}{2}\frac{m_b(x)u_b^2}{c^2_b}\frac{\dd p_b}{\dd x}[u-u_b]\nonumber\\
&&\quad \quad - \frac{\lambda(\gamma-1)^2}{2}\frac{m_b(x)u^3_b}{c^4_b}\frac{\dd p_b}{\dd x}[c^2-c^2_b]+D_3(U),
\end{eqnarray}
where
\begin{eqnarray*}
&&D_3(U) =\frac{\lambda(\gamma-1)^2}{2}\Big[\frac{m(\mathbf{x})u^3}{c^2}-\frac{m_b(x)u^3_b}{c^2_b}\Big]\p_x\hat{p}+\frac{\lambda(\gamma-1)^2}{2}\frac{\dd p_b}{\dd x}\frac{u^3}{c^2}[m(\mathbf{x})-m_b(x)]\nonumber\\
&&\quad \quad +\frac{\lambda(\gamma-1)^2}{2}\frac{m_b(x)(u^2+u u_b-2u_b^2)}{c^2_b}\frac{\dd p_b}{\dd x}[u-u_b]+\frac{\lambda(\gamma-1)^2}{2}m_b(x)\frac{\dd p_b}{\dd x}[u^3-u^3_b]\nonumber\\
&&\quad \quad \times\Big[\frac{1}{c^2}-\frac{1}{c^2_b}\Big]-\frac{\lambda(\gamma-1)^2}{2}\frac{m_b(x)u^3_b}{c^2_b}\frac{\dd p_b}{\dd x}\Big[\frac{1}{c^2}-\frac{1}{c^2_b}\Big][c^2-c^2_b]
\end{eqnarray*}
consists of $\hat{U}$ and the first-order derivatives of $\hat{U}$.

The sixth term:
\begin{eqnarray}\label{4431}
&&\frac{\lambda(\gamma-1)}{2}m(\mathbf{x})u\p_x p-\frac{\lambda(\gamma-1)}{2}m_b(x)u_b\frac{\dd p_b}{\dd x}\nonumber\\
&&\quad =\frac{(\gamma-1)\lambda m_b(x) u_b}{2}\p_x\hat{p}+\frac{\lambda (\gamma-1)m_b(x)}{2}\frac{\dd p_b}{\dd x}[u-u_b]+D_4(U),
\end{eqnarray}
where
\begin{eqnarray*}
D_4(U) = \frac{\lambda (\gamma-1)}{2}[m(\mathbf{x})u-m_b(x)u_b]\p_x\hat{p}-\frac{\lambda (\gamma-1)u}{2}\frac{\dd p_b}{\dd x}[m(\mathbf{x})-m_b(x)]
\end{eqnarray*}
depends only on the first-order derivatives of $\hat{U}$ and $\hat{U}$ itself.

The seventh term:
\begin{eqnarray}\label{4432}
&&\frac{\lambda m(\mathbf{x})c^2}{u}\p_x p - \frac{\lambda m_b(x)c^2_b}{u_b}\frac{\dd p_b}{\dd x}\nonumber\\
&&\quad =\frac{\lambda m_b(x)c^2_b}{u_b}\p_x\hat{p}+\frac{\lambda m_b(x)}{u_b}\frac{\dd p_b}{\dd x}[c^2-c^2_b]-\frac{\lambda m_b(x)c^2_b}{u^2_b}\frac{\dd p_b}{\dd x}[u-u_b]+D_5(U),
\end{eqnarray}
with
\begin{eqnarray*}
&&D_5(U) = \Big[\frac{\lambda m(\mathbf{x})c^2}{u} - \frac{\lambda m_b(x)c^2_b}{u_b}\Big]\p_x \hat{p}+\frac{c^2}{u}\frac{\dd p_b}{\dd x}[\lambda m(\mathbf{x})-\lambda m_b(x)]\nonumber\\
&&\quad \quad + \lambda m_b(x)\frac{\dd p_b}{\dd x}[c^2-c^2_b]\Big[\frac{1}{u}-\frac{1}{u_b}\Big]-\frac{\lambda m_b(x)c^2_b}{u_b}\frac{\dd p_b}{\dd x}\Big[\frac{1}{u}-\frac{1}{u_b}\Big][u-u_b]
\end{eqnarray*}
involving only $\mathrm{D}\hat{U}$ and $\hat{U}$ itself.

The eighth term:
\begin{eqnarray}\label{4433}
&&\frac{\lambda(\gamma-1)M^2+2\lambda}{2}\gamma pu\p_xm(\mathbf{x})- \frac{\lambda(\gamma-1)M^2_b+2\lambda}{2}\gamma p_bu_b\frac{\dd m_b(x)}{\dd x}\nonumber\\
&&\quad =\frac{\gamma u_b(\lambda(\gamma-1)M^2_b+2\lambda)}{2} \frac{\dd m_b(x)}{\dd x}\hat{p}+\frac{\gamma p_b(\lambda(\gamma-1)M^2_b+2\lambda)}{2}\frac{\dd m_b(x)}{\dd x}[u-u_b]\nonumber\\
&&\quad \quad +\lambda(\gamma-1)\rho_bu_b\frac{\dd m_b(x)}{\dd x}\hat{E}-\lambda\rho_bu_b\left[1+\frac{\gamma-1}{2}M^2_b\right]\frac{\dd m_b(x)}{\dd x}[c^2-c^2_b]\nonumber\\
&&\quad \quad\quad \quad+D_6(U),
\end{eqnarray}
where
\begin{eqnarray*}
&&D_6(U) = \frac{\gamma pu(\lambda(\gamma-1)M^2+2\lambda)}{2}\Big[\p_1m(\mathbf{x})-\frac{\dd m_b(x)}{\dd x}\Big]-\frac{\lambda(\gamma-1)\rho_bu^3_b}{2}\frac{\dd m_b(x)}{\dd x}\nonumber\\
&&\quad \quad \times\Big[\frac{1}{c^2}-\frac{1}{c^2_b}\Big][c^2-c^2_b]+\frac{\lambda(\gamma-1)M^2_b+2\lambda}{2}\frac{\dd m_b(x)}{\dd x}\gamma \hat{p}[u-u_b]\nonumber\\
&&\quad +\frac{\lambda\gamma (\gamma-1) p_bu_b}{2}\frac{\dd m_b(x)}{\dd x}\Big[\frac{1}{c^2}-\frac{1}{c^2_b}\Big][|\mathbf{u}|^2-u^2_b]\nonumber\\
&&\quad \quad +\frac{\dd m_b(x)}{\dd x}\Big[\frac{\lambda(\gamma-1)M^2+2\lambda}{2}-\frac{\lambda(\gamma-1)M^2_b+2\lambda}{2}\Big][\gamma pu-\gamma p_bu_b].
\end{eqnarray*}
Each term in it only contains $\hat{U}$.

The ninth item:
\begin{eqnarray}\label{4434}
&&\gamma\lambda^2m^2(\mathbf{x})\Big[\frac{(\gamma-1)^2}{4}\frac{u^4}{c^4}-1\Big]p-\gamma\lambda^2m^2_b(x)\Big[\frac{(\gamma-1)^2}{4}\frac{u^4_b}{c^4_b}-1\Big]p_b\nonumber\\
&&\quad =\gamma\lambda^2m^2_b(x)\Big[\frac{(\gamma-1)^2}{4}\frac{u^4_b}{c^4_b}-1\Big]\hat{p} +\frac{\gamma\lambda^2(\gamma-1)^2m^2_b(x)u_b^3p_b}{c^4_b}[u-u_b]\nonumber\\
&&\quad \quad -\frac{\gamma\lambda^2(\gamma-1)^2m^2_b(x)u_b^4p_b}{2c^6_b}[c^2-c^2_b]+D_7(U),
\end{eqnarray}
where
\begin{eqnarray*}
&&D_7(U)=\gamma p_b\Big[\frac{(\gamma-1)^2}{4}\frac{u^4}{c^4}-1\Big][\lambda^2m^2(\mathbf{x})-\lambda^2m^2_b(x)]+\frac{\gamma \lambda^2(\gamma-1)^2m^2_b(x) p_b}{4}\nonumber\\
&&\quad \quad \times [u^4-u^4_b]\Big[\frac{1}{c^4}-\frac{1}{c^4_b}\Big]-\frac{\gamma \lambda^2(\gamma-1)^2m^2_b(x)u^4_b p_b}{4c^4_b}[c^4-c^4_b]\Big[\frac{1}{c^4}-\frac{1}{c^4_b}\Big]\nonumber\\
&&\quad \quad + \frac{\gamma \lambda^2(\gamma-1)^2m^2_b(x) p_b}{4c^4_b}[u^2-u^2_b]^2-\frac{\gamma \lambda^2(\gamma-1)^2m^2_b(x)u^4_b p_b}{4c^8_b}[c^2-c^2_b]^2\nonumber\\
&&\quad \quad +\gamma\left(\lambda^2m^2(\mathbf{x})\Big[\frac{(\gamma-1)^2}{4}\frac{u^4}{c^4}-1\Big]-\lambda^2m^2_b(x)\Big[\frac{(\gamma-1)^2}{4}\frac{u^4_b}{c^4_b}-1\Big]\right)\hat{p}\nonumber\\
&&\quad \quad
+\frac{\gamma \lambda^2(\gamma-1)^2m^2_b(x)u^2_b p_b}{2c^4_b}[u-u_b]^2
\end{eqnarray*}
contains terms depending only on $\hat{U}$.

For the tenth term, by $\eqref{4423}$, we have
\begin{eqnarray}\label{4435}
&&\lambda \gamma m(\mathbf{x})pk(x)=\lambda \gamma[m(\mathbf{x})-m_b(x)] pk(x)+\lambda \gamma m_b(x)k(x)\hat{p}+\lambda \gamma m_b(x)p_bk(x)\nonumber\\
&&\quad =\lambda \gamma[m(\mathbf{x})-m_b(x)] pk(x)+\lambda \gamma m_b(x)k(x)\hat{p}+\frac{\lambda \gamma m_b(x)p_bc_1(\lambda,x)}{\pi}\int_{0}^{\pi}\p_x\hat{p}\,\dd y\nonumber\\
&&\quad \quad +\frac{\lambda \gamma m_b(x)p_bc_2(\lambda,x)}{\pi}\int_{0}^{\pi}\hat{p}\,\dd y+\frac{\lambda \gamma m_b(x)p_bc_3(\lambda,x)}{\pi}\int_{0}^{\pi}\hat{E}\,\dd y\nonumber\\
&&\quad \quad +\frac{\lambda \gamma m_b(x)p_bc_4(\lambda,x)}{\pi}\int_{0}^{\pi}\widehat{A(s)}\,\dd y+\frac{\lambda \gamma m_b(x)p_b}{\pi}\int_{0}^{\pi}F_{v}\,\dd y\nonumber\\
&&\quad =\frac{\lambda \gamma m_b(x)p_bc_1(\lambda,x)}{\pi}\int_{0}^{\pi}\p_x\hat{p}\,\dd y+\frac{\lambda \gamma m_b(x)p_bc_2(\lambda,x)}{\pi}\int_{0}^{\pi}\hat{p}\,\dd y\nonumber\\
&&\quad \quad +\frac{\lambda \gamma m_b(x)p_bc_3(\lambda,x)}{\pi}\int_{0}^{\pi}\hat{E}\,\dd y+\frac{\lambda \gamma m_b(x)p_bc_4(\lambda,x)}{\pi}\int_{0}^{\pi}\widehat{A(s)}\,\dd y\nonumber\\
&&\quad \quad +D_8(U),
\end{eqnarray}
where
\begin{eqnarray*}
&&D_8(U)=\lambda \gamma[m(\mathbf{x})-m_b(x)] pk(x)+\lambda \gamma m_b(x)k(x)\hat{p}+\frac{\lambda \gamma m_b(x)p_b}{\pi}\int_{0}^{\pi}F_{v}\dd y
\end{eqnarray*}
depends only on $\hat{U}$.

Now substituting \eqref{4426}-\eqref{4435} into $\eqref{4410}$, we have
\begin{eqnarray}\label{4436}
&&L_p(\hat{U})=\left[2\hat{E}- \frac{\gamma + 1}{\gamma - 1}(c^2 - c^2_b)\right]\p_x^2\hat{p} + (u^2_b - c^2_b)\p_x^2\hat{p} - c^2_b\p^2_y \hat{p}
- [c^2 -c^2_b]\p^2_y \hat{p}\nonumber\\
&&\quad \quad +\Big[v^2\p^2_yp + 2uv\p^2_{xy}p\Big]+d_1(\lambda,x)\p_x\hat{p}+d_2(\lambda,x)\hat{p}+d_3(\lambda,x)\hat{E}\nonumber\\
&&\quad \quad +d_4(\lambda,x)[u-u_b]+d_5(\lambda,x)[c^2 - c^2_b]+d_6(\lambda,x)\int_{0}^{\pi}\p_x\hat{p}\,\dd y\nonumber\\
&&\quad \quad +d_7(\lambda,x)\int_{0}^{\pi}\hat{p}\,\dd y+d_8(\lambda,x)\int_{0}^{\pi}\hat{E}\,\dd y+d_9(\lambda,x)\int_{0}^{\pi}\widehat{A(s)}\,\dd y\nonumber\\
&&\quad \quad -D_1(U)-D_2(U)+D_3(U)+D_4(U)-D_5(U)-D_6(U)\nonumber\\
&&\quad \quad +D_7(U)+D_8(U),
\end{eqnarray}
where
\begin{equation}\label{4437}
\begin{cases}
d_1(\lambda,x)\triangleq-\frac{2u_b^2}{p_b}\Big[1+\frac{1}{\gamma M^4}\Big]\frac{\dd p_b}{\dd x}+\frac{\lambda m_b(x)c^2_b}{u_b}\Big[\frac{(\gamma-1)^2}{2}M^4_b+\frac{\gamma - 1}{2}M^2_b-1\Big]
,\\
d_2(\lambda,x)\triangleq\frac{u_b^2}{p_b}\Big[1+\frac{1}{\gamma M^4_b}\Big]\Big[\frac{\dd p_b}{\dd x}\Big]^2-\lambda\gamma u_b\Big[1+\frac{\gamma-1}{2}M^2_b\Big]\frac{\dd m_b(x)}{\dd x}\\ \quad\quad\quad \quad+\gamma\lambda^2m^2_b(x)\Big[\frac{(\gamma-1)^2}{4}\frac{u^4_b}{c^4_b}-1\Big],\\
d_3(\lambda,x)\triangleq2\frac{\dd^2 p_b}{\dd x^2}-\frac{2}{p_b}\Big[1-\frac{1}{\gamma M^4_b}\Big]\Big[\frac{\dd p_b}{\dd x}\Big]^2-\lambda(\gamma-1)\rho_b u_b\frac{\dd m_b(x)}{\dd x},\\
d_4(\lambda,x)\triangleq\frac{\lambda m_b(x)\left([3(\gamma-1)^2+2]M^2_b+(\gamma-1)\right)}{2}\frac{\dd p_b}{\dd x}-\frac{\gamma p_b(\lambda(\gamma-1)M^2_b+2\lambda)}{2}\frac{\dd m_b(x)}{\dd x}\\ \quad \quad\quad \quad+\frac{\gamma\lambda^2(\gamma-1)^2m^2_b(x)u_b^3p_b}{c^4_b},\\
d_5(\lambda,x)\triangleq-\frac{\gamma+1}{\gamma-1}\frac{\dd^2 p_b}{\dd x^2}+\frac{2}{(\gamma-1)p_b}\Big[1-\frac{1+(\gamma-1)M^2_b}{\gamma M_b^4}\Big]\Big[\frac{\dd p_b}{\dd x}\Big]^2\\
\quad \quad \quad \quad+\frac{\lambda m_b(x)}{u_b}\frac{\dd p_b}{\dd x}\Big[1-\frac{(\gamma-1)^2M_b^4}{2}\Big]+\lambda\rho_bu_b\Big[1+\frac{\gamma-1}{2}M^2_b\Big]\frac{\dd m_b(x)}{\dd x}\\
\quad \quad \quad \quad-\frac{\gamma\lambda^2(\gamma-1)^2m^2_b(x)u_b^4p_b}{2c^6_b},\\
d_6(\lambda,x)\triangleq\frac{\lambda \gamma m_b(x)p_bc_1(\lambda,x)}{\pi},\quad d_7(\lambda,x)\triangleq\frac{\lambda \gamma m_b(x)p_bc_2(\lambda,x)}{\pi},\\
d_8(\lambda,x)\triangleq\frac{\lambda \gamma m_b(x)p_bc_3(\lambda,x)}{\pi},\quad d_9(\lambda,x)\triangleq\frac{\lambda \gamma m_b(x)p_bc_4(\lambda,x)}{\pi}.
\end{cases}
\end{equation}
After replacing $u-u_b$ and $c^2-c^2_b$ in $\eqref{4436}$ by $\eqref{4412}$ and $\eqref{4416}$, and dividing $c^2_b$, we have
\begin{eqnarray}\label{4438}
&&\mathcal {L}(\hat{p}) \triangleq e_1(\lambda,x)\p_x^2\hat{p}-\p^2_y\hat{p}+
e_2(\lambda,x)\p_x\hat{p} + e_3(\lambda,x)\hat{p} +  e_4(\lambda,x)\hat{E} + e_5(\lambda,x)\widehat{A(s)}\nonumber\\
&&\quad \quad + e_6(\lambda,x)\int_{0}^{\pi}\p_x\hat{p}\,\dd y+e_7(\lambda,x)\int_{0}^{\pi}\hat{p}\,\dd y+e_8(\lambda,x)\int_{0}^{\pi}\hat{E}\,\dd y\nonumber\\
&&\quad \quad + e_9(\lambda,x)\int_{0}^{\pi}\widehat{A(s)}\,\dd y+\frac{1}{c_b^2}\Big[2\hat{E}- \frac{\gamma + 1}{\gamma - 1}(c^2 - c^2_b)\Big]\p_x^2\hat{p}-\frac{c^2-c_b^2}{c_b^2}\p^2_y\hat{p}\nonumber\\
&&\quad \quad + \frac{v^2}{c_b^2}\p^2_y\hat{p}+\frac{2uv}{c_b^2}\p^2_{xy}\hat{p} = F_5\triangleq \frac{1}{c_b^2}(F_3+F_4),
\end{eqnarray}
where
\begin{equation}\label{4439}
\begin{cases}
e_1(\lambda,x)\triangleq \Big[M_b^2(\lambda,x) - 1\Big] < 0,\quad e_2(\lambda,x)\triangleq \frac{d_1(\lambda,x)}{c^2_b},\\
e_3(\lambda,x)\triangleq \frac{d_2(\lambda,x)}{c^2_b}+\frac{(\gamma-1)d_5(\lambda,x)-\frac{d_4(\lambda,x)}{u_b}}{\rho_bc^2_b},\quad e_4(\lambda,x)\triangleq \frac{d_3(\lambda,x)}{c^2_b}+\frac{d_4(\lambda,x)}{u_bc^2_b},\\
e_5(\lambda,x)\triangleq \frac{\rho^{\gamma-1}_b}{(\gamma-1)c^2_b}[(\gamma-1)d_5(\lambda,x)-\frac{d_4(\lambda,x)}{u_b}],\quad e_6(\lambda,x)\triangleq \frac{d_6(\lambda,x)}{c^2_b},\\
e_7(\lambda,x)\triangleq \frac{d_7(\lambda,x)}{c^2_b},\quad e_8(\lambda,x)\triangleq \frac{d_8(\lambda,x)}{c^2_b},\quad e_9(\lambda,x)\triangleq \frac{d_9(\lambda,x)}{c^2_b},
\end{cases}
\end{equation}
and
\begin{eqnarray}\label{4440}
&&F_4(U) = - D_1(U)-D_2(U)+D_3(U)+D_4(U)-D_5(U)+D_6(U)+D_7(U)\nonumber\\
&&\quad \quad + D_8(U) + c_4(\lambda,x)\frac{u_b-u}{(u+u_b)u_b}\Big[\hat{E}-\frac{1}{\rho_b}\hat{p}-\frac{\rho_b^{\gamma-1}}{\gamma-1 }\widehat{A(s)}\Big] \nonumber\\
&&\quad \quad - c_4(\lambda,x)\frac{v^2}{u+u_b} + O(|\hat{p}|^2+|\widehat{A(s)}|^2).
\end{eqnarray}
Note that $F_4(U)$ consists of higher-order terms,  and  contains only $\hat{U}$, as well as its first-order derivatives. Therefore, the same is true for $F_5(U)$ defined in \eqref{4438}.

\begin{remark}\label{Remp1}
In previous works, such as the geometric effect studied in \cite[p.415, (4.22)]{Liu-Xu-Yuan-2016} or frictions considered in \cite[p.478, (2.26)]{Yuan-Zhao-2020}, the pressure equations do not contain nonlocal terms for purely subsonic flows. However, $\eqref{4438}$ has four integral nonlocal terms. This indicates that mass-additions have a stronger coupling effect. In addition, compared with  \cite{Liu-Xu-Yuan-2016} and \cite{Yuan-Zhao-2020}, we retain all the second-order derivatives of pressure $\hat{p}$ in the definition of the operator $\mathcal{L}$, which is useful for decreasing regularity required to design a nonlinear iteration scheme for further studies of subsonic Euler equations. The induced difficulty is solved by openness of invertible bounded linear operators on Banach spaces. \qed
\end{remark}

\subsubsection{Linearize boundary conditions of pressure $p$ on inlet}
We have obtained the boundary condition $\eqref{4402}$ for the pressure at the inlet. Now we linearize it at the background solution $p_b$:
\begin{eqnarray*}
\p_x(p-p_b) - \lambda \gamma\left[\frac{m(\mathbf{x})u [1+\frac{(\gamma-1)M^2}{2}] }{u^2 - c^2}p-\frac{m_b(x)u_b [1+\frac{(\gamma-1)M^2_b}{2}] }{u^2_b - c^2_b}p_b\right] = G_1+G_2,
\end{eqnarray*}
which implies the Robin condition
\begin{eqnarray}\label{4450}
\p_x\hat{p} + \lambda \gamma_0\hat{p}=G_p\triangleq G_1+ G_2+ \lambda G_3\qquad\text{on}\quad\Sigma_0,
\end{eqnarray}
where
\begin{eqnarray*}
\gamma_0 \triangleq -  \frac{\gamma m_b(0)u_b(0)}{c^2_b(0)}\frac{\frac{\gamma^2+1}{2 \gamma}M^4_b(0)+\frac{\gamma^2+3}{2 \gamma}M^2_b(0)+\frac{3\gamma-1}{2 \gamma}} {(M^2_b(0)-1)^2}<0,
\end{eqnarray*}
and
\begin{eqnarray}\label{4451}
&&G_3 = - \gamma\left\{\left[\frac{m(\mathbf{x})u[1+\frac{(\gamma-1)M^2}{2}] }{u^2 - c^2}-\frac{m_b(x)u_b [1+\frac{(\gamma-1)M^2_b}{2}] }{u^2_b - c^2_b}\right]\hat{p}+\frac{p_bu [1+\frac{(\gamma-1)M^2}{2}] }{u^2 - c^2}\right.\nonumber\\
&&\phantom{=\;\;}\left.\times[m(\mathbf{x})-m_b(x)] +\frac{(\gamma-1)}{2}\frac{p_bm_b(x)}{u^2_b - c^2_b}\Big[\frac{|\mathbf{u}|^2}{c^2}-\frac{u^2_b}{c^2_b}\Big][u-u_b]+\frac{(\gamma-1)}{2}\frac{p_bm_b(x)u_b}{u^2_b - c^2_b}\right.\nonumber\\
&&\phantom{=\;\;}\left.\times[|\mathbf{u}|^2-u^2_b]\Big[\frac{1}{c^2}-\frac{1}{c^2_b}\Big]+p_bm_b(x)\left[u \Big[1+\frac{(\gamma-1)M^2}{2}\Big]-u_b \Big[1+\frac{(\gamma-1)M^2_b}{2}\Big]\right]\right.\nonumber\\
&&\phantom{=\;\;}\left.\times\Big[\frac{1}{u^2 - c^2}-\frac{1}{u^2_b - c^2_b}\Big] +\frac{(\gamma-1)}{2}\frac{p_bm_b(x)M^2_b}{u^2_b - c^2_b}\frac{u_b-u}{(u+u_b)u_b}\Big[\hat{E}-\frac{1}{\rho_b}\hat{p}-\frac{\rho_b^{\gamma-1}}{\gamma-1 }\widehat{A(s)}\Big] \right.\nonumber\\
&&\phantom{=\;\;}\left.- \frac{(\gamma-1)}{2}\frac{p_bm_b(x)M^2_b}{u^2_b - c^2_b}\frac{v^2}{u+u_b} -\frac{\gamma+1}{2}\frac{p_bm_b(x)u_b}{(u^2_b - c^2_b)^2} [2\hat{E}-\frac{2\rho_b^{\gamma - 1}}{\gamma-1}\widehat{A(s)}]\right.\nonumber\\
&&\phantom{=\;\;}\left.+\frac{(\gamma-1)}{2}\frac{p_bm_b(x)M^2_b}{u^2_b - c^2_b}\frac{1}{u_b}\Big[\hat{E}-\frac{\rho_b^{\gamma-1}}{\gamma-1 }\widehat{A(s)}\Big] +\frac{p_bm_b(x)u_bv^2}{(u^2_b - c^2_b)^2} \Big[1+\frac{(\gamma-1)M^2_b}{2}\Big]\right.\nonumber\\
&&\phantom{=\;\;}\left.+\frac{p_bm_b(x)u_b}{(u^2_b - c^2_b)^2} \Big[1+(\gamma-1)M^2_b-\frac{\gamma-1}{2}M^4_b\Big]\Big[\frac{\gamma - 1}{\rho_b}\hat{p}+\rho_b^{\gamma - 1}\widehat{A(s)}\Big]\right.\nonumber\\
&&\phantom{=\;\;}\left.-\frac{p_bm_b(x)u_b}{u^2_b - c^2_b} \Big[1+\frac{(\gamma-1)M^2_b}{2}\Big]\Big[[u^2-u^2_b]-[c^2-c^2_b]\Big]\Big[\frac{1}{u^2 - c^2}-\frac{1}{u^2_b - c^2_b}\Big]\right.\nonumber\\
&&\phantom{=\;\;}\left.-\frac{(\gamma-1)}{2}\frac{p_bm_b(x)u^3_b}{c^2_b(u^2_b - c^2_b)}\left[\frac{1}{c^2}-\frac{1}{c^2_b}\right][c^2-c^2_b]
\right\}+ O(|\hat{p}|^2 + |\widehat{A(s)}|^2),
\end{eqnarray}
which are  higher-order terms depending only on $\hat{U}$.

\subsection{A decoupled nonlinear problem of pressure}

Up to now the (perturbed) total enthalpy $\hat{E}$ and entropy $\widehat{A(s)}$ are coupled with  \eqref{4438}, the equation of  (perturbed) pressure $\hat{p}$.  We solve the Cauchy problems for the transport equations of $\hat{E}$ and $\widehat{A(s)}$, i.e.,  $\eqref{4413}$ and $\eqref{4417}$, to obtain a decoupled second-order elliptic equation of $\hat{p}$. We will see the price is introducing more nonlocal terms that reflecting the effects of the hyperbolic modes.

Let $\gamma(t;x,y)$ be the integral curve of the vector field $\p_x+\frac{v}{u}\p_y$ that passing through the point $(x,y)$, satisfying $y=\gamma(x;x,y)$, cf. \eqref{4708}. Set \begin{align}\label{eq452new}\xi(x,y)=\gamma(0;x,y).\end{align} Then  $(0,\xi(x,y))$ is the point where the streamline  intersects with the entrance $\Sigma_0$. Now introducing
\begin{eqnarray*}
\begin{cases}
\hat{E}^{\flat}(t) \triangleq \hat{E}(t, \gamma(t;x,y)),&\\
\widehat{A(s^{\flat})}(t) \triangleq \widehat{A(s)}(t, \gamma(t;x,y)),&\\
\hat{p}^{\flat}(t) \triangleq \hat{p}(t, \gamma(t;x,y)),&\\
F_{E}^{\flat}(t) \triangleq F_{E}'(t, \gamma(t;x,y)),& (\text{cf.}~ \eqref{4415})\\
F_{s}^{\flat}(t) \triangleq F_{s}'(t, \gamma(t;x,y)),& (\text{cf}.~ \eqref{4419})\\
\end{cases}
\end{eqnarray*}
then equations $\eqref{4413}$ and $\eqref{4417}$ read
\begin{eqnarray}\label{4441}
\begin{cases}
\frac{\dd\hat{E}^{\flat}}{\dd t} = \lambda a_1(\lambda,t)\hat{E}^{\flat} + \lambda a_2(\lambda,t)\widehat{A(s^{\flat})} + \lambda a_3(\lambda,t)\hat{p}^{\flat} + \lambda F_{E}^{\flat}&\quad t \in [0,l],\\
\frac{\dd\widehat{A(s^{\flat})}}{\dd t} =  \lambda b_1(\lambda,t)\hat{E}^{\flat} + \lambda b_2(\lambda,t)\widehat{A(s^{\flat})} + \lambda b_3(\lambda,t)\hat{p}^{\flat} + \lambda F_{s}^{\flat}&\quad t \in [0,l],\\
\hat{E}^{\flat}(0)= \hat{E}_0(\xi(x,y))= E_0(\xi(x,y)) - E_b|_{t=0}&\quad t = 0,\\
\widehat{A(s^{\flat}(0))} = \widehat{A(s_0)}(\xi(x,y)) = A(s_0)(\xi(x,y)) - A(s_b|_{t=0})&\quad t = 0.
\end{cases}
\end{eqnarray}
It is considered as a Cauchy problem for a nonhomogeneous system of first-order linear ordinary differential equations with variable coefficients and the unknowns $\hat{E}^{\flat}$, ~$\widehat{A(s^{\flat})}$. Let $\Phi(\lambda, t)$ be the $2\times2$ standard fundamental matrix, where $\Phi(0, t)$ is the unit matrix. Then $\eqref{4441}$ is solved by
\begin{eqnarray}\label{4442}
&\left[
\begin{matrix}
\hat{E}(x,y)\\
\widehat{A(s)}(x,y)
\end{matrix}
\right]=
\left[
\begin{matrix}
\hat{E}^{\flat}(x)\\
\widehat{A(s^{\flat})}(x)
\end{matrix}
\right]=\Phi(\lambda, x)\left[
\begin{matrix}
\hat{E}_0(\xi(x,y))\\
\widehat{A(s_0)}(\xi(x,y))
\end{matrix}
\right]\\
&\qquad\qquad\qquad\qquad\qquad+\lambda\Phi(\lambda, x)\int^{x}_0\Phi^{-1}(\lambda, t)\Big(\left[
\begin{matrix}
a_3(\lambda,t)\\
b_3(\lambda,t)
\end{matrix}
\right]\hat{p}^{\flat}+\left[
\begin{matrix}
F_{E}^{\flat}\\
F_{s}^{\flat}
\end{matrix}
\right]\Big)\dd t.\nonumber
\end{eqnarray}
We use $\eqref{4442}$ to replace the term $$e_4(\lambda,x)\hat{E} + e_5(\lambda,x)\widehat{A(s)}+e_8(\lambda,x)\int_{0}^{\pi}\hat{E}\,\dd y+e_9(\lambda,x)\int_{0}^{\pi}\widehat{A(s)}\,\dd y$$ in $\eqref{4438}$. 
For the convenience of calculations, we divide it into two parts. Firstly, for  $e_4(\lambda,x)\hat{E} + e_5(\lambda,x)\widehat{A(s)}$, we have 
\begin{eqnarray}\label{4443}
&&e_4(\lambda,x)\hat{E}(x,y) + e_5(\lambda,x)\widehat{A(s)}(x,y)= \Big[
\begin{matrix}
e_4(\lambda,x),e_5(\lambda,x)\\
\end{matrix}
\Big]\Phi(\lambda, x)\left[
\begin{matrix}
\hat{E}_0(\xi(x,y))\\
\widehat{A(s_0)}(\xi(x,y))
\end{matrix}
\right]\nonumber\\
&&\quad \quad + \lambda\Big[
\begin{matrix}
e_4(\lambda,x),e_5(\lambda,x)\\
\end{matrix}
\Big]\Phi(\lambda, x)\int^{x}_0\Phi^{-1}(\lambda, t)\Big(\left[
\begin{matrix}
a_3(\lambda,t)\\
b_3(\lambda,t)
\end{matrix}
\right]\hat{p}^{\flat} + \left[
\begin{matrix}
F_{E}^{\flat}\\
F_{s}^{\flat}
\end{matrix}
\right]\Big)\dd t\nonumber\\
&&\quad =e_{10}(\lambda,x)\hat{E}_0(\xi(x,y)) + e_{11}(\lambda,x)\widehat{A(s_0)}(\xi(x,y)) +  \lambda e_{10}(\lambda,x)\int^{x}_0a_4(\lambda, t)\hat{p}^{\flat}\,\dd t\nonumber\\
&&\quad \quad +\lambda e_{11}(\lambda, x)\int^{x}_0b_4(\lambda, t)\hat{p}^{\flat}\,\dd t + \lambda \left[
\begin{matrix}
e_{10}(\lambda,x), e_{11}(\lambda,x)\\
\end{matrix}
\right]\int^{x}_0\Phi^{-1}(\lambda, t)\left[
\begin{matrix}
F_{E}^{\flat}\\
F_{s}^{\flat}
\end{matrix}
\right] \dd t\nonumber\\
&&\quad =e_{10}(\lambda, x)\hat{E}_0(\xi(x,y)) + e_{11}(\lambda, x)\widehat{A(s_0)}(\xi(x,y)) +  \lambda e_{10}(\lambda, x)\int^{x}_0a_4(\lambda, t)\hat{p}(t, y)\,\dd t\nonumber\\
&&\quad \quad + \lambda e_{11}(\lambda, x)\int^{x}_0b_4(\lambda, t)\hat{p}(t, y)\,\dd t + \lambda e_{10}(\lambda, x)\int^{x}_0a_4(\lambda, t)(\hat{p}^{\flat}(t) - \hat{p}(t, y))\,\dd t\nonumber\\
&&\quad \quad +\lambda e_{11}(\lambda, x)\int^{x}_0b_4(\lambda, \tau)(\hat{p}^{\flat}(t)-\hat{p}(t, y))\,\dd t \nonumber\\
&&\quad \quad+ \lambda\left[
\begin{matrix}
e_{10}(\lambda, x),e_{11}(\lambda, x)\\
\end{matrix}
\right]\int^{x}_0\Phi^{-1}(\lambda, t)\left[
\begin{matrix}
F_{E}^{\flat}\\
F_{s}^{\flat}
\end{matrix}
\right]\dd t.
\end{eqnarray}
Secondly, for the part $e_8(\lambda,x)\int_{0}^{\pi}\hat{E}\,\dd y+e_9(\lambda,x)\int_{0}^{\pi}\widehat{A(s)}\,\dd y$, we get
\begin{eqnarray}\label{4444}
&&e_8(\lambda,x)\int_{0}^{\pi}\hat{E}(x,y)\,\dd y+e_9(\lambda,x)\int_{0}^{\pi}\widehat{A(s)}(x,y)\,\dd y\nonumber\\
&&\quad=\Big[
\begin{matrix}
e_8(\lambda,x),e_9(\lambda,x)\\
\end{matrix}
\Big]\left[
\begin{matrix}
\int_{0}^{\pi}\hat{E}(x,y)\,\dd y\\
\int_{0}^{\pi}\widehat{A(s)}(x,y)\,\dd y
\end{matrix}
\right]\nonumber\\
&&\quad =\Big[
\begin{matrix}
e_8(\lambda,x),e_9(\lambda,x)\\
\end{matrix}
\Big]\Phi(\lambda, x)\left[
\begin{matrix}
\int_{0}^{\pi}\hat{E}_0(\xi(x,y))\,\dd y\\
\int_{0}^{\pi}\widehat{A(s_0)}(\xi(x,y))\,\dd y
\end{matrix}
\right]+\lambda\Big[
\begin{matrix}
e_8(\lambda,x),e_9(\lambda,x)\\
\end{matrix}
\Big]\Phi(\lambda, x)\nonumber\\
&&\quad \quad \times\int_{0}^{\pi}\int^{x}_0\Phi^{-1}(\lambda, t)\left[
\begin{matrix}
a_3(\lambda,t)\\
b_3(\lambda,t)
\end{matrix}
\right]\hat{p}^{\flat}\,\dd t\dd y \nonumber\\
&&\quad\quad\quad+\lambda\Big[
\begin{matrix}
e_8(\lambda,x),e_9(\lambda,x)\\
\end{matrix}
\Big]\Phi(\lambda, x)\int_{0}^{\pi}\int^{x}_0\Phi^{-1}(\lambda, t)\left[
\begin{matrix}
F_{E}^{\flat}\\
F_{s}^{\flat}
\end{matrix}
\right]\,\dd t\dd y\nonumber\\
&&\quad =e_{12}(\lambda,x)\int_{0}^{\pi}\hat{E}_0(\xi(x,y))\,\dd y+e_{13}(\lambda,x)\int_{0}^{\pi}\widehat{A(s_0)}(\xi(x,y))\,\dd y\nonumber\\
&&\quad \quad +\lambda e_{12}(\lambda,x)\int_{0}^{\pi}\int^{x}_0a_4(\lambda, t)\hat{p}^{\flat}\,\dd t\dd y+\lambda e_{13}(\lambda,x)\int_{0}^{\pi}\int^{x}_0b_4(\lambda, t)\hat{p}^{\flat}\,\dd t\dd y\nonumber\\
&&\quad \quad +\lambda\left[
\begin{matrix}
e_{12}(\lambda, x),e_{13}(\lambda, x)\\
\end{matrix}
\right]\int_{0}^{\pi}\int^{x}_0\Phi^{-1}(\lambda, t)\left[
\begin{matrix}
F_{E}^{\flat}\\
F_{s}^{\flat}
\end{matrix}
\right]\,\dd t\dd y\nonumber\\
&&\quad = e_{12}(\lambda,x)\int_{0}^{\pi}\hat{E}_0(\xi(x,y))\,\dd y+e_{13}(\lambda,x)\int_{0}^{\pi}\widehat{A(s_0)}(\xi(x,y))\,\dd y\nonumber\\
&&\quad \quad +\lambda e_{12}(\lambda,x)\int_{0}^{\pi}\int^{x}_0a_4(\lambda, t)\hat{p}(t, y)\,\dd t\dd y+\lambda e_{13}(\lambda,x)\int_{0}^{\pi}\int^{x}_0b_4(\lambda, t)\hat{p}(t, y)\,\dd t\dd y\nonumber\\
&&\quad \quad +\lambda e_{12}(\lambda,x)\int_{0}^{\pi}\int^{x}_0a_4(\lambda, t)(\hat{p}^{\flat}(t)-\hat{p}(t, y))\,\dd t\dd y \nonumber\\
&&\qquad\qquad\qquad+\lambda e_{13}(\lambda,x)\int_{0}^{\pi}\int^{x}_0b_4(\lambda, t)(\hat{p}^{\flat}(t)-\hat{p}(t, y))\,\dd t\dd y\nonumber\\
&&\qquad\qquad\qquad+\lambda\left[
\begin{matrix}
e_{12}(\lambda, x),e_{13}(\lambda, x)\\
\end{matrix}
\right]\int_{0}^{\pi}\int^{x}_0\Phi^{-1}(\lambda, t)\left[
\begin{matrix}
F_{E}^{\flat}\\
F_{s}^{\flat}
\end{matrix}
\right]\dd t\dd y.
\end{eqnarray}
Notice that "$\times$" in the third line of $\eqref{4444}$  represents  multiplication of matrices. For simplicity, we have also introduced the following notations
\begin{align}\label{4445}
&\left[
\begin{matrix}
a_4(\lambda, x),b_4(\lambda, x)\\
\end{matrix}
\right]^{\top} = \Phi^{- 1}(\lambda, x)\left[
\begin{matrix}
a_3(\lambda, x),b_3(\lambda, x)\\
\end{matrix}
\right]^{\top},\\
&\left[
\begin{matrix}
e_{10}(\lambda, x),e_{11}(\lambda, x)\\
\end{matrix}
\right] = \left[
\begin{matrix}
e_4(\lambda, x),e_5(\lambda, x)\\
\end{matrix}
\right]\Phi(\lambda, x),\\
&\left[
\begin{matrix}
e_{12}(\lambda, x),e_{13}(\lambda, x)\\
\end{matrix}
\right] = \left[
\begin{matrix}
e_8(\lambda, x),e_9(\lambda, x)\\
\end{matrix}
\right]\Phi(\lambda, x).
\end{align}
Substituting $\eqref{4443}$ and $\eqref{4444}$ into $\eqref{4438}$, we have
\begin{align}\label{4446}
\mathcal {L}(\hat{p}) = \mathcal {L}_1(\hat{p}) + \mathcal {L}_2(\hat{p})=F_0+F_p,\end{align}
where $\mathcal {L}_1(\hat{p})$ is the principal part of $\mathcal {L}(\hat{p})$, and $\mathcal {L}_2(\hat{p})$ is the quadratic higher-order part:
\begin{equation}\label{4449}
\begin{cases}
\mathcal {L}_1(\hat{p})\triangleq e_1(\lambda,x)\p_x^2\hat{p}-\p_y^2\hat{p}+\lambda e_2(\lambda,x)\p_x\hat{p}+e_3(\lambda,x)\hat{p}+e_6(\lambda,x)\int_{0}^{\pi}\p_x\hat{p}(x, y)\,\dd y\\
\quad\quad+e_7(\lambda,x)\int_{0}^{\pi}\hat{p}(x, y)\,\dd y+\lambda e_{10}(\lambda, x)\int^{x}_0a_4(\lambda, t)\hat{p}(t, y)\,\dd t\\
\quad\quad+\lambda e_{11}(\lambda, x)\int^{x}_0b_4(\lambda, t)\hat{p}(t, y)\,\dd t+\lambda e_{12}(\lambda,x)\int_{0}^{\pi}\int^{x}_0a_4(\lambda, t)\hat{p}(t, y)\,\dd t\dd y\\
\quad\quad+\lambda e_{13}(\lambda,x)\int_{0}^{\pi}\int^{x}_0b_4(\lambda, t)\hat{p}(t, y)\,\dd t\dd y,&\\
\mathcal {L}_2(\hat{p})\triangleq \left\{\frac{1}{c_b^2}\left[2({E-E_b})- \frac{\gamma + 1}{\gamma - 1}(c^2 - c^2_b)\right]\right\}\p_x^2\hat{p}-\left\{\frac{c^2-c_b^2}{c_b^2}\right\}\p^2_y\hat{p}+\left\{\frac{v^2}{c_b^2}
\right\}\p^2_y\hat{p}+\left\{\frac{2uv}{c_b^2}\right\}\p^2_{xy}\hat{p}.
\end{cases}
\end{equation}
The right-hand side of \eqref{4446} is given by
\begin{align}\label{4447}
&- F_0 = e_{10}(\lambda, x)\hat{E}_0(\xi(x,y)) + e_{11}(\lambda, x)\widehat{A(s_0)}(\xi(x,y))+e_{12}(\lambda,x)\int_{0}^{\pi}\hat{E}_0(\xi(x,y))\,\dd y\nonumber\\
&\quad \quad \quad+e_{13}(\lambda,x)\int_{0}^{\pi}\widehat{A(s_0)}(\xi(x,y))\,\dd y,\\
&- F_6 = e_{10}(\lambda, x)\int^{x}_{0}a_4(\lambda, t)(\hat{p}(t, \gamma(t;x,y)) - \hat{p}(t, y))\,\dd t\nonumber\\
&\quad \quad \quad+ e_{11}(\lambda, x)\int^{x}_{0}b_4(\lambda, t)(\hat{p}(t, \gamma(t;x,y)) - \hat{p}(t, y))\,\dd t\nonumber\\
&\quad \quad \quad+e_{12}(\lambda,x)\int_{0}^{\pi}\int^{x}_0a_4(\lambda, t)(\hat{p}(t, \gamma(t;x,y))-\hat{p}(t, y))\,\dd t\dd y\nonumber\\
&\quad \quad \quad+e_{13}(\lambda,x)\int_{0}^{\pi}\int^{x}_0b_4(\lambda, t)(\hat{p}(t, \gamma(t;x,y))-\hat{p}(t, y))\,\dd t\dd y\nonumber\\
&\quad \quad \quad+ \left[
\begin{matrix}
e_{10}(\lambda, x),e_{11}(\lambda, x)
\end{matrix}
\right]\int^{x}_{0}\Phi^{-1}(\lambda, t)\left[
\begin{matrix}
F_{E}^{\flat}\\
F_{s}^{\flat}
\end{matrix}
\right]\dd t\nonumber\\
&\quad \quad \quad+\left[
\begin{matrix}
e_{12}(\lambda, x),e_{13}(\lambda, x)\\
\end{matrix}
\right]\int_{0}^{\pi}\int^{x}_0\Phi^{-1}(\lambda, t)\left[
\begin{matrix}
F_{E}^{\flat}\\
F_{s}^{\flat}
\end{matrix}
\right]\,\dd t\dd y,\label{eq462addnew12}\\
& F_p =  F_5 + \lambda F_6.\label{eq463addnew2}
\end{align}
By Lagrange's mean value theorem,
\begin{eqnarray}\label{4448}
&&\hat{p}(t, \gamma(t;x,y)) - \hat{p}(t, y)=\hat{p}(t, \gamma(t;x,y)) - \hat{p}(t, \gamma(x;x,y))\nonumber\\
&&\quad =\p_{\gamma}\hat{p}(t,\xi_0)(\gamma(t;x,y)-\gamma(x;x,y))=\p_{\gamma}\hat{p}(t, \xi_0)\Big[\frac{\dd \gamma(t;x,y)}{\dd t}\Big]\Big|_{t=\theta_0}(t-x)\nonumber\\
&&\quad =\p_{\gamma}\hat{p}(t, \xi_0)\frac{v(\theta_0, \gamma(\theta_0;x,y))}{u(\theta_0, \gamma(\theta_0;x,y))}(t-x),
\end{eqnarray}
where  $\xi_0$ lies between $\gamma(t;x,y)$ and $y=\gamma(x;x,y)$, and $\theta_0$ is between $t$ and $x$. From $\eqref{4448}$,  we infer that $F_6(U)$ is a high-order term.

\begin{remark}\label{rmkp2}
We see that $\eqref{4446}$ contains six integral nonlocal terms in different directions, which further indicates that the mass-additions have stronger coupling effect. Comparing this with frictions, for which the decoupled pressure equation in \cite{Yuan-Zhao-2020} (p.487, (2.75)) contains only one integral nonlocal term. This brings more trouble to the analysis of linear problem in the next section. \qed
\end{remark}

\subsection{Problem \uppercase\expandafter{\romannumeral2}}\label{sec44new}
By Lemma \ref{Lem41}, from $\eqref{4413},\,\eqref{4417},\,\eqref{4424},\,\eqref{4450},\,\eqref{4446},$ and \eqref{4449}, we state Problem \uppercase\expandafter{\romannumeral2}, which is equivalent to  Problem \uppercase\expandafter{\romannumeral1}.
\begin{framed}
\underline{Problem \uppercase\expandafter{\romannumeral2}}:\
Show well-posedeness of classical subsonic solutions for the following coupled nonlinear problems \eqref{4452}-\eqref{4455} in $\Omega$.
\end{framed}
\begin{eqnarray}
&&\begin{cases}\label{4452}
{\mathrm{D}}'_{\mathbf{u}}\hat{E} = \lambda a_1(\lambda,x)\hat{E} + \lambda a_2(\lambda,x)\widehat{A(s)} + \lambda a_3(\lambda,x)\hat{p}+ \lambda F_{E}'&\quad \text{in}\quad \Omega,\\
\hat{E}=E_0(y)-E_b&\quad \text{on}\quad \Sigma_0;\\
\end{cases}\\
&&\begin{cases}\label{4453}
{\mathrm{D}}'_{\mathbf{u}}\widehat{A(s)} =   \lambda b_1(\lambda,x) \hat{E} + \lambda b_2(\lambda,x)\widehat{A(s)}  + \lambda b_3(\lambda,x)\hat{p} + \lambda F_{s}'&\quad \text{in}\quad \Omega,\\
\widehat{A(s)}=A(s_0)(y)-A(s_b)&\quad \text{on}\quad \Sigma_0;\\
\end{cases}\\
&&\begin{cases}\label{4454}
\mathcal{L}(\hat{p}) = \mathcal{L}_1(\hat{p}) + \mathcal {L}_2(\hat{p}) = F_0 + F_p &\quad \text{in}\quad \Omega,\\
\p_y\hat{p}= 0&\quad \text{on}\quad W_-\cup W_+,\\
\p_x\hat{p} + \lambda\gamma_0\hat{p}=G_p&\quad \text{on}\quad \Sigma_0,\\
\hat{p}=p_l(y) - p_b(l)&\quad \text{on}\quad \Sigma_l;\\
\end{cases}\\
&&\begin{cases}\label{4455}
\p_yv=c_1(\lambda,x)\p_x\hat{p}+c_2(\lambda,x)\hat{p}+c_3(\lambda,x)\hat{E}  +c_4(\lambda,x)\widehat{A(s)}+\frac{c_1(\lambda,x)}{\pi}\int_{0}^{\pi}\p_x\hat{p}\,\dd y\\
\quad \quad +\frac{c_2(\lambda,x)}{\pi}\int_{0}^{\pi}\hat{p}\,\dd y +\frac{c_3(\lambda,x)}{\pi}\int_{0}^{\pi}\hat{E}\,\dd y+\frac{c_4(\lambda,x)}{\pi}\int_{0}^{\pi}\widehat{A(s)}\,\dd y+F_{v}' \quad\text{in}\ \Omega,\\
v=0\quad \text{on}\ W_-\cup W_+.
\end{cases}
\end{eqnarray}
We see that \eqref{4452} and \eqref{4453} are Cauchy problems for transport equations of (perturbed) total enthalpy $\hat{E}$ and entropy $\widehat{A(s)}$; $\eqref{4454}$ is a mixed boundary value problem for a second-order elliptic equation of pressure $\hat{p}$ with multiple integral nonlocal terms; and $\eqref{4455}$ is a two-point boundary value problem for ordinary differential equations of tangential velocity $v$ on each cross-section $\Sigma_x$, with $x\in[0,l]$ being a parameter. 
In the following four sections, we study the related linear problems in various H\"{o}lder spaces.


\section{Mixed boundary-value problems for second-order linear elliptic equations with multiple integral nonlocal terms }\label{sec44}
In this section, we establish existence and estimates of solutions for the following boundary-value problem of a second-order elliptic equation with multiple integral nonlocal terms, which is deduced from \eqref{4454}, and the nonhomogeneous terms satisfy $h\in C^{k - 2,\alpha}(\bar{\Omega}), g_0\in C^{k - 1,\alpha}(\Sigma_0),$ and $g_1\in C^{k - 1,\alpha}(\Sigma_l)$, $k = 2,3,\cdots$:
\begin{eqnarray}\label{4503}
(\mathbb{L}):~~
\begin{cases}
\mathcal {L}(\hat{p}) = \mathcal {L}_1(\hat{p}) + \mathcal {L}_2(\hat{p}) = h(x)&\quad \text{in} \quad \Omega,\\
\p_y\hat{p}= 0&\quad \text{on} \quad W_-\cup W_+,\\
\p_x\hat{p} + \lambda\gamma_0\hat{p} = g_0(y)&\quad \text{on} \quad \Sigma_0,\\
\hat{p} = g_l(y)&\quad \text{on} \quad \Sigma_l.
\end{cases}
\end{eqnarray}
This is a linear problem, since the coefficients in ``$\{~\}$" in the definition of  $\mathcal {L}_2(\hat{p})$ in \eqref{4449} are assumed to be given $C^{1,\alpha}$ functions. Moreover, coefficients of the operator $\mathcal {L}_1(\hat{p})$ are $C^\infty$ and could be calculated explicitly for a given background solution.  

The methods we developed to show uniqueness and existence of Problem ($\mathbb{L}$) are based upon Fourier series, which require that the solutions, the coefficients, the boundary conditions, and all non-homogeneous terms,  being periodic in the $y$-direction with a period $2\pi$.   
This is achieved by an argument as in  \cite[pp.5310--5311]{Gao-Liu-Yuan-2020}. Briefly speaking,  we extend a solution $U$ 
to $(0,l)\times(-\infty, +\infty)$ as follows. According to the slip condition $\eqref{4204}$ and symmetry condition $\eqref{4208}$, $v$ is odd-extended along $y$-variable with respect to $y = k\pi$, while $E,\,p,\,s$ and $m(\mathbf{x})$ are even-extended along $y$-variable with respect to $y = k\pi$, with $k\in\mathbb{Z}$ here. It can be verified that $u, \rho$ have the same symmetry as $E, p, s$. Based on the assumption of symmetry $\eqref{4207}$ at the inlet and outlet of the duct, we can also make similar extensions to these boundary conditions.  

The main part of this section is devoted to studying well-posedness of the following simpler problem:
\begin{eqnarray}\label{4504}
(\mathbb{L}_1):~~
\begin{cases}
\mathcal {L}_1(\hat{p}) = h(\mathbf{x})&\quad \text{in} \quad \Omega,\\
\p_y\hat{p}= 0&\quad \text{on} \quad W_-\cup W_+,\\
\p_x\hat{p} + \lambda\gamma_0\hat{p} = g_0(y)&\quad \text{on} \quad \Sigma_0,\\
\hat{p} = g_l(y)&\quad \text{on} \quad \Sigma_l.
\end{cases}
\end{eqnarray}
Conclusions on Problem ($\mathbb{L}$) follows easily by smallness of coefficients in the operator $\mathcal{L}_2(\hat{p})$ and a well-known fact in functional analysis.

\subsection{Uniqueness}\label{secp21}
We need to find sufficient conditions so that the homogeneous problem
\begin{eqnarray}\label{4505}
\begin{cases}
\mathcal {L}_1(\hat{p}) = 0&\quad \text{in} \quad \Omega,\\
\p_y\hat{p}= 0&\quad \text{on} \quad W_-\cup W_+,\\
\p_x\hat{p} + \lambda\gamma_0\hat{p} = 0&\quad \text{on} \quad \Sigma_0,\\
\hat{p} = 0&\quad \text{on} \quad \Sigma_l
\end{cases}
\end{eqnarray}
admits only one strong solution, namely $\hat{p}\equiv0$, in the Sobolev space $H^2(\Omega)$.
With such a regularity,  the nonlocal term
\begin{eqnarray}\label{44addnew}
&&e_6(\lambda,x)\int_{0}^{\pi}\p_x\hat{p}(x, y)\,\dd y+e_7(\lambda,x)\int_{0}^{\pi}\hat{p}(x, y)\,\dd y+\lambda e_{10}(\lambda, x)\int^{x}_0a_4(\lambda, t)\hat{p}(t, y)\,\dd t\nonumber\\
&&\quad \quad +\lambda e_{11}(\lambda, x)\int^{x}_0b_4(\lambda, t)\hat{p}(t, y)\,\dd t+\lambda e_{12}(\lambda,x)\int_{0}^{\pi}\int^{x}_0a_4(\lambda, t)\hat{p}(t, y)\,\dd t\dd y\nonumber\\
&&\quad \quad +\lambda e_{13}(\lambda,x)\int_{0}^{\pi}\int^{x}_0b_4(\lambda, t)\hat{p}(t, y)\,\dd t\dd y,
\end{eqnarray}
lies in $H^{\frac12}(\Omega)$ by trace theorem.
Regard it as a nonhomogeneous term for the elliptic operator $e_1(\lambda,x)\p_x^2\hat{p}-\p_y^2\hat{p}+\lambda e_2(\lambda,x)\p_x\hat{p}+e_3(\lambda,x)\hat{p}$ appeared in \eqref{4449}, whose coefficients are $C^\infty$ because of $m_b\in C^\infty$,   by the standard regularity theory of  second order elliptic equations with Dirichlet and oblique derivative conditions in Sobolev spaces (see Gilbarg-Trudinger \cite[Chapter 8, and p.215]{Gilbarg-Tudinger-1998}), we have $\hat{p}\in H^{5/2}(\Omega)$, so \eqref{44addnew} belongs to $H^1(\Omega)$. Regularity theory then implies that $\hat{p}\in H^3(\Omega)$, hence \eqref{44addnew} lies in $H^2(\Omega)$.   Sobolev embedding theorem yields that  for any $\alpha \in (0,1)$, $H^2(\Omega)\subset C^{0,\alpha}(\bar{\Omega})$.  Thus by the Schauder theory \cite[Chapter 6]{Gilbarg-Tudinger-1998}, we infer $\hat{p}\in C^{2,\alpha}(\bar{\Omega})$. A bootstrap argument shows that  
$\hat{p} \in C^{\infty}(\bar{\Omega})$.

The smoothness of $\hat{p}$ implies  that its  Fourier series
\begin{eqnarray}\label{4506}
\hat{p}(x,y) = \frac{p_0(x)}{2}+\sum\limits_{m = 1}^{\infty}p_m(x)\cos(my)
\end{eqnarray}
converges uniformly in $\Omega$ and one can calculate its derivatives term by term, where the Fourier coefficients
\begin{eqnarray*}
p_m(x) = \frac{2}{\pi}\int^{\pi}_0\hat{p}(x,y)\cos(my)\,\dd y, \quad m=0,1,2\cdots
\end{eqnarray*}
are smooth functions on $[0,l]$.

Substituting $\eqref{4506}$ into $\eqref{4505}_1$, one has
\begin{eqnarray*}
&&0=\mathcal {L}_1(\hat{p})=\frac{1}{2}\Big[e_1(\lambda,x)p''_0(x)+[\lambda e_2(\lambda,x)+\pi e_6(\lambda,x)]p'_0(x)+[e_3(\lambda,x)\nonumber\\
&&\quad \quad +\pi e_7(\lambda,x)]p_0(x)+\lambda [e_{10}(\lambda, x)+\pi e_{12}(\lambda, x)]\int^{x}_0a_4(\lambda, t)p_0(\tau)\,\dd t\nonumber\\
&&\quad \quad +\lambda [e_{11}(\lambda, x)+\pi e_{13}(\lambda, x)]\int^{x}_0b_4(\lambda, t)p_0(t)\,\dd t\Big]+\sum\limits_{m = 1}^{\infty}\Big[e_1(\lambda,x)p''_m(x)\nonumber\\
&&\quad \quad +\lambda e_2(\lambda,x)p'_m(x)+[e_3(\lambda,x)+m^2]p_m(x)+\lambda e_{10}(\lambda, x)\int^{x}_0a_4(\lambda, t)p_m(t)\,\dd t\nonumber\\
&&\quad \quad +\lambda e_{11}(\lambda, x)\int^{x}_0b_4(\lambda, t)p_m(t)\,\dd t\Big]\cos(my).
\end{eqnarray*}
Substitute $\eqref{4506}$ into $\eqref{4505}_3$ shows
\begin{eqnarray*}
0=\p_x\hat{p}(0,y) + \lambda\gamma_0\hat{p}(0,y)=\frac{1}{2}\Big[p'_0(0)+\lambda\gamma_0p'_0(0)\Big] + \sum\limits_{m = 1}^{\infty}\Big[p'_m(0)+\lambda\gamma_0p'_m(0)\Big]\cos(my).
\end{eqnarray*}
Similarly, $\eqref{4506}$ and $\eqref{4505}_4$ yield
\begin{eqnarray*}
0=\hat{p}(l,y)=\frac{p_0(l)}{2}+\sum\limits_{m = 1}^{\infty}p_m(l)\cos(my).
\end{eqnarray*}
From these,  we obtain countable-infinite two-point boundary value problems for integro-differential equations of the Fourier coefficients.
For $m=0$,
\begin{eqnarray}\label{4507}
\begin{cases}
e_1(\lambda,x)p''_0(x)+[\lambda e_2(\lambda,x)+\pi e_6(\lambda,x)]p'_0(x)+[e_3(\lambda,x)+\pi e_7(\lambda,x)]p_0(x)\\
\quad \quad +\lambda [e_{10}(\lambda, x)+\pi e_{12}(\lambda, x)]\int^{x}_0a_4(\lambda, t)p_0(t)\,\dd t\\
\quad \quad +\lambda [e_{11}(\lambda, x)+\pi e_{13}(\lambda, x)]\int^{x}_0b_4(\lambda, t)p_0(t)\,\dd t=0,\qquad  x \in (0, l),\\
p_{0}'(0) + \lambda\gamma_0p_{0}(0) = 0,\\
p_{0}(l) = 0.
\end{cases}
\end{eqnarray}
For $m=1,\,2,\,\cdots$,
\begin{eqnarray}\label{4508}
\begin{cases}
e_1(\lambda,x)p_{m}'' + \lambda e_2(\lambda,x)p_{m}'+ [e_3(\lambda,x) + m^2]p_{m}+\lambda e_{10}(\lambda,x)\int^{x}_0a_4(\lambda,t)p_{m}(t)\,\dd t\\
\quad \quad +\lambda e_{11}(\lambda,x)\int^{x}_0b_4(\lambda,t)p_{m}(t)\,\dd t = 0,\qquad x \in (0, l),\\
p_{m}'(0) + \lambda\gamma_0p_{m}(0) = 0,\\
p_{m}(l) = 0.
\end{cases}
\end{eqnarray}
We call  \eqref{4507}-\eqref{4508} as  Problem ($\mathscr{C}$).  
Notice that the Problem ($\mathscr{C}$) is totally determined by the background solution $U_b$ and $l$, from which the coefficients were calculated explicitly.
\begin{definition}
  We say the background solution $U_b$ fulfills the S-condition, if Problem ($\mathscr{C}$) has only the trivial solutions  $p_m\equiv0$ for each $m\in\mathbb{N}$.
\end{definition}

Obviously the S-condition guarantees uniqueness of Problem $(\mathbb{L}_1)$. To characterize it,
we  introduce
\begin{align}\label{eq48add}
&\mathcal {P}^{(1)}_{m}(x) \triangleq \int^{x}_0b_4(\lambda,t)p_{m}(t) \,\dd t,\quad
\mathcal {P}^{(2)}_{m}(x) \triangleq \int^{x}_0a_4(\lambda,t)p_{m}(t) \,\dd t,\nonumber\\
&\qquad\qquad\mathcal {P}^{(3)}_{m}(x) \triangleq p_{m}(x),\quad
\mathcal {P}^{(4)}_{m}(x) \triangleq p'_{m}(x)
\end{align}
to eliminate the nonlocal terms in Problem ($\mathscr{C}$), which are transformed to the following  two-point boundary value problems for system of ordinary differential equations:
\begin{eqnarray}\label{4509}
\begin{cases}
\tilde{e}_1(\lambda, x){\mathcal {P}^{(4)}}'_{m}(x) + \tilde{e}_2(\lambda, x){\mathcal {P}^{(4)}}_{m}(x) + \tilde{e}_3(\lambda, x){\mathcal {P}^{(3)}}_{m}(x) + \tilde{e}_4(\lambda, x){\mathcal {P}^{(2)}}_{m}(x)\\
\quad \quad + \tilde{e}_5(\lambda, x){\mathcal {P}^{(1)}}_{m}(x) = 0,\\
{\mathcal {P}^{(3)}}'_{m}(x) = {\mathcal {P}^{(4)}}_{m}(x),\\
{\mathcal {P}^{(2)}}'_{m}(x) = a_4(\lambda,x){\mathcal {P}^{(3)}}_{m}(x),\\
{\mathcal {P}^{(1)}}'_{m}(x) = b_4(\lambda,x){\mathcal {P}^{(3)}}_{m}(x),\quad x \in (0, l),\\
{\mathcal {P}^{(1)}}_{m}(0)=0,\\
{\mathcal {P}^{(2)}}_{m}(0) = 0,\\
{\mathcal {P}^{(4)}}_{m}(0) + \lambda\gamma_0{\mathcal {P}^{(3)}}_{m}(0) = 0,\\
{\mathcal {P}^{(3)}}_{m}(l) = 0,
\end{cases}
\end{eqnarray}
where
\begin{eqnarray}\label{4510}
\begin{cases}
\tilde{e}_1(\lambda, x) \triangleq e_1(\lambda, x)\quad \quad \quad m=0,1,2,\cdots\\
\tilde{e}_2(\lambda, x) \triangleq
\begin{cases}
\lambda e_2(\lambda,x)+\pi e_6(\lambda,x)&\quad m=0, \\
\lambda e_2(\lambda,x)&\quad m=1,2,\cdots
\end{cases}\\
\tilde{e}_3(\lambda, x) \triangleq
\begin{cases}
e_3(\lambda,x)+\pi e_7(\lambda,x)&\quad m=0, \\
e_3(\lambda,x) + m^2&\quad m=1,2,\cdots
\end{cases}\\
\tilde{e}_4(\lambda, x) \triangleq
\begin{cases}
\lambda [e_{10}(\lambda, x)+\pi e_{12}(\lambda, x)]&\quad m=0, \\
\lambda e_{10}(\lambda, x)&\quad m=1,2,\cdots
\end{cases}\\
\tilde{e}_5(\lambda, x) \triangleq
\begin{cases}
\lambda [e_{11}(\lambda, x)+\pi e_{13}(\lambda, x)]&\quad m=0, \\
\lambda e_{11}(\lambda, x)&\quad m=1,2,\cdots
\end{cases}
\end{cases}
\end{eqnarray}
Let $z_{mi}={\mathcal {P}^{(i)}}_{m}(x), ~i=1,\, 2,\, 3,\, 4$. Then $\eqref{4509}$ can be written as
\begin{eqnarray}\label{4511}
\begin{cases}
\mathbf{z}_m'= A_m\mathbf{z}_m,&\\
B\mathbf{z}_m(0)+C\mathbf{z}_m(l)=\mathbf{0},
\end{cases}
\end{eqnarray}
where
\begin{eqnarray*}
\mathbf{z}_m=\left[\begin{matrix}
z_{m1}\\
z_{m2}\\
z_{m3}\\
z_{m4}\\
\end{matrix}\right],\quad
A_m \triangleq \left[\begin{matrix}
0&0&b_4(\lambda, x)&0\\
0&0&a_4(\lambda, x)&0\\
0&0&0&1\\
-\frac{\tilde{e}_5(\lambda, x)}{\tilde{e}_1(\lambda, x)}&-\frac{\tilde{e}_4(\lambda, x)}{\tilde{e}_1(\lambda, x)}&-\frac{\tilde{e}_3(\lambda, x)}{\tilde{e}_1(\lambda, x)}&-\frac{\tilde{e}_2(\lambda, x)}{\tilde{e}_1(\lambda, x)}\\
\end{matrix}\right],\\
B \triangleq \left[\begin{matrix}
1&0&0&0\\
0&1&0&0\\
0&0&\lambda \gamma_0&1\\
0&0&0&0\\
\end{matrix}\right],\quad C \triangleq \left[\begin{matrix}
0&0&0&0\\
0&0&0&0\\
0&0&0&0\\
0&0&1&0\\
\end{matrix}\right].
\end{eqnarray*}
Let $\Psi_m(\lambda,x)=(z_{mi}^j)_{1\le i,j\le 4}$ be the standard fundamental matrix corresponding to the linear system of  ordinary differential equations in $\eqref{4511}$, thus $\Psi_m(\lambda,0)$ is the identity matrix.
Then $\mathbf{z}_m = \Psi_m(\lambda,x)[z^1_m, z^2_m, z^3_m, z^4_m]^{\top}$ is the solution to the Cauchy problem
\begin{equation}\label{4512}
\begin{cases}
\mathbf{z}_m'= A_m\mathbf{z}_m,&\\
\mathbf{z}_m(0)=[z^1_m, z^2_m, z^3_m, z^4_m]^{\top}.
\end{cases}
\end{equation}
It solves \eqref{4511} if and only if
\begin{eqnarray}\label{4513}
\Big[B + C\Psi_m(\lambda,l)\Big][z^1_m, z^2_m, z^3_m, z^4_m]^{\top} = \mathbf{0}.
\end{eqnarray}
So the determinant of
\begin{eqnarray*}
B + C\Psi_m(\lambda,l) =
\left[\begin{matrix}
1&0&0&0&\\
0&1&0&0&\\
0&0&\lambda \gamma_0&1\\
z_{m3}^1(l)&z_{m3}^2(l)&z_{m3}^3(l)&z_{m3}^4(l)&\\
\end{matrix}\right]
\end{eqnarray*}
is nonzero, namely
\begin{eqnarray}\label{4514}
\chi_m(\lambda) \triangleq \lambda \gamma_0z_{m3}^4(l)-z_{m3}^3(l) \neq 0
\end{eqnarray}
is both sufficient and necessary for $\mathbf{z}_m\equiv0$, or equivalently, $p_m\equiv0$.
Thus, we have proved the following lemma.
\begin{lemma}\label{Lem43}
The following statements are equivalent:
\begin{itemize}
\item[i)] if  Problem ($\mathbb{L}_1$) has a strong solution in $H^2(\Omega)$, then it is unique;
\item[ii)]
$U_b$ satisfies S-condition;
\item[iii)]
for any $m=0, 1, 2, \cdots$, the matrix $B + C\Psi_m(l)$ is invertible;
\item[iv)]
for any $m=0, 1, 2, \cdots$, $\eqref{4514}$ holds.
\end{itemize}
\end{lemma}

The following lemma shows that for {\em almost all} $\lambda \in \mathbb{R}$, the resultant background solution $U_b$ determined by the parameters $U_b(0)$, $m_b$ and $\lambda$ satisfies the S-condition. 

\begin{lemma}\label{Lem44}
For given constant state  $U_b(0)$ on the entrance and positive $C^\infty(\mathbb{R}^+)$  function $m_b(x)$, there is a  set $\mathcal {S}\subset \mathbb{R}$ with at most countable infinite points, so that for $\lambda\in \mathbb{R}\backslash \mathcal {S}$, the resultant background solution $U_b$ satisfies the S-condition.
\end{lemma}
We remark that for the given   $U_b(0)$  and  $m_b(x)$, for $\lambda$ in a bounded set, the critical length $l_*$ has a positive lower bound which is uniform with respect to any such $\lambda$. Hence the length $l$ of the duct could be chosen a priori.
\begin{proof}
By real-analytic dependance on parameters for solutions of ordinary differential equations (Walter \cite[Section 13, p.149]{Walter-1998}) and Lemma \ref{lem421}, we know that  for any $m=0, 1, 2, \cdots$, the coefficients in  $\eqref{4509}$ are real analytic with respect to $\lambda\in \mathbb{R}$.  Thus for all $m=0, 1, 2, \cdots$, $\chi_m(\lambda)$ is real analytic on $\mathbb{R}$.

Now for the special case  $\lambda = 0$,
\begin{itemize}
\item[i)] by $\eqref{4414}$, $a_j(0,x)\,(j=1,\,2,\,3)$ are constants, and $a_2(0,x)>0,\,a_3(0,x)>0$;
\item[ii)]
by $\eqref{4418}$, $b_j(0,x)\,(j=1,\,2,\,3)$ are constants, and $b_1(0,x)>0,\,b_3(0,x)>0$;
\item[iii)]
by $\eqref{4421}$, $c_j(0,x)\,(j=1,\,2,\,3,\,4)$ are constants, and $c_1(0,x)>0, c_k(0,x)\equiv0,\,(k=2,\,3,\,4)$;
\item[iv)]
by $\eqref{4437}$, $d_j(0,x)\equiv0,\,(j=1,\,\cdots,\,9)$;
\item[v)]
by $\eqref{4439}$, $e_j(0,x)\,(j=1,\,\cdots,\,9)$ are constants, and $e_1(0,x)<0,\, e_k(0,x)\equiv0, \,(k=2,\,\cdots,\,9)$;
\item[vi)]
by $\eqref{4445}$, $a_4(0,x),~ b_4(0,x),~ e_{j}(0,x)\,(j=10,\,11,\,12,\,13)$ are constants, and $a_4(0,x)=a_3(0,x)>0,~ b_4(0,x)=b_3(0,x)>0,~ e_{k}(0,x)\equiv0,\,(k=10,\,11,\,12,\,13)$;
\item[vii)]
by $\eqref{4510}$, $\tilde{e}_j(0,x)\,(j=1,\,\cdots,\,5)$ are constants, and $\tilde{e}_1(0,x)<0,~ \tilde{e}_3(0,x)=|m|^2\geq0,~ ~\tilde{e}_k(0,x)\equiv0,\,(k=2,\,4,\,5)$.
\end{itemize}
Hence, we obtain that
\begin{eqnarray}\label{4515}
\chi_m(0)=-z_{m3}^3(l),
\end{eqnarray}
where $z_{m3}^3(l)$ is the value of $z_{m3}^3$ at $x = l$, and $z_{m3}^3$ is the third component of the solution to Cauchy problem $\eqref{4512}$ (with $\lambda = 0$) corresponding to the initial data $\mathbf{z}_m^3(0) = [0, 0, 1, 0]^{\top}$. 
We now show that for any $m=0,\,1,\,2,\,\cdots$, it holds $z_{m3}^3(l)\neq 0$.

When $\lambda = 0$, for any $m=0,\,1,\,2,\,\cdots$, the equations in $\eqref{4512}$ can be written as the following linear homogeneous system with constant-coefficients
\begin{eqnarray}\label{4516}
\left[\begin{matrix}
z_{m1}'\\
z_{m2}'\\
z_{m3}'\\
z_{m4}'\\
\end{matrix}\right]
=\left[\begin{matrix}
0&0&b_4(0,x)&0\\
0&0&a_4(0,x)&0\\
0&0&0&1\\
0&0&-\frac{m^2}{\tilde{e}_1(0,x)}&0\\
\end{matrix}\right]\left[\begin{matrix}
z_{m1}\\
z_{m2}\\
z_{m3}\\
z_{m4}\\
\end{matrix}\right].
\end{eqnarray}
The solution $(z_{m1}^3, z_{m2}^3, z_{m3}^3, z_{m4}^3)^\top$ to $\eqref{4516}$ with initial data $\mathbf{z}_m^3(0) = [0, 0, 1, 0]^{\top}$ are as follows:
\begin{itemize}
\item[i)]
$z_{01}^3=b_4(0,x)x,~ z_{02}^3=a_4(0,x)x,~ z_{03}^3=1, ~z_{04}^3=0$;
\item[ii)]
for $m=1,\,2,\,\cdots$,
\begin{eqnarray*}
\begin{cases}
z^3_{m1}=\frac{b_4(0,x)}{2}\left(-\frac{\sqrt{-\tilde{e}_1(0,x)}}{m}\exp(-\frac{m}{\sqrt{-\tilde{e}_1(0,x)}}x)+\frac{\sqrt{-\tilde{e}_1(0,x)}}{m}\exp(\frac{m}{\sqrt{-\tilde{e}_1(0,x)}}x)\right),&\\
z^3_{m2}=\frac{a_4(0,x)}{2}\left(-\frac{\sqrt{-\tilde{e}_1(0,x)}}{m}\exp(-\frac{m}{\sqrt{-\tilde{e}_1(0,x)}}x)+\frac{\sqrt{-\tilde{e}_1(0,x)}}{m}\exp(\frac{m}{\sqrt{-\tilde{e}_1(0,x)}}x)\right),&\\
z^3_{m3}=\frac{1}{2}\left(\exp(-\frac{m}{\sqrt{-\tilde{e}_1(0,x)}}x)+\exp(\frac{m}{\sqrt{-\tilde{e}_1(0,x)}}x)\right),&\\
z^3_{m4}=\frac{1}{2}\left(-\frac{m}{\sqrt{-\tilde{e}_1(0,x)}}\exp(-\frac{m}{\sqrt{-\tilde{e}_1(0,x)}}x)+\frac{m}{\sqrt{-\tilde{e}_1(0,x)}}\exp(\frac{m}{\sqrt{-\tilde{e}_1(0,x)}}x)\right),&
\end{cases}
\end{eqnarray*}
where 
$a_4(0,x)>0,~ b_4(0,x)>0,~ \tilde{e}_1(0,x)<0$ are constants now.
\end{itemize}
Therefore, for $m=0$, there holds $\chi_m(0) = -z_{03}^3(l) = -1<0$; for $m=1,2,\cdots$, there holds $\chi_m(0) = -z_{m3}^3(l)=-\frac{1}{2}\left(\exp(-\frac{m}{\sqrt{-\tilde{e}_1(0,x)}}l)+\exp(\frac{m}{\sqrt{-\tilde{e}_1(0,x)}}l)\right)<0$.
Since for a nontrivial real-analytic function, there are at most finite roots lying in a compact set,  we see that there exists a countable set $\mathcal {S}_m\subset \mathbb{R}$ for any $m=0,\,1,\,2,\,\cdots$, such that   $\chi_m(\lambda) \neq 0$ on $\mathbb{R}\backslash \mathcal {S}_m$.  Let $\mathcal {S}=\cup_{m\in \mathbb{N}}\mathcal {S}_m$.  Then $\mathcal {S}$ is also a countable subset of $\mathbb{R}$, and for any $\lambda\in \mathbb{R}\backslash \mathcal {S}$, $m=0,\,1,\,2,\,\cdots$, there holds $\chi_m(\lambda) \neq 0$; that is, \eqref{4514} is valid. By Lemma \ref{Lem43}, the resultant background solution $U_b$ satisfies the S-condition.
\end{proof}

\subsection{A priori estimates}\label{secp22}

For Problem $(\mathbb{L}_1)$,  we consider the nonlocal terms in \eqref{44addnew}
as nonhomogeneous terms. 
Then by Schauder estimates and interpolation inequality (see Gilbarg-Truding \cite[Chapter 6]{Gilbarg-Tudinger-1998}), if there is a solution $\hat{p}\in C^{k, \alpha}(\bar{\Omega})\,(k = 2,\,3)$, then it satisfies the following inequality
\begin{eqnarray}\label{4734}
\|\hat{p}\|_{C^{k, \alpha}(\bar{\Omega})} \leq C\left(\|\hat{p}\|_{C^{0}(\bar{\Omega})} + \|h\|_{C^{k - 2, \alpha}(\bar{\Omega})} + \|g_0\|_{C^{k - 1, \alpha}(\Sigma_0)}
 + \|g_l\|_{C^{k, \alpha}(\Sigma_l)}\right),
\end{eqnarray}
where $C$ is a positive constant that  depends only on the background solution $U_b$ and $\alpha\in (0,1)$, as well as length  $l$ of the duct. Furthermore, if the solution to Problem $(\mathbb{L}_1)$ is unique in $C^{k, \alpha}(\bar{\Omega})\,(k = 2,\,3)$, then the following estimate could be obtained from \eqref{4734}, by a contradiction-compactness argument as in \cite[Lemma 9.17]{Gilbarg-Tudinger-1998},
\begin{eqnarray}\label{4735}
\|\hat{p}\|_{C^{k, \alpha}(\bar{\Omega})} \leq C\left(\|h\|_{C^{k - 2, \alpha}(\bar{\Omega})} + \|g_0\|_{C^{k - 1, \alpha}(\Sigma_0)} + \|g_l\|_{C^{k, \alpha}(\Sigma_l)}\right).
\end{eqnarray}
See also Lemma 4.10 in \cite{Gao-Liu-Yuan-2020} for details.

\subsection{Approximate solutions} \label{secp23}
We now  construct approximate solutions to  Problem $(\mathbb{L}_1)$ by the Galerkin method used in \cite{Gao-Liu-Yuan-2020}. As indicated there, for the given $h,\,g_0$ and $g_1$, there are sequences of functions  $\{h^{(n)}\}_n$ defined in $\Omega^{\infty}\triangleq(0,l)\times(-\infty,+\infty)$, $\{g^{(n)}_0\}_n$ on $\Sigma_0^{\infty}\triangleq\{0\}\times(-\infty,+\infty)$, and $\{g^{(n)}_l\}_n$ on $\Sigma_l^{\infty}\triangleq\{l\}\times(-\infty,+\infty)$, such that
\begin{eqnarray}\label{47361}
\p_yh^{(n)}=\frac{\dd g^{(n)}_0}{\dd y}=\frac{\dd g^{(n)}_l}{\dd y}=0\quad \text{if}\quad y=k\pi,\quad k\in\mathbb{Z},
\end{eqnarray}
and as $n\rightarrow +\infty$,  $\{h^{(n)}\}_n$ converges to $h$ in $C^{1,\alpha'}(\Omega^{\infty})$, $\{g^{(n)}_0\}_n$ converges to $g_0$ in $C^{1, \alpha'}(\Sigma^{\infty}_0)$, $\{g^{(n)}_l\}_n$ converges to $g_l$ in $C^{1, \alpha'}(\Sigma^{\infty}_l)$,  whenever $0<\alpha'<\alpha$.

By $\eqref{47361}$, we have the following Fourier series: 
\begin{eqnarray}\label{4736}
\begin{cases}
h^{(n)}(\lambda, \mathbf{x}) = \frac{h^{(n)}_{0}(\lambda, x)}{2}+\sum\limits_{m = 1}^{\infty}h^{(n)}_{m}(\lambda, x)\cos(my),\\
g^{(n)}_0(y) = \frac{(g^{(n)}_0)_{0}}{2}+\sum\limits_{m=1}^{\infty}(g^{(n)}_0)_{ m}\cos(my),\\
g^{(n)}_l(y) = \frac{(g^{(n)}_l)_{ 0}}{2}+\sum\limits_{m=1}^{\infty}(g^{(n)}_l)_{ m}\cos(my).
\end{cases}
\end{eqnarray}
Taking these as non-homogeneous terms in Problem $(\mathbb{L}_1)$, and supposing that $\hat{p}$ is given by \eqref{4506}, 
then $p_{m}$ should solve the following problems:
\begin{itemize}
\item[i)]
for $m=0$,
\begin{eqnarray}\label{4737}
\begin{cases}
e_1(\lambda,x)p''_0(x)+[\lambda e_2(\lambda,x)+\pi e_6(\lambda,x)]p'_0(x)+[e_3(\lambda,x)+\pi e_7(\lambda,x)]p_0(x)\\
\quad +\lambda [e_{10}(\lambda, x)+\pi e_{12}(\lambda, x)]\int^{x}_0a_4(\lambda, t)p_0(t)\,\dd t\\
\quad +\lambda [e_{11}(\lambda, x)+\pi e_{13}(\lambda, x)]\int^{x}_0b_4(\lambda, t)p_0(t)\,\dd t=h^{(n)}_{0},\\
p_{0}'(0) + \lambda\gamma_0p_{0}(0) = (g^{(n)}_0)_{0}\\
p_{0}(l) = (g^{(n)}_l)_{0};
\end{cases}
\end{eqnarray}
\item[ii)]
for $m=1,\,2,\,\cdots$,
\begin{eqnarray}\label{4738}
\begin{cases}
e_1(\lambda,x)p_{m}'' + \lambda e_2(\lambda,x)p_{m}'+ [e_3(\lambda,x) + m^2]p_{m} + \lambda e_{10}(\lambda,x)\int^{x}_0a_4(\lambda,t)p_{m}(t)\,\dd t\\
 \quad+\lambda e_{11}(\lambda,x)\int^{x}_0b_4(\lambda,t)p_{m}(t)\,\dd t= h^{(n)}_{m}(\lambda,x),\\
p_{m}'(0) + \lambda\gamma_0p_{m}(0) = (g^{(n)}_0)_{m},\\
p_{m}(l) = (g^{(n)}_l)_{m}.
\end{cases}
\end{eqnarray}
\end{itemize}
Introducing the new unknowns  $\mathcal {P}^{(j)}_{m}(x) $ $(j=1, 2, 3, 4)$ as in \eqref{eq48add}, these problems could be written equivalently as
\begin{eqnarray}\label{4740}
\begin{cases}
\tilde{e}_1(\lambda, x){\mathcal {P}^{(4)}}'_{m}(x) + \tilde{e}_2(\lambda, x){\mathcal {P}^{(4)}}_{m}(x) + \tilde{e}_3(\lambda, x){\mathcal {P}^{(3)}}_{m}(x) + \tilde{e}_4(\lambda, x){\mathcal {P}^{(2)}}_{m}(x)\\
\quad \quad + \tilde{e}_5(\lambda, x){\mathcal {P}^{(1)}}_{m}(x) = h^{(n)}_{m},\\
{\mathcal {P}^{(3)}}'_{m}(x) = {\mathcal {P}^{(4)}}_{m}(x),\\
{\mathcal {P}^{(2)}}'_{m}(x) = a_4(\lambda,x){\mathcal {P}^{(3)}}_{m}(x),\\
{\mathcal {P}^{(1)}}'_{m}(x) = b_4(\lambda,x){\mathcal {P}^{(3)}}_{m}(x),\\
{\mathcal {P}^{(1)}}_{m}(0)=0,\\
{\mathcal {P}^{(2)}}_{m}(0) = 0,\\
{\mathcal {P}^{(4)}}_{m}(0) + \lambda\gamma_0{\mathcal {P}^{(3)}}_{m}(0)= (g^{(n)}_0)_{m},\\
{\mathcal {P}^{(3)}}_{m}(l) = (g^{(n)}_l)_{m},
\end{cases}
\end{eqnarray}
where the coefficients $\tilde{e}_i,\,(i = 1,\,2,\,3,\,4,\,5)$ were given by $\eqref{4510}$. Using notations defined in \eqref{4511}, 
problem $\eqref{4740}$ reads
\begin{eqnarray}\label{4741}
\begin{cases}
\mathbf{z}_m'= A_m\mathbf{z}_m+\Big[0,0,0,\frac{h^{(n)}_{m}(\lambda,x)}{\tilde{e}_1(\lambda, x)}\Big]^{\top},&\\
B\mathbf{z}_m(0)+C\mathbf{z}_m(l)=\Big[0,0,(g^{(n)}_0)_{m},(g^{(n)}_l)_{m}\Big]^{\top}.
\end{cases}
\end{eqnarray}
To solve it, we firstly consider the Cauchy problem 
\begin{equation}\label{4742}
\begin{cases}
\mathbf{z}_m'= A_m\mathbf{z}_m+\Big[0,0,0,\frac{h^{(n)}_{m}(\lambda,x)}{\tilde{e}_1(\lambda, x)}\Big]^{\top},&\\
\mathbf{z}_m(0)=\Big[z^1_m, z^2_m, z^3_m, z^4_m\Big]^{\top},
\end{cases}
\end{equation}
whose solution is
\begin{eqnarray}\label{4743}
\mathbf{z}_m = \Psi_m(\lambda,x)\Big[z^1_m, z^2_m, z^3_m, z^4_m\Big]^{\top} + \Psi_m(\lambda,x)\int^{x}_{0}\Psi^{- 1}_m(\lambda,t)\Big[0,0,0,\frac{h^{(n)}_{m}(\lambda,t)}{\tilde{e}_1(\lambda, t)}\Big]^{\top}\,\dd t,
\end{eqnarray}
where $\Psi_m(\lambda,x)$ is the standard fundamental matrix for $\eqref{4512}$.  For  $\eqref{4743}$ to satisfy the boundary conditions in problem $\eqref{4740}$, it requires that
\begin{eqnarray}\label{4744}
&&\Big[B + C\Psi_m(\lambda,l)\Big]\Big[z^1_m, z^2_m, z^3_m, z^4_m\Big]^{\top} = \Big[0,0,(g^{(n)}_0)_{m},(g^{(n)}_l)_{m}\Big]^{\top}\nonumber\\
&&\quad \quad - C\Psi_m(\lambda,l)\int^{l}_0\Psi^{- 1}_m(\lambda, t)\Big[0,0,0,\frac{h^{(n)}_{m}(\lambda,t)}{\tilde{e}_1(\lambda, t)}\Big]^{\top}\,\dd t.
\end{eqnarray}
By Lemma \ref{Lem43}, if the background solution satisfies the S-condition, 
we have
\begin{eqnarray}\label{4745}
\Big[z^1_m, z^2_m, z^3_m, z^4_m\Big]^{\top} = \Big[B + C\Psi_m(\lambda,l)\Big]^{-1}\Big\{\Big[0,0,(g^{(n)}_0)_{m},(g^{(n)}_l)_{m}\Big]^{\top}\nonumber\\
\quad \quad \quad\quad- C\Psi_m(\lambda,l)\int^{l}_0\Psi^{- 1}_m(\lambda, t)\Big[0,0,0,\frac{h^{(n)}_{m}(\lambda,t)}{\tilde{e}_1(\lambda,t)}\Big]^{\top}\,\dd t\Big\}.
\end{eqnarray}
Therefore, the $C^\infty$ smooth funciton $p^{(n)}_m(x)$ could be uniquely determined by $\eqref{4745}$ and \eqref{4743}.

Now set
\begin{eqnarray*}
\begin{cases}
\hat{p}^{(n)}_N(\mathbf{x}) = \frac{p^{(n)}_0(x)}{2}+\sum\limits_{m = 1}^{N}p^{(n)}_m(x)\cos(my),\\
h^{(n)}_N(\lambda,\mathbf{x}) = \frac{h^{(n)}_{0}(\lambda, x)}{2}+\sum\limits_{m = 1}^{N}h^{(n)}_{m}(\lambda, x)\cos(my),\\
(g^{(n)}_0)_N(y) = \frac{(g^{(n)}_0)_{0}}{2}+\sum\limits_{m=1}^{N}(g^{(n)}_0)_{ m}\cos(my),\\
(g^{(n)}_l)_N(y) = \frac{(g^{(n)}_l)_{ 0}}{2}+\sum\limits_{m=1}^{\infty}(g^{(n)}_l)_{ m}\cos(my).
\end{cases}
\end{eqnarray*}
It can be verified easily that $\hat{p}^{(n)}_{N}, h^{(n)}_{N}\in C^{\infty}(\bar{\Omega}^{\infty})$, $(g^{(n)}_0)_{N}\in  C^{\infty}(\Sigma_{0}^{\infty})$, $(g^{(n)}_l)_{N}\in  C^{\infty}(\Sigma_l^{\infty})$, and $\hat{p}^{(n)}_{N}$ is a solution to the following problem:
\begin{equation*}
\begin{cases}
\mathcal{L}_1(\hat{p}^{(n)}_{N}) = h^{(n)}_{N}&\quad \mathbf{x} \in \Omega^{\infty}, \\
\p_{x}\hat{p}^{(n)}_{N} + \lambda\gamma_0\hat{p}^{(n)}_{N} = (g^{(n)}_0)_{N}&\quad \mathbf{x} \in \Sigma_{0}^{\infty},\\
\hat{p}^{(n)}_{N} = (g^{(n)}_l)_{N}&\quad \mathbf{x} \in \Sigma_l^{\infty}.
\end{cases}
\end{equation*}
For any $N_1,\,N_2\in\mathbb{N}$, where $N_1 < N_2$, by the a priori estimate $\eqref{4735}$,
\begin{eqnarray}\label{4746}
&&\|\hat{p}^{(n)}_{N_2} - \hat{p}^{(n)}_{N_1}\|_{C^{k, \alpha}(\bar{\Omega})} \leq C\Big(\|h^{(n)}_{N_2} - h^{(n)}_{N_1}\|_{C^{k - 2, \alpha}(\bar{\Omega})} + \|(g^{(n)}_0)_{N_2} - (g^{(n)}_0)_{N_1}\|_{C^{k - 1, \alpha}(\Sigma_0)}\nonumber\\
&&\quad \qquad\qquad\quad\quad + \|(g^{(n)}_l)_{N_2} - (g^{(n)}_l)_{N_1}\|_{C^{k - 1, \alpha}(\Sigma_l)}\Big).
\end{eqnarray}
Noticing that as $N \rightarrow + \infty$, the Fourier series $h^{(n)}_N\rightarrow h^{(n)}$ in $C^{k - 2, \alpha}(\bar{\Omega})$, $(g^{(n)}_0)_{N}\rightarrow (g^{(n)}_0)$ in $C^{k - 1, \alpha}(\Sigma_0)$, and $(g^{(n)}_l)_{N}\rightarrow (g^{(n)}_l)$ in $C^{k - 1, \alpha}(\Sigma_{l})$. Therefore, by $\eqref{4746}$,  $\{\hat{p}^{(n)}_N\}_N$ is a Cauchy sequence in $C^{k, \alpha}(\bar{\Omega})$. So 
there exists $\hat{p}^{(n)} \in C^{k, \alpha}(\bar{\Omega})$, such that as $N \rightarrow + \infty$, $\hat{p}^{(n)}_N\rightarrow \hat{p}^{(n)}$, and $\hat{p}^{(n)}$ is the solution to the following problem:
\begin{equation}\label{4747}
\begin{cases}
\mathcal{L}_1(\hat{p}^{(n)}) = h^{(n)}&\quad \mathbf{x} \in \Omega, \\
\p_y\hat{p}^{(n)}= 0&\quad \mathbf{x} \in W_-\cup W_+,\\
\p_x\hat{p}^{(n)} + \lambda\gamma_0\hat{p}^{(n)} = g^{(n)}_0&\quad \mathbf{x} \in \Sigma_{0},\\
\hat{p}^{(n)} = g^{(n)}_l&\quad \mathbf{x} \in \Sigma_l.
\end{cases}
\end{equation}
Hence we get a sequence of approximate solutions $\{\hat{p}^{(n)}\}_n$ to Problem $(\mathbb{L}_1)$.

\subsection{Existence}\label{secp24}

Applying the estimate \eqref{4735} to the approximate solutions $\hat{p}^{(n)}$ obtained from \eqref{4747}, we have
\begin{eqnarray}\label{47355}
&&\|\hat{p}^{(n)}\|_{C^{k, \alpha}(\bar{\Omega})} \leq C\left(\|h^{(n)}\|_{C^{k - 2, \alpha}(\bar{\Omega})} + \|g_0^{(n)}\|_{C^{k - 1, \alpha}(\Sigma_0)} + \|g_l^{(n)}\|_{C^{k, \alpha}(\Sigma_l)}\right)\nonumber\\
&&\quad\qquad\qquad\qquad \leq C\left(\|h\|_{C^{k - 2, \alpha}(\bar{\Omega})} + \|g_0\|_{C^{k - 1, \alpha}(\Sigma_0)} + \|g_l\|_{C^{k, \alpha}(\Sigma_l)}\right).
\end{eqnarray}
By a standard compactness argument (cf. Lemma 4.11 in \cite{Gao-Liu-Yuan-2020} and Lemma \ref{Lem52}), there exists a subsequence $\{\hat{p}^{(n_j)}\}_j$ of $\{\hat{p}^{(n)}\}_n$, which converges to a function $\hat{p}$ in $C^{k, \alpha}(\bar{\Omega})$. 
Let $j\rightarrow +\infty$ in \eqref{4747}. Then $\hat{p}$ is a classical solution to Problem $(\mathbb{L}_1)$. Combining this with conclusions in Sections \ref{secp21}-\ref{secp23}, we proved the following theorem.
\begin{theorem}\label{Thm42}
Suppose that the background solution $U_b$ satisfies the S-condition. Then the nonlocal elliptic problem $(\mathbb{L}_1)$ has a unique solution $\hat{p} \in C^{k, \alpha}(\bar{\Omega})$, and the following estimate holds for $k=2,3,$
\begin{eqnarray}\label{4517}
\|\hat{p}\|_{C^{k, \alpha}(\bar{\Omega})} \leq C\left(\|h\|_{C^{k - 2, \alpha}(\bar{\Omega})} + \|g_0\|_{C^{k - 1, \alpha}(\Sigma_0)} + \|g_l\|_{C^{k, \alpha}(\Sigma_l)}\right).
\end{eqnarray}
\end{theorem}

\subsection{Well-posedness of Problem $(\mathbb{L})$}\label{LP}
Finally we show well-posedness of Problem $(\mathbb{L})$. Without loss of generality, suppose that $g_l(y)=0$. Otherwise, we consider $\tilde{p} = \hat{p} - g_l(y)$.
Set $$X = \{\hat{p} \in C^{k, \alpha}(\bar{\Omega}):~ \text{$\p_y(i_*\hat{p}) = 0$ on $W_-\cup W_+$ and $i_*\hat{p} = 0$ on $\Sigma_l$}\},$$
equipped with the norm $\norm{\hat{p}}_X=\norm{\hat{p}}_{C^{k,\alpha}(\bar{\Omega})}$,
where $k = 2,\,3,\,\cdots$ is fixed.
We define a linear bounded operator $\mathcal {F}$ from $X$ to $C^{k - 2, \alpha}(\bar{\Omega}) \times C^{k - 1, \alpha}(\Sigma_0)$, by $$\mathcal {F}(\hat{p})\triangleq(\mathcal {L}(\hat{p}),~ \p_x(i_*\hat{p}) + \lambda\gamma_0(i_*\hat{p})),$$ where $i_*$ is the trace operator.   We also define the linear bounded operators
$$\mathcal {F}_1(\hat{p}) \triangleq (\mathcal {L}_1(\hat{p}),~ \p_x(i_*\hat{p}) + \lambda\gamma_0(i_*\hat{p})),~~~~\qquad\mathcal {F}_2(\hat{p}) \triangleq (\mathcal {L}_2(\hat{p}), 0),$$ then  $\mathcal {F} = \mathcal {F}_1 + \mathcal {F}_2$, where the operator $\mathcal {F}_1$ is totally determined by the background solution $U_b$. Theorem \ref{Thm42} implies that if $U_b$ satisfies the S-condition, then the operator $\mathcal {F}_1$ is invertible. Therefore, if it holds that
\begin{align}\label{eq433add}
\|\mathcal {F}_2\|< \frac{1}{\|\mathcal {F}_1^{-1}\|},
\end{align}
where $\|\cdot\|$ is the operator norm, 
then $\mathcal {F}$ is also reversible, which means that Problem $(\mathbb{L})$ is well-posed. We state this as the following theorem. 

\begin{theorem}\label{Thm43}
Let  the background solution $U_b$ satisfy the S-condition, and \eqref{eq433add} holds. Suppose also that  $h \in C^{k - 2, \alpha}(\bar{\Omega}), g_0 \in C^{k - 1, \alpha}(\Sigma_0), g_l \in C^{k - 1, \alpha}(\Sigma_l)~(k = 2, 3)$. Then Problem $(\mathbb{L})$ has a unique solution $\hat{p} \in C^{k, \alpha}(\bar{\Omega})$, which obeys the estimate
\begin{eqnarray}\label{4518}
\|\hat{p}\|_{C^{k, \alpha}(\bar{\Omega})} \leq C\left(\|h\|_{C^{k - 2, \alpha}(\bar{\Omega})} + \|g_0\|_{C^{k - 1, \alpha}(\Sigma_0)} + \|g_l\|_{C^{k, \alpha}(\Sigma_l)}\right).
\end{eqnarray}
\end{theorem}
\begin{remark}\label{rmknewa}
We notice that the constants $C$ in \eqref{4517} and \eqref{4518} depend only on $U_b,\,  \alpha,\,  l$ and $k$.\qed
\end{remark}
\section{Anisotropic H\"{o}lder spaces}\label{sec53}
In the previous section we have dealt with the elliptic mode of the subsonic Euler equations. To treat the hyperbolic mode in a way compatible with the elliptic mode, we need to introduce anisotropic H\"{o}lder spaces of functions. The results in the following three sections might be of independent interests.

Let $U$ be a connected open subset in $\mathbb{R}^n$, and $\alpha=(\alpha_1, \cdots,\alpha_n)$ be a multi-index, with $|\alpha|=\sum_{i=1}^{n}\alpha_i$. For $k\ge1$, set %
\begin{eqnarray*}
C^{k}_*(\bar{U}) = \{u\in C^{k-1}(\bar{U}): ~ \text{If $|\alpha| \leq k $ and $\alpha\neq (k, 0,\cdots, 0)$, then ${\mathrm{D}}^\alpha u(x) \in C(\bar{U})$} \},
\end{eqnarray*}
which is equipped with the norm
\begin{eqnarray*}
\|u\|_{C^{k}_*(\bar{U})} = \sum\limits_{|\alpha| \leq k,\, \alpha\neq (k, 0,\cdots, 0) }\sup\limits_{x\in \bar{U}}|{\mathrm{D}}^{\alpha}u(x)|.
\end{eqnarray*}


\begin{theorem}\label{Thm5}
$(C^{k}_*(\bar{U}), \|\cdot\|_{C^{k}_*(\bar{U})})$ is a real Banach space.
\end{theorem}

\begin{proof}
It is obvious that $\|\cdot\|_{C^{k}_*(\bar{U})}$ is a norm on the linear space $C^{k}_*(\bar{U})$. We only need to prove that every Cauchy sequence $\{u_n\}$ in $C^{k}_*(\bar{U})$ is convergent.  
So for any $\epsilon > 0$, there exists $N \in\mathbb{N}$, such that when $m, n\geq N$, there holds $\|u_n - u_m\|_{C^{k}_*(\bar{U})} < \epsilon$; that is,
\begin{eqnarray*}
\sum\limits_{|\alpha| \leq k,\, \alpha\neq (k, 0,\cdots, 0) }\sup\limits_{x\in \bar{U}}|{\mathrm{D}}^{\alpha}u_n(x) - {\mathrm{D}}^{\alpha}u_m(x)| < \epsilon.
\end{eqnarray*}
Hence $\{{\mathrm{D}}^{\alpha}u_n\}$ is a Cauchy sequence in $C(\bar{U})$ whenever $\alpha\neq (k, 0,\cdots, 0)$ and $|\alpha| \leq k$. Then there exist continuous functions $u$ and $g^{\alpha}$, such that $u_n$ and ${\mathrm{D}}^{\alpha}u_n$  converge uniformly to $u$ and $g^{\alpha}$  respectively. 
By the uniform convergence of derivatives, we could change the order of limit and derivatives: 
 for any $|\alpha| \leq k,\alpha\neq (k, 0,\cdots, 0)$, there holds
\begin{eqnarray*}
{\mathrm{D}}^{\alpha}u = {\mathrm{D}}^{\alpha}(\lim\limits_{n\rightarrow + \infty} u_n) = \lim\limits_{n\rightarrow + \infty}{\mathrm{D}}^{\alpha}( u_n) = g^{\alpha}.
\end{eqnarray*}
Consequently, we proved that $\{u_n\}$ converges in $C^{k}_*(\bar{U})$ to $u$.
\end{proof}


\begin{definition}\label{Def41}
Let $U \subset \mathbb{R}^n$ be an open connected set, $0 < \lambda \leq 1$. For $u \in C^{k}_*(\bar{U})$, set
\begin{eqnarray*}
\|u\|_{C^{k, \lambda}_*(\bar{U})} = \|u\|_{C^{k}_*(\bar{U})} +  \sum\limits_{|\alpha| \le k,\,\alpha\neq (k, 0,\cdots, 0) }\sup\limits_{x,y\in \bar{U},\,x \neq y }\left\{\frac{\mid {\mathrm{D}}^{\alpha}u(x) - {\mathrm{D}}^{\alpha}u(y)\mid}{\mid x - y\mid^\lambda}\right\}.
\end{eqnarray*}
We call 
\begin{eqnarray*}
C^{k, \lambda}_*(\bar{U}) = \{u \in C^{k}_*(\bar{U}):\|u\|_{C^{k, \lambda}_*(\bar{U})} < +\infty\}
\end{eqnarray*}
as a ($x$-directional) anisotropic H\"{o}lder space.\qed
\end{definition}

\begin{theorem}\label{Thm51}
$C^{k, \lambda}_*(\bar{U})$ is a real Banach space.
\end{theorem}

\begin{proof}
It is easy to verify that $\|\cdot\|_{C^{k, \lambda}_*(\bar{U})}$ is a norm.
Let $\{u_n\}$ be a Cauchy sequence in $C^{k, \lambda}_*(\bar{U})$. Then for any $\epsilon>0$, as $m,\,n\rightarrow +\infty$, we have
\begin{eqnarray*}
\|u_n - u_m\|_{C^{k}_*(\bar{U})}<\epsilon,
\end{eqnarray*}
and
\begin{eqnarray}\label{eq61addnew1}
\sum\limits_{|\alpha| \le k,\,\alpha\neq (k, 0,\cdots, 0) }\sup\limits_{x,y\in \bar{U},\,x \neq y }\left\{\frac{\mid {\mathrm{D}}^{\alpha}u_n(x) - {\mathrm{D}}^{\alpha}u_m(x) - ({\mathrm{D}}^{\alpha}u_n(y) - {\mathrm{D}}^{\alpha}u_m(y))\mid}{\mid x - y\mid^\lambda}\right\}<\epsilon.
\end{eqnarray}
Since $(C^{k}_*(\bar{U}), \|\cdot\|_{C^{k}_*(\bar{U})})$ is complete, there exists $u \in C^{k}_*(\bar{U})$, so that $u_n$ converges to $u$ in $C^{k}_*(\bar{U})$. Therefore, for any $|\alpha| \le k,\alpha\neq (k, 0,\cdots, 0)$, the sequence $\{{\mathrm{D}}^{\alpha}u_n\}$ uniformly converges to ${\mathrm{D}}^{\alpha}u$ in $\bar{U}$. Taking $m\rightarrow +\infty$ in \eqref{eq61addnew1}, there holds
\begin{eqnarray}\label{eq62addnew1}
\sum\limits_{|\alpha| \le k,\,\alpha\neq (k, 0,\cdots, 0) }\sup\limits_{x,y\in \bar{U},\, x \neq y }\left\{\frac{\mid {\mathrm{D}}^{\alpha}u_n(x) - {\mathrm{D}}^{\alpha}u(x) - ({\mathrm{D}}^{\alpha}u_n(y) - {\mathrm{D}}^{\alpha}u(y))\mid}{\mid x - y\mid^\lambda}\right\}\le\epsilon.
\end{eqnarray}
Furthermore, $\forall x,y \in \bar{U},\,  x \neq y,\, \forall|\alpha| \le
 k,\, \alpha\neq (k, 0,\cdots, 0)$,
\begin{eqnarray*}
&&\frac{\mid {\mathrm{D}}^{\alpha}u(x) - {\mathrm{D}}^{\alpha}u(y)\mid}{\mid x - y\mid^\lambda} \leq \limsup\limits_{n\rightarrow +\infty}\frac{\mid {\mathrm{D}}^{\alpha}u_n(x) - {\mathrm{D}}^{\alpha}u_n(y)\mid}{\mid x - y\mid^\lambda} \nonumber\\
&&\quad\quad\quad\quad \leq \limsup\limits_{n\rightarrow +\infty}\sup\limits_{x,y\in \bar{U},\, x \neq y }\left\{\frac{\mid {\mathrm{D}}^{\alpha}u_n(x) - {\mathrm{D}}^{\alpha}u_n(y)\mid}{\mid x - y\mid^\lambda}\right\} < + \infty,
\end{eqnarray*}
The last inequality holds because $\sup\limits_{x,y\in \bar{U},\, x \neq y }\left\{\frac{\mid {\mathrm{D}}^{\alpha}u_n(x) - {\mathrm{D}}^{\alpha}u_n(y)\mid}{\mid x - y\mid^\lambda}\right\}$ is a Cauchy sequence, thanks to \eqref{eq61addnew1}.   Therefore we show $u \in C^{k, \lambda}_*(\bar{U})$, and \eqref{eq62addnew1} implies that $u_n\rightarrow u \,~(C^{k,\lambda}_*(\bar{U}))$.
\end{proof}

\begin{lemma}\label{Lem52}
A closed ball $B$ in $C^{m, \lambda}(\bar{U})$ is a compact set in $C^{m - 1, \lambda}(\bar{U})$.
\end{lemma}
\begin{proof}
By the Ascoli-Arzela Lemma, $B$ is a pre-compact set in $C^{m - 1, \lambda}(\bar{U})$. Let $\bar{B}$ be the closure of $B$ in $C^{m - 1, \lambda}(\bar{U})$. Then $\bar{B}$ is compact in $C^{m - 1, \lambda}(\bar{U})$. Therefore, we only need to verify $\bar{B} \subseteq B$; that is, if there is a sequence $\{a_n\} \subseteq B$ such that in $C^{m - 1, \lambda}(\bar{U})$, there holds $a_{n} \rightarrow a~(n\rightarrow + \infty)$, then one has $a\in B$.

Recall that for any $\mu \in (0, \lambda)$, $C^{m, \lambda}(\bar{U})$ can be compactly embedded in $C^{m, \mu}(\bar{U})$. Because the sequence $\{a_n\}$ is bounded in $C^{m, \lambda}(\bar{U})$, there is a subsequences $\{a_{n_k}\}$ of $\{a_n\}$ and $\tilde{a}\in C^{m, \mu}(\bar{U})$, so that $a_{n_k} \rightarrow \tilde{a}~ (k\rightarrow + \infty)$ in $C^{m, \mu}(\bar{U})$. Since $C^{m, \mu}(\bar{U})$ can be embedded in $C^{m - 1, \lambda}(\bar{U})$, thus $\|a_{n_k} - \tilde{a}\|_{C^{m - 1, \lambda}(\bar{U})}\leq C\|a_{n_k} - \tilde{a}\|_{C^{m, \mu}(\bar{U})}$; that is, we have $a_{n_k} \rightarrow \tilde{a}~(k\rightarrow + \infty)$ in $C^{m - 1, \lambda}(\bar{U})$. Thus we infer that $\tilde{a} = a$, and $a\in C^{m, \mu}(\overline{U})$ for each $0<\mu<\lambda$.

Next, we prove $a \in C^{m, \lambda}(\overline{U})$. For $|\alpha| \leq m$, and $x,y \in \bar{U}$, there are two cases.


Case 1: $0 < |x - y| < 1$. There exists $n_0 \in \mathbb{N}$ such that $2^{- n_0}\leq |x - y| < 2^{- (n_0 - 1)}$. Let $\mu = \lambda - \frac{1}{n_0}$.  Then there is a subsequence $\{a_{k'}\}$ such that $\|a_{k'} - a\|_{C^{m, \mu}(\bar{U})}\rightarrow 0\,(k'\rightarrow + \infty)$. Therefore, for any $\delta > 0$, there exists $k_0 \in \mathbb{N}$ such that $\|a_{k_0} - a\|_{C^{m, \mu}(\bar{U})} < \delta$, and
\begin{eqnarray*}
&&\frac{\mid {\mathrm{D}}^{\alpha}a(x) - {\mathrm{D}}^{\alpha}a(y)\mid}{\mid x - y\mid^\lambda} \leq \frac{\mid ({\mathrm{D}}^{\alpha}a(x) - {\mathrm{D}}^{\alpha}a(y)) - ({\mathrm{D}}^{\alpha}a_{k_0}(x) - {\mathrm{D}}^{\alpha}a_{k_0}(y))\mid}{\mid x - y\mid^\lambda}\nonumber\\
&&\qquad\qquad\qquad\qquad\qquad\qquad+\frac{\mid {\mathrm{D}}^{\alpha}a_{k_0}(x) - {\mathrm{D}}^{\alpha}a_{k_0}(y)\mid}{\mid x - y\mid^\lambda}\nonumber\\
&&\quad \leq \|a_{k_0}\|_{C^{m, \lambda}(\overline{U})} + \frac{\mid ({\mathrm{D}}^{\alpha}a(x) - {\mathrm{D}}^{\alpha}a(y)) - ({\mathrm{D}}^{\alpha}a_{k_0}(x) - {\mathrm{D}}^{\alpha}a_{k_0}(y))\mid}{\mid x - y\mid^{\lambda - \frac{1}{n_0}}} \frac{1}{\mid x - y\mid^{\frac{1}{n_0}}}\nonumber\\
&&\quad \leq 2\frac{\mid ({\mathrm{D}}^{\alpha}a(x) - {\mathrm{D}}^{\alpha}a(y)) - ({\mathrm{D}}^{\alpha}a_{k_0}(x) - {\mathrm{D}}^{\alpha}a_{k_0}(y))\mid}{\mid x - y\mid^{\mu}} + \|a_{k_0}\|_{C^{m, \lambda}(\bar{U})}\nonumber\\
&&\quad \leq 2\|a_{k_0} - a\|_{C^{m, \mu}(\overline{U})} + \|a_{k_0}\|_{C^{m, \lambda}(\overline{U})} \leq 2\delta + \|a_{k_0}\|_{C^{m, \lambda}(\bar{U})} < +\infty.
\end{eqnarray*}

Case 2:  $|x - y| \geq 1$. We just modify the above computation as follows:
\begin{eqnarray*}
&&\frac{\mid {\mathrm{D}}^{\alpha}a(x) - {\mathrm{D}}^{\alpha}a(y)\mid}{\mid x - y\mid^\lambda} \leq \frac{\mid ({\mathrm{D}}^{\alpha}a(x) - {\mathrm{D}}^{\alpha}a(y)) - ({\mathrm{D}}^{\alpha}a_{k_0}(x) - {\mathrm{D}}^{\alpha}a_{k_0}(y))\mid}{\mid x - y\mid^\lambda}\nonumber\\
&&\qquad\qquad\qquad\qquad\qquad\qquad+\frac{\mid {\mathrm{D}}^{\alpha}a_{k_0}(x) - {\mathrm{D}}^{\alpha}a_{k_0}(y)\mid}{\mid x - y\mid^\lambda}\nonumber\\
&&\quad \leq \|a_{k_0}\|_{C^{m, \lambda}(\overline{U})} + {\mid ({\mathrm{D}}^{\alpha}a(x) - {\mathrm{D}}^{\alpha}a(y)) - ({\mathrm{D}}^{\alpha}a_{k_0}(x) - {\mathrm{D}}^{\alpha}a_{k_0}(y))\mid}\nonumber\\
&&\quad \leq 2\|a_{k_0} - a\|_{C^{m, \mu}(\overline{U})} + \|a_{k_0}\|_{C^{m, \lambda}(\overline{U})} \leq 2\delta + \|a_{k_0}\|_{C^{m, \lambda}(\bar{U})} < +\infty.
\end{eqnarray*}

We then conclude that $a\in C^{m, \alpha}(\bar{U})$. Furthermore, from the above computation,  we also could infer that  $\|a\|_{C^{m, \lambda}(\bar{U})} \leq \liminf\limits_{n \rightarrow + \infty}\|a_n\|_{C^{m, \lambda}(\bar{U})}$. {This implies that $a \in B$.} 
\end{proof}

In a similar way, we could prove the following conclusion which is a counter part of Lemma \ref{Lem52} to $x$-directional anisotropic H\"{o}lder spaces.

\begin{lemma}\label{Lem53}
A closed ball in $C_*^{m, \lambda}(\bar{U})$ is a compact set in  $C_*^{m - 1, \lambda}(\bar{U})$.
\end{lemma}

\section{Solutions to transport equations with non-homogeneous terms in anisotropic H\"{o}lder spaces}\label{sec54}

Let $\Omega=[0,L]\times\Sigma=\{(x,\mathbf{y}):~0\le x\le L,\, \mathbf{y}\in\Sigma\}$ be a duct lying in $\mathbb{R}^n$ with constant  cross-section $\Sigma\subset\mathbb{R}^{n-1}$ and length $L$, and $\Sigma_0$ be the inlet, $\Sigma_L$ be the exit. The lateral wall is denoted by $\Gamma=[0,L]\times\p\Sigma$. We consider the Cauchy  problem
\begin{eqnarray}\label{4501}
\begin{cases}
\p_x\mathbf{E}(x,\mathbf{y}) + \nabla_{\mathbf{y}} \mathbf{E}(x,\mathbf{y})\cdot \mathbf{q}(x,\mathbf{y}) = \mathbf{A}(x)\mathbf{E}(x,\mathbf{y})+\mathbf{F}(x,\mathbf{y})&\quad (x,\mathbf{y})\in \Omega,\\
\mathbf{E} =\mathbf{E}_0(\mathbf{y})&\quad (x,\mathbf{y})\in \Sigma_0,
\end{cases}
\end{eqnarray}
where $\mathbf{y}\in \mathbb{R}^{n-1}$, $\nabla_{\mathbf{y}}$ is the gradient operator with respect to $\mathbf{y}$, and  $$\mathbf{q}(x,\mathbf{y})\triangleq(q_1(x,\mathbf{y}),\,q_2(x,\mathbf{y}),\,\cdots,\,q_{n-1}(x,\mathbf{y}))$$ is  an $(n-1)$-dimensional vector-valued function defined in $\bar{\Omega}$. The unknowns $\mathbf{E}$ and the nonhomogeneous terms  $\mathbf{F}$ are $m$-dimensional vector-valued function defined in $\bar{\Omega}$.  The initial data $\mathbf{E}_0$ is a $m$-dimensional vector-valued  function defined on $\Sigma_0$, and $\mathbf{A}(x)$ is a given $m\times m$ matrix-valued function defined for $x\in[0,L]$. The term $\nabla_{\mathbf{y}} \mathbf{E}(x,\mathbf{y})\cdot \mathbf{q}(x,\mathbf{y}) $ is understood as each row of the Jacobi matrix $ \nabla_{\mathbf{y}} \mathbf{E}$ taking scalar product with $\mathbf{q}$, thus obtain a column vector like $\mathbf{E}$ or $\mathbf{F}$.   We state the existence, uniqueness and regularity of solutions to problem $\eqref{4501}$ as follows.

\begin{theorem}\label{Thm41}
For given $\alpha\in(0, 1)$ and $k \in \mathbb{N}^{+}$, suppose that $\p\Sigma$ is $C^{k,\alpha}$,  $\mathbf{q}(x,\mathbf{y})$ and $\mathbf{F}\in C_*^{k, \alpha}(\bar{\Omega})$,  $\mathbf{A}\in C^{k-1, \alpha}([0,L])$, $\mathbf{E}_0 \in C^{k, \alpha}(\Sigma_0)$, and $(1,\mathbf{q}(x,\mathbf{y}))\cdot\mathbf{\nu}\equiv 0$ on $\Gamma$, with $\mathbf{\nu}$ being the unit outward normal vector on $\Gamma$.  Then problem $\eqref{4501}$ is uniquely solvable, and there exists a positive constant $C = C(L, \|\mathbf{q}\|_{C^{k, \alpha}_*(\bar{\Omega})},\,\|\mathbf{A}\|_{C^{k-1, \alpha}([0,L])})$, such that
\begin{eqnarray}\label{4502}
\|\mathbf{E}\|_{C^{k, \alpha}(\bar{\Omega})} \leq C\big(\|\mathbf{E}_0\|_{C^{k, \alpha}(\Sigma_0)} + \|\mathbf{F} \|_{C^{k, \alpha}_*(\bar{\Omega})}\big).
\end{eqnarray}
\end{theorem}
Notice that the assumption $(1,\mathbf{q}(x,\mathbf{y}))\cdot\mathbf{\nu}\equiv 0$ on $\Gamma$ ensures that the lateral wall is a  characteristic boundary of the transport equations.
\begin{proof}
Here we only present details of the proof for the case $k =1$. The general cases could be proved similarly by induction.

Let $\xi=\Upsilon(t;x,\mathbf{y})=(\gamma_1(t;x,\mathbf{y}),\,\gamma_2(t;x,\mathbf{y}),\,\cdots,\,\gamma_{n-1}(t;x,\mathbf{y}))\in\mathbb{R}^{n-1}$ be the characteristic curve of problem $\eqref{4501}$ passing through $(x,\mathbf{y})$, i.e.,
\begin{eqnarray}\label{4708}
\frac{\dd}{\dd t} \Upsilon(t;x,\mathbf{y})=\mathbf{q}(t,\Upsilon(t;x,\mathbf{y})),\qquad
\Upsilon(x;x,\mathbf{y})=\mathbf{y}.
\end{eqnarray}
Set
\begin{eqnarray*}
\Xi(x,\mathbf{y})\triangleq(\xi_1(x,\mathbf{y}),\, \xi_2(x,\mathbf{y}),\, \cdots,\, \xi_{n-1}(x,\mathbf{y}))=\Upsilon(0; x,\mathbf{y}).
\end{eqnarray*}
Then $P:\, (0,\Xi(x,\mathbf{y}))$ is the point where the  characteristic curve $\eqref{4708}$ intersects with $\Sigma_0$.

In the following we  prove that $\Upsilon(t;x,\mathbf{y})$ belongs to $C^{1,\alpha}(\overline{(0,L)\times\Omega})$.

Integrating $\eqref{4708}$ from $0$ to $x$ yields
\begin{eqnarray}\label{4709}
\Upsilon(t;x,\mathbf{y})=\mathbf{y}+\int_x^t \mathbf{q}(\tau,\Upsilon(\tau;x,\mathbf{y}))\,\dd \tau,
\end{eqnarray}
consequently,  
\begin{align}\label{4710}
\p_x\Upsilon(t;x,\mathbf{y})&=\int_x^t \p_{\Upsilon}\mathbf{q}(\tau,\Upsilon(\tau;x,\mathbf{y}))\p_x\Upsilon(\tau;x,\mathbf{y})\,\dd \tau-q(x,\Upsilon(x;x,\mathbf{y}))\nonumber\\
&= \int_x^t \p_{\Upsilon}\mathbf{q}(\tau,\Upsilon(\tau;x,\mathbf{y}))\p_x\Upsilon(\tau;x,\mathbf{y})\,\dd \tau-\mathbf{q}(x,\mathbf{y}),
\end{align}
and
\begin{eqnarray}\label{4711}
\p_{y_i}\Upsilon(t;x,\mathbf{y})=\mathbf{e}_i+\int_x^t \p_{\Upsilon}\mathbf{q}(\tau,\Upsilon(\tau;x,\mathbf{y}))\p_{y_i}\Upsilon(\tau;x,\mathbf{y})\,\dd \tau, \quad(i=1,\,2,\,\cdots,\,n-1),
\end{eqnarray}
where $\mathbf{e}_i$ refers to the column vector in $\mathbb{R}^{n-1}$ whose $i$-th component is $1$ and all the other components are $0$. In particular,
\begin{eqnarray*}
&&\left|\p_x\Upsilon(t;x,\mathbf{y})+\sum_{i=1}^{n-1}q_i(x,\mathbf{y})\p_{y_i}\Upsilon(t;x,\mathbf{y})\right|\nonumber\\
&&\quad =\Big|\int_x^t \p_{\Upsilon}\mathbf{q}(\tau,\Upsilon(\tau;x,\mathbf{y}))\big[\p_x\Upsilon(\tau;x,\mathbf{y})+\sum_{i=1}^{n-1}q_i(x,\mathbf{y})\p_{y_i}\Upsilon(\tau;x,\mathbf{y})\big]\,\dd \tau\Big|\nonumber\\
&&\quad \leq \|\mathbf{q}\|_{C^1_*(\bar{\Omega})}\int_x^t \big|\p_x\Upsilon(\tau;x,\mathbf{y}) +\sum_{i=1}^{n-1}q_i(x,\mathbf{y})\p_{y_i}\Upsilon(\tau;x,\mathbf{y})\big|\,\dd \tau.
\end{eqnarray*}
The Gronwall inequality implies that
\begin{eqnarray}
\p_x\Upsilon(t;x,\mathbf{y})+\sum_{i=1}^{n-1}q_i(x,\mathbf{y})\p_{y_i}\Upsilon(t;x,\mathbf{y})\equiv 0.
\end{eqnarray}

By $\eqref{4708}$, we directly have
\begin{eqnarray}\label{4712}
\left|\frac{\dd \Upsilon(t;x,\mathbf{y})}{\dd t}\right|=|\mathbf{q}(t,\Upsilon(\tau;x,\mathbf{y}))|\leq \|\mathbf{q}\|_{C(\bar{\Omega})}.
\end{eqnarray}
From $\eqref{4710}$,
\begin{eqnarray*}
|\p_x\Upsilon(t;x,\mathbf{y})|\leq \|\mathbf{q}\|_{C^1_*(\bar{\Omega})}\int_x^t |\p_x\Upsilon(\tau;x,\mathbf{y})|\,\dd \tau+\|\mathbf{q}\|_{C(\bar{\Omega})},
\end{eqnarray*}
thus Gronwall inequality tells us
\begin{align}\label{4713}
|\p_x\Upsilon(t;x,\mathbf{y})| &\leq \|\mathbf{q}\|_{C(\bar{\Omega})}\left(1+\|\mathbf{q}\|_{C^1_*(\bar{\Omega})}t\exp(\|\mathbf{q}\|_{C^1_*(\bar{\Omega})}t)\right) \nonumber\\&\leq\|\mathbf{q}\|_{C(\bar{\Omega})}\left(1+\|\mathbf{q}\|_{C^1_*(\bar{\Omega})}L\exp(L\|\mathbf{q}\|_{C^1_*(\bar{\Omega})})\right).
\end{align}
Similarly, by $\eqref{4711}$,
\begin{eqnarray*}
|\p_{y_i}\Upsilon(t;x,\mathbf{y})|\leq \|\mathbf{q}\|_{C^1_*(\bar{\Omega})}\int_x^t |\p_{y_i}\Upsilon(\tau;x,\mathbf{y})|\,\dd \tau+1,
\end{eqnarray*}
whence
\begin{eqnarray}\label{4714}
|\p_{y_i}\Upsilon(t;x,\mathbf{y})|\leq 1+\|\mathbf{q}\|_{C^1_*(\bar{\Omega})}t\exp(\|\mathbf{q}\|_{C^1_*(\bar{\Omega})}t) \leq 1+\|\mathbf{q}\|_{C^1_*(\bar{\Omega})}L\exp(L\|\mathbf{q}\|_{C^1_*(\bar{\Omega})}).
\end{eqnarray}
Combining $\eqref{4712},\,\eqref{4713}$ and $\eqref{4714}$, we obtain that $\Upsilon\in C^{1}(\overline{(0,L)\times\Omega})$.

Now fix $\mathbf{z}=(x,\mathbf{y})\in \Omega$. For any $t_1\neq t_2$, substituting $(t_1,\mathbf{z}),\,(t_2,\mathbf{z})$ into $\eqref{4708}$ respectively and subtracting one from the other,
\begin{eqnarray}\label{4715}
&&\left|\frac{\dd \Upsilon(t_1;\mathbf{z})}{\dd t}-\frac{\dd \Upsilon(t_2;\mathbf{z})}{\dd t}\right|=|\mathbf{q}(t_1,\Upsilon(t_1;\mathbf{z}))-\mathbf{q}(t_2,\Upsilon(t_2;\mathbf{z}))|\nonumber\\
&&\quad \leq |\mathbf{q}(t_1,\Upsilon(t_1;\mathbf{z}))-\mathbf{q}(t_2,\Upsilon(t_1;\mathbf{z}))|+|\mathbf{q}(t_2,\Upsilon(t_1;\mathbf{z}))-\mathbf{q}(t_2,\Upsilon(t_2;\mathbf{z}))|\nonumber\\
&&\quad \leq \|\mathbf{q}\|_{C^{0,\alpha}(\bar{\Omega})}|t_1-t_2|^{\alpha} + \|\mathbf{q}\|_{C^{0,\alpha}(\bar{\Omega})}\|\Upsilon\|^{\alpha}_{C^1(\overline{(0,L)\times\Omega})}|t_1-t_2|^{\alpha}\leq C|t_1-t_2|^{\alpha},
\end{eqnarray}
where $C$ is a constant  depends only on $\|\mathbf{q}\|_{C^{0,\alpha}(\bar{\Omega})}$ and $\|\Upsilon\|_{C^1(\overline{(0,L)\times\Omega})}$.

Similarly, from $\eqref{4710}$, one has, if $t_1<t_2$,
\begin{eqnarray}\label{4716}
&&|\p_x\Upsilon(t_1;\mathbf{z})-\p_x\Upsilon(t_2;\mathbf{z})|\nonumber\\
&&\quad \leq\left|\int_{x}^{t_1} \p_{\Upsilon}\mathbf{q}(\tau,\Upsilon(\tau;\mathbf{z}))\p_x\Upsilon(\tau;\mathbf{z})\,\dd \tau - \int_{x}^{t_2} \p_{\Upsilon}\mathbf{q}(\tau,\Upsilon(\tau;\mathbf{z}))\p_x\Upsilon(\tau;\mathbf{z})\,\dd \tau\right|\nonumber\\
&&\quad \leq \int_{t_1}^{t_2} \left|\p_{\Upsilon}\mathbf{q}(\tau,\Upsilon(\tau;\mathbf{z}))\p_x\Upsilon(\tau;\mathbf{z})\right|\,\dd \tau \leq \|\mathbf{q}\|_{C^{1}_*(\bar{\Omega})}\|\Upsilon\|_{C^{1}(\overline{(0,L)\times\Omega})}|t_1-t_2|\nonumber\\
&&\quad \leq C|t_1-t_2|^{\alpha},
\end{eqnarray}
with $C$  a constant depending only on $\|\mathbf{q}\|_{C^{1}_*(\bar{\Omega})}, \|\Upsilon\|_{C^{1}(\overline{(0,L)\times\Omega})}$ and $L$.
For any $i=1,\,2,\,\cdots,\,n-1$, substituting $(t_1,\mathbf{z}), (t_2,\mathbf{z})$ into $\eqref{4711}$ and subtracting, then
\begin{eqnarray}\label{4717}
&&|\p_{y_i}\Upsilon(t_1;\mathbf{z})-\p_{y_i}\Upsilon(t_2;\mathbf{z})|\nonumber\\
&&\quad =\left|\int_{x}^{t_1} \p_{\Upsilon}\mathbf{q}(\tau,\Upsilon(\tau;\mathbf{z}))\p_{y_i}\Upsilon(\tau;\mathbf{z})\,\dd \tau -\int_{x}^{t_2} \p_{\Upsilon}\mathbf{q}(\tau,\Upsilon(\tau;\mathbf{z}))\p_{y_i}\Upsilon(\tau;\mathbf{z})\,\dd \tau\right|\nonumber\\
&&\quad \leq \left|\int_{t_1}^{t_2} \p_{\Upsilon}\mathbf{q}(\tau,\Upsilon(\tau;\mathbf{z}))\p_{y_i}\Upsilon(\tau;\mathbf{z})\,\dd \tau\right|\leq \|\mathbf{q}\|_{C^{1}_*(\bar{\Omega})}\|\Upsilon\|_{C^{1}(\overline{(0,L)\times\Omega})}|t_1-t_2|\nonumber\\
&&\quad \leq C|t_1-t_2|^{\alpha}.
\end{eqnarray}

Next we fix $t$.  For any $\mathbf{z}_1=(x_1,\mathbf{y}_1), \mathbf{z}_2=(x_2,\mathbf{y}_2)\in \Omega$ and $\mathbf{z}_1\neq\mathbf{z}_2$, substituting $(t,\mathbf{z}_1)$ and $(t,\mathbf{z}_2)$ into $\eqref{4708}$ leads to
\begin{eqnarray}\label{4718}
&&\left|\frac{\dd \Upsilon(t;\mathbf{z}_1)}{\dd t}-\frac{\dd \Upsilon(t;\mathbf{z}_2)}{\dd t}\right|=|\mathbf{q}(t,\Upsilon(t;\mathbf{z}_1))-\mathbf{q}(t,\Upsilon(t;\mathbf{z}_2))|\nonumber\\
&&\qquad\qquad\quad \leq \|\mathbf{q}\|_{C^{0,\alpha}(\bar{\Omega})}\|\Upsilon\|^{\alpha}_{C^1(\overline{(0,L)\times\Omega})}|\mathbf{z}_1-\mathbf{z}_2|^{\alpha}\leq C|\mathbf{z}_1-\mathbf{z}_2|^{\alpha}.
\end{eqnarray}
Substituting $(t,\mathbf{z}_1), (t,\mathbf{z}_2)$ into $\eqref{4710}$ yields
\begin{eqnarray*}
&&|\p_x\Upsilon(t;\mathbf{z}_1)-\p_x\Upsilon(t;\mathbf{z}_2)|\nonumber\\
&&\quad \leq\left|\int_{x_1}^{t} \p_{\Upsilon}\mathbf{q}(\tau,\Upsilon(\tau;\mathbf{z}_1))\p_x\Upsilon(\tau;\mathbf{z}_1)\,\dd \tau - \int_{x_2}^{t} \p_{\Upsilon}\mathbf{q}(\tau,\Upsilon(\tau;\mathbf{z}_2))\p_x\Upsilon(\tau;\mathbf{z}_2)\,\dd \tau\right|\nonumber\\
&&\quad \quad +|\mathbf{q}(\mathbf{z}_1)-\mathbf{q}(\mathbf{z}_2)|\nonumber\\
&&\quad \leq \left|\int_{x_1}^{t} [\p_{\Upsilon}\mathbf{q}(\tau,\Upsilon(\tau;\mathbf{z}_1))-\p_{\Upsilon}\mathbf{q}(\tau,\Upsilon(\tau;\mathbf{z}_2))]\p_x\Upsilon(\tau;\mathbf{z}_1)\,\dd \tau\right|\nonumber\\
&&\quad \quad +\left|\int_{x_1}^{t} \p_{\Upsilon}\mathbf{q}(\tau,\Upsilon(\tau;\mathbf{z}_2))[\p_x\Upsilon(\tau;\mathbf{z}_1)-\p_x\Upsilon(\tau;\mathbf{z}_2)]\,\dd \tau\right|\nonumber\\
&&\quad \quad +\left|\int_{x_1}^{x_2} \p_{\Upsilon}\mathbf{q}(\tau,\Upsilon(\tau;\mathbf{z}_2))\p_x\Upsilon(\tau;\mathbf{z}_2)\,\dd \tau\right|+\|\mathbf{q}\|_{C^{0,\alpha}(\bar{\Omega})}|\mathbf{z}_1-\mathbf{z}_2|^{\alpha}\nonumber\\
&&\quad \leq L\|\mathbf{q}\|_{C^{1,\alpha}_*(\bar{\Omega})}\|\Upsilon\|^{1+\alpha}_{C^{1}(\overline{(0,L)\times\Omega})}|\mathbf{z}_1-\mathbf{z}_2|^{\alpha}+\|\mathbf{q}\|_{C^{1}_*(\bar{\Omega})}\|\Upsilon\|_{C^{1}(\overline{(0,L)\times\Omega})}|x_1-x_2|\nonumber\\
&&\quad \quad +\|\mathbf{q}\|_{C^{0,\alpha}(\bar{\Omega})}|\mathbf{z}_1-\mathbf{z}_2|^{\alpha}+\|\mathbf{q}\|_{C^{1}_*(\bar{\Omega})}\int_{x_1}^{t} |\p_x\Upsilon(\tau;\mathbf{z}_1)-\p_x\Upsilon(\tau;\mathbf{z}_2)|\,\dd \tau\nonumber\\
&&\quad \leq C_1|\mathbf{z}_1-\mathbf{z}_2|^{\alpha}+\|\mathbf{q}\|_{C^{1}_*(\bar{\Omega})}\int_{x_1}^{t} |\p_x\Upsilon(\tau;\mathbf{z}_1)-\p_x\Upsilon(\tau;\mathbf{z}_2)|\,\dd \tau,
\end{eqnarray*}
where $C_1$ is a constant depends only on $\|\mathbf{q}\|_{C^{1,\alpha}_*(\bar{\Omega})}$, $\|\Upsilon\|_{C^{1}(\overline{(0,L)\times\Omega})}$ and $L$. Thanks to Gronwall inequality, one infers
\begin{align}\label{4719}
|\p_x\Upsilon(t;\mathbf{z}_1)-\p_x\Upsilon(t;\mathbf{z}_2)|&\leq C_1\left(1+\|\mathbf{q}\|_{C^1_*(\bar{\Omega})}t\exp(\|\mathbf{q}\|_{C^1_*(\bar{\Omega})}t)\right)|\mathbf{z}_1-\mathbf{z}_2|^{\alpha} \nonumber\\ &\leq C|\mathbf{z}_1-\mathbf{z}_2|^{\alpha}.
\end{align}
where $C=C_1\left(1+\|\mathbf{q}\|_{C^1_*(\bar{\Omega})}L\exp(\|\mathbf{q}\|_{C^1_*(\bar{\Omega})}L)\right)$ also depends only on $\|\mathbf{q}\|_{C^{1,\alpha}_*(\bar{\Omega})}$, $\|\Upsilon\|_{C^{1}(\overline{(0,L)\times\Omega})}$ and $L$.

Similarly, for any $i=1,\,2,\,\cdots,\,n-1$, substituting $(t,\mathbf{z}_1),\,(t,\mathbf{z}_2)$ into $\eqref{4711}$,
\begin{eqnarray*}
&&|\p_{y_i}\Upsilon(t;\mathbf{z}_1)-\p_{y_i}\Upsilon(t;\mathbf{z}_2)|\nonumber\\
&&\quad =\left|\int_{x_1}^{t} \p_{\Upsilon}\mathbf{q}(\tau,\Upsilon(\tau;\mathbf{z}_1))\p_{y_i}\Upsilon(\tau;\mathbf{z}_1)\,\dd \tau -\int_{x_2}^{t} \p_{\Upsilon}\mathbf{q}(\tau,\Upsilon(\tau;\mathbf{z}_2))\p_{y_i}\Upsilon(\tau;\mathbf{z}_2)\,\dd \tau\right|\nonumber\\
&&\quad \leq \left|\int_{x_1}^{t} [\p_{\Upsilon}\mathbf{q}(\tau,\Upsilon(\tau;\mathbf{z}_1)) -\p_{\Upsilon}\mathbf{q}(\tau,\Upsilon(\tau;\mathbf{z}_2))]\p_{y_i}\Upsilon(\tau;\mathbf{z}_1)\,\dd \tau\right|\nonumber\\
&&\quad \quad+\left|\int_{x_1}^{t} \p_{\Upsilon}\mathbf{q}(\tau,\Upsilon(\tau;\mathbf{z}_2))[\p_{y_i}\Upsilon(\tau;\mathbf{z}_1)-\p_{y_i}\Upsilon(\tau;\mathbf{z}_2)]\,\dd \tau\right|\nonumber\\
&&\quad \quad +\left|\int_{x_1}^{x_2} \p_{\Upsilon}\mathbf{q}(\tau,\Upsilon(\tau;\mathbf{z}_2))\p_{y_i}\Upsilon(\tau;\mathbf{z}_2)\,\dd \tau\right|\nonumber\\
&&\quad \leq L\|\mathbf{q}\|_{C^{1,\alpha}_*(\bar{\Omega})}\|\Upsilon\|^{1+\alpha}_{C^{1}(\overline{(0,L)\times\Omega})}|\mathbf{z}_1-\mathbf{z}_2|^{\alpha}+\|\mathbf{q}\|_{C^{1}_*(\bar{\Omega})}\|\Upsilon\|_{C^{1}(\overline{(0,L)\times\Omega})}|x_1-x_2|\nonumber\\
&&\quad \quad +\|\mathbf{q}\|_{C^{1}_*(\bar{\Omega})}\int_{x_1}^{t} |\p_{y_i}\Upsilon(\tau;\mathbf{z}_1)-\p_{y_i}\Upsilon(\tau;\mathbf{z}_2)|\,\dd \tau\nonumber\\
&&\quad \leq C_2|\mathbf{z}_1-\mathbf{z}_2|^{\alpha}+\|\mathbf{q}\|_{C^{1}_*(\bar{\Omega})}\int_{x_1}^{t} |\p_{y_i}\Upsilon(\tau;\mathbf{z}_1)-\p_{y_i}\Upsilon(\tau;\mathbf{z}_2)|\,\dd \tau,
\end{eqnarray*}
hence by Gronwall inequality,
\begin{align}\label{4720}
|\p_{y_i}\Upsilon(t;\mathbf{z}_1)-\p_{y_i}\Upsilon(t;\mathbf{z}_2)| &\leq  C_2\left(1+\|\mathbf{q}\|_{C^1_*(\bar{\Omega})}t\exp(\|\mathbf{q}\|_{C^1_*(\bar{\Omega})}t)\right)|\mathbf{z}_1-\mathbf{z}_2|^{\alpha}\nonumber\\
& \leq C|\mathbf{z}_1-\mathbf{z}_2|^{\alpha},
\end{align}
where $C=C_2\left(1+\|\mathbf{q}\|_{C^1_*(\bar{\Omega})}L\exp(\|\mathbf{q}\|_{C^1_*(\bar{\Omega})}L)\right)$ is still a constant depends only on $L$,\,$\|\mathbf{q}\|_{C^{1,\alpha}_*(\bar{\Omega})}$ and $\|\Upsilon\|_{C^{1}(\overline{(0,L)\times\Omega})}$, while $\|\Upsilon\|_{C^{1}(\overline{(0,L)\times\Omega})}$ depends in turn on $\|\mathbf{q}\|_{C^{1,\alpha}_*(\bar{\Omega})}$.

Now for any $(t_1,\mathbf{z}_1)=(t_1,x_1,\mathbf{y}_1),~ (t_2,\mathbf{z}_2)=(t_2,x_2,\mathbf{y}_2)\in \overline{(0,L)\times\Omega}$ with $(t_1,\mathbf{z}_1)\neq(t_2,\mathbf{z}_2)$, it is obtained from $\eqref{4715}$ and $\eqref{4718}$ that
\begin{eqnarray}\label{4721}
&&\frac{\left|\frac{\dd \Upsilon(t_1;\mathbf{z}_1)}{\dd t}-\frac{\dd \Upsilon(t_2;\mathbf{z}_2)}{\dd t}\right|}{|(t_1,\mathbf{z}_1)-(t_2,\mathbf{z}_2)|^{\alpha}} \leq \frac{\left|\frac{\dd \Upsilon(t_1;\mathbf{z}_1)}{\dd t}-\frac{\dd \Upsilon(t_2;\mathbf{z}_1)}{\dd t}\right|}{|(t_1,\mathbf{z}_1)-(t_2,\mathbf{z}_2)|^{\alpha}}+\frac{\left|\frac{\dd \Upsilon(t_2;\mathbf{z}_1)}{\dd t}-\frac{\dd \Upsilon(t_2;\mathbf{z}_2)}{\dd t}\right|}{|(t_1,\mathbf{z}_1)-(t_2,\mathbf{z}_2)|^{\alpha}}\nonumber\\
&&\quad \leq C\left(\frac{|t_1-t_2|^{\alpha}}{|(t_1,\mathbf{z}_1)-(t_2,\mathbf{z}_2)|^{\alpha}}+\frac{|\mathbf{z}_1-\mathbf{z}_2|^{\alpha}}{|(t_1,\mathbf{z}_1)-(t_2,\mathbf{z}_2)|^{\alpha}}\right)\leq 2C;
\end{eqnarray}
by $\eqref{4716}$ and $\eqref{4719}$, one has
\begin{eqnarray}\label{4722}
&&\frac{\left|\p_x\Upsilon(t_1;\mathbf{z}_1)-\p_x\Upsilon(t_2;\mathbf{z}_2)\right|}{|(t_1,\mathbf{z}_1)-(t_2,\mathbf{z}_2)|^{\alpha}} \leq \frac{\left|\p_x\Upsilon(t_1;\mathbf{z}_1)-\p_x \Upsilon(t_2;\mathbf{z}_1)\right|}{|(t_1,\mathbf{z}_1)-(t_2,\mathbf{z}_2)|^{\alpha}}+\frac{\left|\p_x \Upsilon(t_2;\mathbf{z}_1)-\p_x \Upsilon(t_2;\mathbf{z}_2)\right|}{|(t_1,\mathbf{z}_1)-(t_2,\mathbf{z}_2)|^{\alpha}}\nonumber\\
&&\leq C\left(\frac{|t_1-t_2|^{\alpha}}{|(t_1,\mathbf{z}_1)-(t_2,\mathbf{z}_2)|^{\alpha}}+\frac{|\mathbf{z}_1-\mathbf{z}_2|^{\alpha}}{|(t_1,\mathbf{z}_1)-(t_2,\mathbf{z}_2)|^{\alpha}}\right)\leq 2C;
\end{eqnarray}
for any $i=1,\,2,\,\cdots,\,n-1$, by $\eqref{4717}$ and $\eqref{4720}$, one also has
\begin{eqnarray}\label{4723}
&&\frac{\left|\p_{y_i}\Upsilon(t_1;\mathbf{z}_1)-\p_{y_i}\Upsilon(t_2;\mathbf{z}_2)\right|}{|(t_1,\mathbf{z}_1)-(t_2,\mathbf{z}_2)|^{\alpha}} \leq \frac{\left|\p_{y_i}\Upsilon(t_1;\mathbf{z}_1)-\p_{y_i} \Upsilon(t_2;\mathbf{z}_1)\right|}{|(t_1,\mathbf{z}_1)-(t_2,\mathbf{z}_2)|^{\alpha}}+\frac{\left|\p_{y_i} \Upsilon(t_2;\mathbf{z}_1)-\p_{y_i} \Upsilon(t_2;\mathbf{z}_2)\right|}{|(t_1,\mathbf{z}_1)-(t_2,\mathbf{z}_2)|^{\alpha}}\nonumber\\
&&\quad \leq C\left(\frac{|t_1-t_2|^{\alpha}}{|(t_1,\mathbf{z}_1)-(t_2,\mathbf{z}_2)|^{\alpha}}+\frac{|\mathbf{z}_1-\mathbf{z}_2|^{\alpha}}{|(t_1,\mathbf{z}_1)-(t_2,\mathbf{z}_2)|^{\alpha}}\right)\leq 2C.
\end{eqnarray}
Recalling we have shown $\Upsilon\in C^{1}(\overline{(0,L)\times\Omega})$, then combined with $\eqref{4721},\,\eqref{4722}$ and $\eqref{4723}$, we proved $\Upsilon\in C^{1, \alpha}(\overline{(0,L)\times\Omega})$, with $\norm{\Upsilon}_{C^{1, \alpha}(\overline{(0,L)\times\Omega})}\le C=C(L, \|\mathbf{q}\|_{C^{1,\alpha}_*(\bar{\Omega})})$. In particular,  $\Xi(x,\mathbf{y})$ is also a function in $C^{1,\alpha}(\bar{\Omega})$.
\smallskip

Along the characteristic curve $\Upsilon(t; x, \mathbf{y})$, problem $\eqref{4501}$ can be written as
\begin{eqnarray}\label{4724}
\begin{cases}\dl
\frac{\dd \mathbf{E}(t,\Upsilon(t;x,\mathbf{y}))}{\dd t} = \mathbf{A}(t)\mathbf{E}(t,\Upsilon(t;x,\mathbf{y}))+\mathbf{F}(t,\Upsilon(t;x,\mathbf{y}))&\quad t \in (0,L),\\
\mathbf{E} =\mathbf{E}_0(\Xi(x,\mathbf{y}))&\quad t=0.
\end{cases}
\end{eqnarray}
The solution is given by
\begin{eqnarray}\label{4725}
 \mathbf{E}(x,\mathbf{y})=\Phi(x)\mathbf{E}_0(\Xi(x,\mathbf{y}))+\Phi(x)\int_0^x\Phi^{-1}(t)\mathbf{F}(t,\Upsilon(t;x,\mathbf{y}))\,\dd t,
\end{eqnarray}
where $\Phi(t)$ is the standard fundamental matrix, satisfying
\begin{eqnarray*}
\begin{cases}
\Phi'(t)=\mathbf{A}(t)\Phi(t)&\quad t \in (0,L),\\
\Phi(0)=I&\quad t=0,
\end{cases}
\end{eqnarray*}
with $I$ the $m\times m$ identity matrix.

Next, we prove that $\Phi(t)\in C^{1,\alpha}([0,L])$. By the continuity and differentiability  of solutions of ordinary differential equations on initial values and parameters, $\Phi(t)\in C^{1}([0,L])$. For any $t_1,\,t_2\in (0,L)$ with $t_1 \neq t_2$,
\begin{eqnarray*}
&&|\Phi'(t_1)-\Phi'(t_2)| = |\mathbf{A}(t_1)\Phi(t_1)-\mathbf{A}(t_2)\Phi(t_2)|\nonumber\\
&&\quad \leq |\mathbf{A}(t_1)||\Phi(t_1) -\Phi(t_2)|+|\mathbf{A}(t_1)-\mathbf{A}(t_2)||\Phi(t_2)|\nonumber\\
&&\quad \leq\|\mathbf{A}\|_{C([0,L])}\|\Phi\|_{C^{0,\alpha}([0,L])}|t_1-t_2|^{\alpha}+\|\mathbf{A}\|_{C^{0,\alpha}([0,L])}\|\Phi\|_{C([0,L])}|t_1-t_2|^{\alpha},
\end{eqnarray*}
where $C$ is a constant  depends only on $\|A\|_{C^{0,\alpha}([0,L])}$. Thus, $\Phi(t)\in C^{1,\alpha}([0,L])$.

Taking derivatives of $\eqref{4725}$ with respect to $x$ and $y_i, ~(i=1,\,2,\,\cdots,\,n-1)$, we get
\begin{align}\label{4726}
\p_x\mathbf{E}(x,\mathbf{y})=&\Phi'(x)\mathbf{E}_0(\Xi(x,\mathbf{y}))+\Phi(x)\p_{\Xi}\mathbf{E}_0(\Xi(x,\mathbf{y}))\p_x\Xi(x,\mathbf{y})\nonumber\\
& +\Phi'(x)\int_0^x\Phi^{-1}(t)\mathbf{F}(t,\Upsilon(t;x,\mathbf{y}))\,\dd t+\mathbf{F}(x,\mathbf{y})\nonumber\\
& +\Phi(x)\int_0^x\Phi^{-1}(t)\p_{\Upsilon}\mathbf{F}(t,\Upsilon(t;x,\mathbf{y}))\p_x\Upsilon(t;x,\mathbf{y})\,\dd t,
\end{align}
and
\begin{align}\label{4727}
\p_{y_i}\mathbf{E}(x,\mathbf{y})=&\Phi(x)\p_{\Xi}\mathbf{E}_0(\Xi(x,\mathbf{y}))\p_{y_i}\Xi(x,\mathbf{y})
\nonumber\\&+\Phi(x)\int_0^x\Phi^{-1}(t)\p_{\Upsilon}\mathbf{F}(t,\Upsilon(t;x,\mathbf{y}))\p_{y_i}\Upsilon(t;x,\mathbf{y})\,\dd t.
\end{align}
By virtue of \eqref{4725},
\begin{eqnarray}\label{4728}
&&|\mathbf{E}| \leq \|\Phi\|_{C([0,L])}\|\mathbf{E}_0\|_{C(\Sigma_0)}+L\|\Phi\|_{C([0,L])}\|\Phi^{-1}\|_{C([0,L])}\|\mathbf{F}\|_{C(\bar{\Omega})}\nonumber\\
&&\quad\qquad\quad \leq C\left(\|\mathbf{E}_0\|_{C(\Sigma_0)}+\|\mathbf{F}\|_{C(\bar{\Omega})}\right).
\end{eqnarray}
The constant $C$ just depends on $L$ and $\|\Phi\|_{C([0,L])}, \|\Phi^{-1}\|_{C([0,L])}$, while the latter is bounded by $\norm{A}_{C^{0,\alpha}([0,L])}$.
Thanks to \eqref{4726},
\begin{eqnarray}\label{4729}
&&|\p_x\mathbf{E}| \leq  \|\Phi\|_{C^1([0,L])}\|\mathbf{E}_0\|_{C(\Sigma_0)}+\|\Phi\|_{C([0,L])}\|\mathbf{E}_0\|_{C^1(\Sigma_0)}\|\Xi\|_{C^{1}(\bar{\Omega})}\nonumber\\
&&\quad \quad + L\|\Phi\|_{C^1([0,L])}\|\Phi^{-1}\|_{C([0,L])}\|\mathbf{F}\|_{C(\bar{\Omega})}+\|\mathbf{F}\|_{C(\bar{\Omega})}\nonumber\\
&&\quad \quad + L\|\Phi\|_{C([0,L])}\|\Phi^{-1}\|_{C([0,L])}\|\mathbf{F}\|_{C^{1}_*(\bar{\Omega})}\|\Upsilon\|_{C^{1}(\overline{(0,L)\times\Omega})}\nonumber\\
&&\quad \leq C\left(\|\mathbf{E}_0\|_{C^1(\Sigma_0)}+\|\mathbf{F}\|_{C^1_*(\bar{\Omega})}\right).
\end{eqnarray}
We notice that  $C$ only depends on $L,\, \norm{A}_{C^{0,\alpha}([0,L])}$ and $\norm{\mathbf{q}}_{C^{1,\alpha}(\bar{\Omega})}$, by previous estimates of $\norm{\Phi}_{C([0,L])}$ and $\|\Upsilon\|_{C^{1}(\overline{(0,L)\times\Omega})}$.
For any $i=1,\,2,\,\cdots,\,n-1$, utilizing $\eqref{4727}$, we get
\begin{eqnarray}\label{4730}
&&|\p_{y_i}\mathbf{E}| \leq \|\Phi\|_{C([0,L])}\|\mathbf{E}_0\|_{C^1(\Sigma_0)}\|\Xi\|_{C^{1}(\bar{\Omega})}+L\|\Phi\|_{C([0,L])}\|\Phi^{-1}\|_{C([0,L])}\|\mathbf{F}\|_{C^{1}_*(\bar{\Omega})}\|\Upsilon\|_{C^{1}(\overline{(0,L)\times\Omega})}\nonumber\\
&&\quad\qquad \leq C\left(\|\mathbf{E}_0\|_{C^1(\Sigma_0)}+\|\mathbf{F}\|_{C^1_*(\bar{\Omega})}\right).
\end{eqnarray}
Therefore, from $\eqref{4728}, \eqref{4729}$ and $\eqref{4730}$, we have
\begin{eqnarray}\label{4731}
\|\mathbf{E}\|_{C^1(\bar{\Omega})}\leq C\left(\|\mathbf{E}_0\|_{C^1(\Sigma_0)}+\|\mathbf{F}\|_{C^1_*(\bar{\Omega})}\right).
\end{eqnarray}

We then estimate the H\"{o}lder semi-norms of $\mathbf{E}$. For any $\mathbf{z}_1=(x_1,\mathbf{y}_1),\,\mathbf{z}_2=(x_2,\mathbf{y}_2)\in \bar{\Omega}$,  $\mathbf{z}_1\neq\mathbf{z}_2$, 
from $\eqref{4726}$, we have
\begin{eqnarray}\label{4732}
&&|\p_x\mathbf{E}(\mathbf{z}_1)-\p_x\mathbf{E}(\mathbf{z}_2)|\leq|\Phi'(x_1)\mathbf{E}_0(\Xi(\mathbf{z}_1))-\Phi'(x_2)\mathbf{E}_0(\Xi(\mathbf{z}_2))|\nonumber\\
&&\quad \quad +|\Phi(x_1)\p_{\Xi}\mathbf{E}_0(\xi(\mathbf{x}_1))\p_x\Xi(\mathbf{x}_1)-\Phi(x_2)\p_{\Xi}\mathbf{E}_0(\Xi(\mathbf{z}_2))\p_x\Xi(\mathbf{z}_2)|+|\mathbf{F}(\mathbf{z}_1)-\mathbf{F}(\mathbf{z}_2)|\nonumber\\
&&\quad \quad+\Big|\Phi'(x_1)\int_0^{x_1}\Phi^{-1}(t)\mathbf{F}(t,\Upsilon(t;\mathbf{z}_1))\dd t-\Phi'(x_2)\int_0^{x_2}\Phi^{-1}(t)\mathbf{F}(t,\Upsilon(t;\mathbf{z}_2))\,\dd t\Big|\nonumber\\
&&\quad \quad +\Big|\Phi(x_1)\int_0^{x_1}\Phi^{-1}(t)\p_{\Upsilon}\mathbf{F}(t,\Upsilon(t;\mathbf{z}_1))\p_x\Upsilon(t;\mathbf{z}_1)\,\dd t \nonumber\\
&&\qquad\qquad-\Phi(x_2)\int_0^{x_2}\Phi^{-1}(t)\p_{\Upsilon}\mathbf{F}(t,\Upsilon(t;\mathbf{z}_2))\p_x\Upsilon(t;\mathbf{z}_2)\,\dd t\Big|\nonumber\\
&&\quad \leq|\mathbf{E}_0(\Xi(\mathbf{z}_1))||\Phi'(x_1)-\Phi'(x_2)|+|\Phi'(x_2)||\mathbf{E}_0(\Xi(\mathbf{z}_1))-\mathbf{E}_0(\Xi(\mathbf{z}_2))|\nonumber\\
&&\quad +|\Phi(x_1)||\p_{\Xi}\mathbf{E}_0(\Xi(\mathbf{z}_2))||\p_x\Xi(\mathbf{z}_1)-\p_x\Xi(\mathbf{z}_2)|+|\p_x\Xi(\mathbf{z}_2)||\p_{\Xi}\mathbf{E}_0(\Xi(\mathbf{z}_2))||\Phi(x_1)-\Phi(x_2)|\nonumber\\
&&\quad \quad +|\Phi(x_1)||\p_x\Xi(\mathbf{z}_1)||\p_{\Xi}\mathbf{E}_0(\Xi(\mathbf{x}_1))-\p_{\Xi}\mathbf{E}_0(\xi(\mathbf{z}_2))|\nonumber\\
&&\qquad\qquad\qquad+|\Phi'(x_1)-\Phi'(x_2)|\Big|\int_0^{x_1}\Phi^{-1}(t)\mathbf{F}(t,\Upsilon(t;\mathbf{z}_1))\,\dd t\Big|\nonumber\\
&&\quad \quad +|\Phi'(x_2)|\Big|\int_0^{x_1}\Phi^{-1}(t)[\mathbf{F}(t,\Upsilon(t;\mathbf{z}_1))-\mathbf{F}(t,\Upsilon(t;\mathbf{z}_2))]\,\dd t\Big|\nonumber\\
&&\quad \quad +|\Phi'(x_2)|\Big|\int_{x_2}^{x_1}\Phi^{-1}(t)\mathbf{F}(t,\Upsilon(t;\mathbf{z}_2))\,\dd t\Big|+|\mathbf{F}(\mathbf{z}_1)-\mathbf{F}(\mathbf{z}_2)|\nonumber\\
&&\quad \quad +|\Phi(x_1)-\Phi(x_2)|\Big|\int_0^{x_1}\Phi^{-1}(t)\p_{\Upsilon}\mathbf{F}(t,\Upsilon(t;\mathbf{z}_1))\p_x\Upsilon(t;\mathbf{z}_1)\,\dd t\Big|\nonumber\\
&&\quad \quad +|\Phi(x_2)|\Big|\int_0^{x_1}\Phi^{-1}(t)\p_{\Upsilon}\mathbf{F}(t,\Upsilon(t;\mathbf{z}_1))[\p_x\Upsilon(t;\mathbf{z}_1)-\p_x\Upsilon(t;\mathbf{z}_2)]\,\dd t\Big|\nonumber\\
&&\quad \quad +|\Phi(x_2)|\Big|\int_0^{x_1}\Phi^{-1}(t)\p_x\Upsilon(t;\mathbf{z}_2)[\p_{\Upsilon}\mathbf{F}(t,\Upsilon(t;\mathbf{z}_1))-\p_{\Upsilon}\mathbf{F}(t,\Upsilon(t;\mathbf{z}_2))]\,\dd t\Big|\nonumber\\
&&\quad \quad +|\Phi(x_2)|\Big|\int_{x_2}^{x_1}\Phi^{-1}(t)\p_{\Upsilon}\mathbf{F}(t,\Upsilon(t;\mathbf{z}_2))\p_x\Upsilon(t;\mathbf{z}_2)\,\dd t\Big|\nonumber\\
&&\leq \|\mathbf{E}_0\|_{C(\Sigma_0)}\|\Phi\|_{C^{1,\alpha}([0,L])}|x_1-x_2|^{\alpha}+\|\mathbf{E}_0\|_{C^{0,\alpha}(\Sigma_0)}\|\Xi\|^{\alpha}_{C^{1}(\bar{\Omega})}\|\Phi\|_{C^{1}([0,L])}|\mathbf{z}_1-\mathbf{z}_2|^{\alpha}\nonumber\\
&&\quad \quad +\|\mathbf{E}_0\|_{C^{1}(\Sigma_0)}\|\Xi\|_{C^{1,\alpha}(\bar{\Omega})}\|\Phi\|_{C([0,L])}|\mathbf{z}_1-\mathbf{z}_2|^{\alpha}\nonumber\\
&&\qquad\qquad+\|\mathbf{E}_0\|_{C^{1}(\Sigma_0)}\|\Xi\|_{C^{1}(\bar{\Omega})}\|\Phi\|_{C^{0,\alpha}([0,L])}|x_1-x_2|^{\alpha}\nonumber\\
&&\quad  +\|\mathbf{E}_0\|_{C^{1,\alpha}(\Sigma_0)}\|\Xi\|^{1+\alpha}_{C^{1}(\bar{\Omega})}\|\Phi\|_{C([0,L])}|\mathbf{z}_1-\mathbf{z}_2|^{\alpha} \nonumber\\ &&\qquad\qquad+L\|\Phi\|_{C^{1,\alpha}([0,L])}\|\Phi^{-1}\|_{C([0,L])}\|\mathbf{F}\|_{C(\bar{\Omega})}|x_1-x_2|^{\alpha}\nonumber\\
&&\quad +L\|\Phi\|_{C^{1}([0,L])}\|\Phi^{-1}\|_{C([0,L])}\|\mathbf{F}\|_{C^{0,\alpha}(\bar{\Omega})}\|\Upsilon\|^{\alpha}_{C^{1}(\overline{(0,L)\times\Omega})}|\mathbf{z}_1-\mathbf{z}_2|^{\alpha}\nonumber\\
&&\qquad\qquad+\|\mathbf{F}\|_{C^{0,\alpha}(\bar{\Omega})}|\mathbf{z}_1-\mathbf{z}_2|^{\alpha} +\|\Phi\|_{C^{1}([0,L])}\|\Phi^{-1}\|_{C([0,L])}\|\mathbf{F}\|_{C(\bar{\Omega})}|x_1-x_2|\nonumber\\
&&\quad
+L\|\Phi\|_{C^{0,\alpha}([0,L])}\|\Phi^{-1}\|_{C([0,L])}\|\mathbf{F}\|_{C^1_*(\bar{\Omega})} \times\|\Upsilon\|_{C^{1}(\overline{(0,L)\times\Omega})}|x_1-x_2|^{\alpha}\nonumber\\
&&\quad\qquad\quad+L\|\Phi\|_{C([0,L])}\|\Phi^{-1}\|_{C([0,L])}\|\mathbf{F}\|_{C^{1}_*(\bar{\Omega})}\|\Upsilon\|_{C^{1,\alpha}(\overline{(0,L)\times\Omega})}|\mathbf{z}_1-\mathbf{z}_2|^{\alpha}\nonumber\\
&&\quad\qquad\qquad +L\|\Phi\|_{C([0,L])}\|\Phi^{-1}\|_{C([0,L])}\|\mathbf{F}\|_{C^{1,\alpha}_*(\bar{\Omega})}\|\Upsilon\|^{1+\alpha}_{C^{1}(\overline{(0,L)\times\Omega})}|\mathbf{z}_1-\mathbf{z}_2|^{\alpha}\nonumber\\
&&\quad \quad +\|\Phi\|_{C([0,L])}\|\Phi^{-1}\|_{C([0,L])}\|\mathbf{F}\|_{C^{1}_*(\bar{\Omega})}\|\Upsilon\|_{C^{1}(\overline{(0,L)\times\Omega})}|x_1-x_2|\nonumber\\
&&\quad \leq C\left(\|\mathbf{E}_0\|_{C^{1,\alpha}(\Sigma_0)}+\|\mathbf{F}\|_{C^{1,\alpha}_*(\bar{\Omega})}\right)|\mathbf{z}_1-\mathbf{z}_2|^{\alpha}.
\end{eqnarray}
For any $i=1,\,2,\,\cdots,\,n-1$, from  $\eqref{4727}$, there holds
\begin{eqnarray}\label{4733}
&&|\p_{y_i}\mathbf{E}(\mathbf{z}_1)-\p_{y_i}\mathbf{E}(\mathbf{z}_2)|\nonumber\\
&&\quad \leq|\Phi(x_1)\p_{\Xi}\mathbf{E}_0(\xi(\mathbf{x}_1))\p_{y_i}\Xi(\mathbf{x}_1)-\Phi(x_2)\p_{\Xi}\mathbf{E}_0(\Xi(\mathbf{z}_2))\p_{y_i}\Xi(\mathbf{z}_2)|\nonumber\\
&&\quad \quad +\Big|\Phi(x_1)\int_0^{x_1}\Phi^{-1}(t)\p_{\Upsilon}\mathbf{F}(t,\Upsilon(t;\mathbf{z}_1))\p_{y_i}\Upsilon(t;\mathbf{z}_1)\,\dd t\nonumber\\
&&\quad \quad \qquad\qquad -\Phi(x_2)\int_0^{x_2}\Phi^{-1}(t)\p_{\Upsilon}\mathbf{F}(t,\Upsilon(t;\mathbf{z}_2))\p_{y_i}\Upsilon(t;\mathbf{z}_2)\,\dd t\Big|\nonumber\\
&&\quad \leq|\Phi(x_1)||\p_{\Xi}\mathbf{E}_0(\Xi(\mathbf{z}_2))||\p_{y_i}\Xi(\mathbf{z}_1)-\p_{y_i}\Xi(\mathbf{z}_2)|\nonumber\\
&&\quad \quad+ |\p_{y_i}\Xi(\mathbf{z}_2)||\p_{\Xi}\mathbf{E}_0(\Xi(\mathbf{z}_2))||\Phi(x_1)-\Phi(x_2)|\nonumber\\
&&\quad \quad +|\Phi(x_1)||\p_{y_i}\Xi(\mathbf{z}_1)||\p_{\Xi}\mathbf{E}_0(\Xi(\mathbf{x}_1))-\p_{\Xi}\mathbf{E}_0(\xi(\mathbf{z}_2))|\nonumber\\
&&\quad \quad +|\Phi(x_1)-\Phi(x_2)|\Big|\int_0^{x_1}\Phi^{-1}(t)\p_{\Upsilon}\mathbf{F}(t,\Upsilon(t;\mathbf{z}_1))\p_{y_i}\Upsilon(t;\mathbf{z}_1)\,\dd t\Big|\nonumber\\
&&\quad \quad +|\Phi(x_2)|\Big|\int_0^{x_1}\Phi^{-1}(t)\p_{\Upsilon}\mathbf{F}(t,\Upsilon(t;\mathbf{z}_1))[\p_{y_i}\Upsilon(t;\mathbf{z}_1)-\p_{y_i}\Upsilon(t;\mathbf{z}_2)]\,\dd t\Big|\nonumber\\
&&\quad \quad +|\Phi(x_2)|\Big|\int_0^{x_1}\Phi^{-1}(t)\p_{y_i}\Upsilon(t;\mathbf{z}_2)[\p_{\Upsilon}\mathbf{F}(t,\Upsilon(t;\mathbf{z}_1)) -\p_{\Upsilon}\mathbf{F}(t,\Upsilon(t;\mathbf{z}_2))]\,\dd t\Big|\nonumber\\
&&\quad \quad +|\Phi(x_2)|\Big|\int_{x_2}^{x_1}\Phi^{-1}(t)\p_{\Upsilon}\mathbf{F}(t,\Upsilon(t;\mathbf{z}_2))\p_{y_i}\Upsilon(t;\mathbf{z}_2)\,\dd t\Big|\nonumber\\
&&\quad \leq \|\mathbf{E}_0\|_{C^{1}(\Sigma_0)}\|\Xi\|_{C^{1,\alpha}(\bar{\Omega})}\|\Phi\|_{C([0,L])}|\mathbf{z}_1-\mathbf{z}_2|^{\alpha}+\|\mathbf{E}_0\|_{C^{1}(\Sigma_0)}\|\Xi\|_{C^{1}(\bar{\Omega})}\nonumber\\
&&\quad \quad \times\|\Phi\|_{C^{0,\alpha}([0,L])}|x_1-x_2|^{\alpha}+\|\mathbf{E}_0\|_{C^{1,\alpha}(\Sigma_0)}\|\Xi\|^{1+\alpha}_{C^{1}(\bar{\Omega})}\|\Phi\|_{C([0,L])}|\mathbf{z}_1-\mathbf{z}_2|^{\alpha}\nonumber\\
&&\quad \quad +L\|\Phi\|_{C^{0,\alpha}([0,L])}\|\Phi^{-1}\|_{C([0,L])}\|\mathbf{F}\|_{C^1_*(\bar{\Omega})}\|\Upsilon\|_{C^{1}(\overline{(0,L)\times\Omega})}|x_1-x_2|^{\alpha}\nonumber\\
&&\quad \quad +L\|\Phi\|_{C([0,L])}\|\Phi^{-1}\|_{C([0,L])}\|\mathbf{F}\|_{C^{1}_*(\bar{\Omega})}\|\Upsilon\|_{C^{1,\alpha}(\overline{(0,L)\times\Omega})}|\mathbf{z}_1-\mathbf{z}_2|^{\alpha}\nonumber\\
&&\quad \quad +L\|\Phi\|_{C([0,L])}\|\Phi^{-1}\|_{C([0,L])}\|\mathbf{F}\|_{C^{1,\alpha}_*(\bar{\Omega})}\|\Upsilon\|^{1+\alpha}_{C^{1}(\overline{(0,L)\times\Omega})}|\mathbf{z}_1-\mathbf{z}_2|^{\alpha}\nonumber\\
&&\quad \quad +\|\Phi\|_{C([0,L])}\|\Phi^{-1}\|_{C([0,L])}\|\mathbf{F}\|_{C^{1}_*(\bar{\Omega})}\|\Upsilon\|_{C^{1}(\overline{(0,L)\times\Omega})}|x_1-x_2|\nonumber\\
&&\quad \leq C\left(\|\mathbf{E}_0\|_{C^{1,\alpha}(\Sigma_0)}+\|\mathbf{F}\|_{C^{1,\alpha}_*(\bar{\Omega})}\right)|\mathbf{z}_1-\mathbf{z}_2|^{\alpha}.
\end{eqnarray}
Similar to the computations in \eqref{4721}-\eqref{4723}, inequality \eqref{4502} could be obtained from \eqref{4731}, \eqref{4732}, and \eqref{4733}. We thus completed the proof of Theorem \ref{Thm41}.
\end{proof}

\section{Two-point boundary-value problem for ordinary differential equation of tangential velocity on cross-sections}\label{sec8}
In this section we consider the following problem, which is  to solve \eqref{4455} for the tangential velocity $v$, as needed in the nonlinear iteration scheme constructed in the next section:
\begin{eqnarray}\label{4519}
\p_y v(\mathbf{x})=f(\mathbf{x})~ \text{in}~ \Omega,\quad
v(x,0)=g_-(x)~ \text{on}~ W_-,\quad
v(x,\pi)=g_+(x)~ \text{on}~ W_+,
\end{eqnarray}
where $\Omega$ is defined by \eqref{eq16}.
It is obvious that for problem \eqref{4519} to have a solution, the following compatibility condition should hold:
\begin{eqnarray}\label{4520}
\int_0^{\pi}f(x,t)\,\,\dd t=g_+(x)-g_-(x), \qquad \forall\, x\in[0, l].
\end{eqnarray}
Then the solution is
\begin{eqnarray}\label{4521}
v(\mathbf{x})=v(x,y)=\int_0^{y}f(x,t)\,\dd t+g_-(x).
\end{eqnarray}
\begin{theorem}\label{Thm44}
For any fixed $\alpha \in (0, 1)$ and $k \in \mathbb{N}^{+}$, if $f\in C^{k - 1, \alpha}(\bar{\Omega}), g_-(x)\in C^{k, \alpha}(\overline{W}_-)$ and they satisfy the compatibility condition $\eqref{4520}$, then problem $\eqref{4519}$ is uniquely solvable, and the solution satisfies the following estimate
\begin{eqnarray}\label{4522}
\|v\|_{C_*^{k, \alpha}(\bar{\Omega})} \leq C\left(\|f\|_{C^{k - 1, \alpha}(\bar{\Omega})}+\|g_-\|_{C^{k, \alpha}(\overline{W}_-)}\right),
\end{eqnarray}
where $C=C(\alpha,\Omega)$.
\end{theorem}

\begin{proof}
The existence and uniqueness of solutions are guaranteed by the compatibility condition $\eqref{4520}$. Next, we show the estimate. We only present the details for the case of $k = 2$, as the general case could be proved similarly.

The derivative of $\eqref{4521}$ with respect to $x$ is
\begin{eqnarray*}
\p_xv(\mathbf{x})=\p_xv(x,y)=\p_x\left[\int_0^{y}f(x,t)\,\dd t+g_-(x)\right]=\int_0^{y}\p_xf(x,t)\,\dd t+g'_-(x).
\end{eqnarray*}
Then, for any $\mathbf{x}=(x_1,y_1), \, \mathbf{y}=(x_2,y_2) \in \bar{\Omega}$, satisfying $\mathbf{x}\neq\mathbf{y}$, there hold
\begin{eqnarray}\label{4523}
&&|\p_xv(\mathbf{x})-\p_xv(\mathbf{y})|=\Big|\int_0^{y_1}\p_xf(x_1,t)\,\dd t+g'_-(x_1)-\int_0^{y_2}\p_xf(x_2,t)\,\dd t-g'_-(x_2)\Big|\nonumber\\
&&\quad \leq \Big|\int_0^{y_1}[\p_xf(x_1,t)-\p_xf(x_2,t)]\,\dd t\Big|+\Big|\int_{y_1}^{y_2}\p_xf(x_2,t)\,\dd t\Big|+|g'_-(x_1)-g'_-(x_2)|\nonumber\\
&&\quad \leq \int_0^{y_1}|\p_xf(x_1,t)-\p_xf(x_2,t)|\,\dd t+\int_{y_1}^{y_2}|\p_xf(x_2,t)|\,\dd t+\|g'_-\|_{C^{1, \alpha}(\overline{W}_-)}|x_1-x_2|^{\alpha}\nonumber\\
&&\quad \leq\pi\|f\|_{C^{1, \alpha}(\bar{\Omega})}|x_1-x_2|^{\alpha}+\|f\|_{C^{1}(\bar{\Omega})}|y_1-y_2|+\|g_-\|_{C^{2, \alpha}(\bar{W}_-)}|x_1-x_2|^{\alpha}\nonumber\\
&&\quad \leq\pi\|f\|_{C^{1, \alpha}(\bar{\Omega})}|x_1-x_2|^{\alpha}+\pi^{1-\alpha}\|f\|_{C^{1}(\bar{\Omega})}|y_1-y_2|^{\alpha} +\|g_-\|_{C^{2, \alpha}(\overline{W}_-)}|x_1-x_2|^{\alpha}\nonumber\\
&&\quad \leq C\left(\|f\|_{C^{1, \alpha}(\bar{\Omega})}+\|g_-\|_{C^{2, \alpha}(\overline{W}_-)}\right)|\mathbf{x}-\mathbf{y}|^{\alpha},\\
&&\label{4524}
|\p_{xy}^2v(\mathbf{x})-\p_{xy}^2v(\mathbf{y})|=|\p_xf(\mathbf{x})-\p_xf(\mathbf{y})|\leq\|f\|_{C^{1,\alpha}}|\mathbf{x}-\mathbf{y}|^{\alpha},\\
&&\label{4525}
|\p^2_{y}v(\mathbf{x})-\p^2_{y}v(\mathbf{y})|=|\p_yf(\mathbf{x})-\p_yf(\mathbf{y})|\leq\|f\|_{C^{1,\alpha}}|\mathbf{x}-\mathbf{y}|^{\alpha}.
\end{eqnarray}
Thus, from $\eqref{4523},\,\eqref{4524}$ and $\eqref{4525}$, we get the estimate $\eqref{4522}$.
\end{proof}

\section{Iteration scheme and proof of main theorem}\label{sec45}
In this section, we prove Theorem \ref{ThmSY} by introducing a nonlinear iteration scheme based upon the coupled boundary-value problems stated in Problem \Rmnum{2},  and the well-posedness of linear problems established in the previous four sections.

\subsection{Iteration set}
For the duct  $\Omega$  defined by \eqref{eq16}, $k =2,\,3$, and $p,\,s,\,E\in C^{k, \alpha}(\bar{\Omega})$, $v \in C_*^{k, \alpha}(\bar{\Omega})$, we set $U = (p, s, E, v)$, and define the norm
\begin{eqnarray*}
\|U\|_k \triangleq \|p\|_{C^{k, \alpha}(\bar{\Omega})} + \|s\|_{C^{k, \alpha}(\bar{\Omega})} + \|E\|_{C^{k, \alpha}(\bar{\Omega})} + \|v\|_{C_*^{k, \alpha}(\bar{\Omega})}.
\end{eqnarray*}
Given a background solution $U_b$ that satisfies S-condition, let $M$ and $\epsilon$ be positive constants which are to be specified later (see \eqref{eq614add}),
depending only on $\alpha\in(0,1)$, $U_b$ and length $l$ of the duct. Set
\begin{eqnarray*}
X_{M\epsilon} \triangleq \{U:~\|U - U_b\|_3\leq M\epsilon,\,U\,\text{satisfies symmetry condition \eqref{4208}}\}.
\end{eqnarray*}
It is straightforward to check that $X_{M\epsilon}$, as a ball, is  bounded, closed and convex in $(C^{k,\alpha}(\bar{\Omega}))^3\times C_*^{k,\alpha}(\bar{\Omega})$, under the norm $\|\cdot\|_k,\,(k =2,\,3$).

In the following, we always require that
\begin{align}\label{eq61add}
M\epsilon < h_0,
\end{align} where $h_0>0$ is a constant such that for any $U\in X_{M\epsilon}$,  the coefficients of the elliptic operator $\mathcal{L}_2$ determined by such $U$  (see the terms in `$\{~\}$' in $\eqref{4449}_2$) satisfy \eqref{eq433add}, namely
$\|\mathcal{F}_2\| < \frac{1}{\|\mathcal{F}_1^{-1}\|}$. By Remark \ref{rmknewa}, $h_0$ depends only on $U_b$ and $l, \alpha$.

\subsection{Iteration scheme and nonlinear mapping $\mathcal {T}$}\label{secgzT}
We now construct a nonlinear mapping $\mathcal {T}$ on $X_{M\epsilon}$ by the following iteration scheme.
First, for any $U \in X_{M\epsilon}$, we determine the source terms $$F_0,\,F_p,\,G_p,\,p(l) - p_b(l),\,F_{E}',\,F_{s}',\,F_{v}'$$ of the problems \eqref{4601}, \eqref{4603}, and \eqref{4604} stated below. Then we solve problem $\eqref{4601}$ to get $\tilde{p}$. Substituting $\tilde{p}$ into problem $\eqref{4603}$, we get $\tilde{E}$ and $\widetilde{A(s)}$. Next, we substitute the obtained $\tilde{p},\,\tilde{E}$ and $\widetilde{A(s)}$ into problem $\eqref{4604}$ and obtain $\tilde{v}$. Then we verify that $\tilde{U}=(\tilde{p},\,\tilde{s},\,\tilde{E},\,\tilde{v})$ obtained this way satisfying the symmetry condition required in $X_{M\epsilon}$. For appropriate positive numbers $M, \epsilon$, it can be proved that the iterative mapping $\mathcal {T}: U\mapsto\tilde{U}$ is a contraction under the norm $\|\cdot\|_2$. Thanks to Schauder fixed-point theorem, the mapping $\mathcal {T}$ has a unique fixed point $U$ in $X_{M\epsilon}$ (contraction guarantees the uniqueness of the fixed point). According to the construction of mapping $\mathcal {T}$, its fixed point $U$ is a solution to \eqref{4452}-\eqref{4455} listed in Problem \Rmnum{2}. Finally, by Lemma \ref{Lem41} and the momentum equation in $y$-direction, we improve the regularity of $v$ to obtain the estimate \eqref{4209}, thus answer Problem \Rmnum{2} in an affirmative way.

We start to define the  mapping $\mathcal {T}$. For any $U\in X_{M\epsilon}$, we write $\mathcal {T}(U)=\tilde{U}$.  In the following, we solve $\hat{U} = \tilde{U} - U_b$, thus $\tilde{U}=\hat{U}+U_b$.

\subsubsection{Solving $\tilde{p}$}

We firstly determine $\tilde{p}$ from the following mixed boundary value problem:
\begin{equation}\label{4601}
\begin{cases}
\mathcal{L}(\hat{p}) = \mathcal{L}_1(\hat{p}) + \mathcal{L}_2(\hat{p}) = F_0 + F_p&\quad \text{in}\quad \Omega, \\
\p_y\hat{p}= 0&\quad \text{on}\quad W_-\cup W_+,\\
\p_x\hat{p} + \lambda\gamma_0\hat{p}=G_p&\quad \text{on}\quad \Sigma_0,\\
\hat{p}=p_l(y) - p_b(l)&\quad \text{on}\quad \Sigma_l.
\end{cases}
\end{equation}
The nonhomogeneous terms are specified as follows. The $F_0$ (see \eqref{4447}) is given by boundary conditions of $E$ and $s$ on the entrance $\Sigma_0$ (see \eqref{4202}), and the transform \eqref{eq452new} determined by the vector field $\p_x+\frac{v}{u}\p_y$, where $(u, v)$ is given by the specified $U$. It is directly verified that \begin{align*}
                                                                          \|F_0(U)\|_{C^{1, \alpha}(\bar{\Omega})} \leq C\epsilon.
                                                                           \end{align*}
For $F_p$ (see \eqref{eq463addnew2}), by $\eqref{4405}$,\,$\eqref{4408}$,\,$\eqref{4410}$,\,$\eqref{4438}$,\,$\eqref{4440}$ and $\eqref{eq462addnew12}$,  for the given $U \in X_{M\epsilon}$, the $\hat{U}$ in $F_p$ is replaced by $U - U_b$. Observing that $v\in C_*^{3,\alpha}(\bar{\Omega})$ implies $\p_xv\in C^{1,\alpha}(\bar{\Omega})$ by definition, straightforward calculation yields
\begin{eqnarray*}
\|F_p\|_{C^{1, \alpha}(\bar{\Omega})} \leq \Big\|\frac{\gamma p}{c^2_b}(F_1 + F_2)\Big\|_{C^{1, \alpha}(\bar{\Omega})} + \Big\|\frac{1}{c^2_b}F_4\Big\|_{C^{1, \alpha}(\bar{\Omega})} + \|\lambda F_6\|_{C^{1, \alpha}(\bar{\Omega})}\leq CM^2\epsilon^2.
\end{eqnarray*}
By $\eqref{4450}$, $G_p = G_1 + G_2 + G_3$ (for $G_1,\,G_2$, see $\eqref{4403}$; for $G_3$, see $\eqref{4451}$), using the inlet boundary condition $\eqref{4206}$ in Theorem \ref{ThmSY}, we get $\|G_1\|_{C^{2, \alpha}(\Sigma)}\leq CM\epsilon^2$. And $\eqref{4203}$ implies $\|G_2\|_{C^{2, \alpha}(\Sigma)}\leq C \epsilon$. For $G_3$, by the  boundary condition $\eqref{4206}$, there holds $\|G_3\|_{C^{2, \alpha}(\Sigma)}\leq C\epsilon + CM\epsilon^2 + CM^2\epsilon^2$. Therefore, 
\begin{eqnarray*}
&&\|G_p(U)\|_{C^{2, \alpha}(\Sigma_0)} \leq C \left(\|G_1(U)\|_{C^{2, \alpha}(\Sigma_0)} + \|G_2(U)\|_{C^{2, \alpha}(\Sigma_0)} + \|G_3(U)\|_{C^{2, \alpha}(\Sigma_0)}\right)\nonumber\\
&&\quad\qquad\qquad\qquad\qquad \leq C(\epsilon + M\epsilon^2 + M^2\epsilon^2).
\end{eqnarray*}
Recalling $\eqref{4206}$ in Theorem \ref{ThmSY},  $\|p_l(y) - p_b(l)\|_{C^{3, \alpha}(\Sigma_l)}\leq C\epsilon$. Then we obtain the following lemma from Theorem \ref{Thm43}.

\begin{lemma}\label{Lem45}
Problem $\eqref{4601}$ has a unique solution $\hat{p}\in C^{3, \alpha}(\bar{\Omega})$. Furthermore,  the following estimate holds:
\begin{eqnarray}\label{4602}
&&\|\hat{p}\|_{C^{3, \alpha}(\bar{\Omega})} \leq C\left(\|F_0(U)\|_{C^{1, \alpha}(\bar{\Omega})} + \|F(U)\|_{C^{1, \alpha}(\bar{\Omega})} + \|G(U)\|_{C^{2, \alpha}(\Sigma_0)}\right.\nonumber\\
&&\left.\quad \quad + \|p_l - p_b\|_{C^{3, \alpha}(\Sigma_l)}\right)\leq C (\epsilon + M\epsilon^2 + M^2\epsilon^2 ).
\end{eqnarray}
\end{lemma}
From this, we determined  $\tilde{p} = \hat{p} + p_b\in C^{3, \alpha}(\bar{\Omega})$.

\subsubsection{Solving $\tilde{E}$ and $\widetilde{A(s)}$}

Next, we specify $\tilde{E}$ and $\widetilde{A(s)}$. Consider the following Cauchy problem 
\begin{eqnarray}
\begin{cases}\label{4603}
{\mathrm{D}}'_{\mathbf{u}}\hat{E} = \lambda a_1(\lambda,x)\hat{E} + \lambda a_2(\lambda,x)\widehat{A(s)} + \lambda a_3(\lambda,x)\hat{p}+ \lambda F_{E}'&\quad \text{in}\quad \Omega,\\
{\mathrm{D}}'_{\mathbf{u}}\widehat{A(s)} =   \lambda b_1(\lambda,x) \hat{E} + \lambda b_2(\lambda,x)\widehat{A(s)} + \lambda b_3(\lambda,x)\hat{p} + \lambda F_{s}'&\quad \text{in}\quad \Omega,\\
\hat{E}=E_0(y)-E_b;\,\widehat{A(s)}=A(s_0)(y)-A(s_b)&\quad \text{on}\quad \Sigma_0.
\end{cases}
\end{eqnarray}
For the $U \in X_{M\epsilon}$, we have the velocity field $\mathbf{u}=(u,v)$ and thus ${\mathrm{D}}'_{\mathbf{u}}$ above is defined. The nonhomogeneous terms $F_{E}',\,F_{s}'$ are also specified by the $U$, and recall that they do not contain any derivative of $u,\,v$  (see $\eqref{4415}$ and $\eqref{4419}$), hence they belong to $ C^{3, \alpha}_*(\bar{\Omega})$, since the latter is a Banach algebra. Furthermore, by the restriction of $m(\mathbf{x}) - m_b(x)$ in the condition $\eqref{4206}$ of Theorem \ref{ThmSY}, we infer that
$$\|\lambda F_{E}'\|_{C^{3, \alpha}_*(\bar{\Omega})} + \|\lambda F_{s}'\|_{C^{3, \alpha}_*(\bar{\Omega})}\leq \lambda(\|F_{E}'\|_{C^{3, \alpha}_*(\bar{\Omega})} + \|F_{s}'\|_{C^{3, \alpha}_*(\bar{\Omega})})\leq C(\epsilon + M^2\epsilon^2).$$
Notice that $C$ depends on $\lambda$,  which is fixed for the given background solution that satisfies the S-condition.
From Theorem \ref{Thm41}, Cauchy problem $\eqref{4603}$ is uniquely solvable in $(C^{3, \alpha}(\bar{\Omega}))^2$, and we have the following estimate
\begin{eqnarray}\label{4604}
\|\hat{E}\|_{C^{3, \alpha}(\bar{\Omega})}+\|\widehat{A(s)}\|_{C^{3, \alpha}(\bar{\Omega})}  \leq C\Big(\|E_0 - E_b\|_{C^{3, \alpha}(\Sigma_{0})} + \|A(s_0) - A(s_b)\|_{C^{3, \alpha}(\Sigma_{0})}\nonumber\\
\phantom{=\;\;}+ \|\hat{p}\|_{C^{3, \alpha}(\bar{\Omega})} + \|\lambda F_{E}'\|_{C^{3, \alpha}_*(\bar{\Omega})} + \|\lambda F_{s}'\|_{C^{3, \alpha}_*(\bar{\Omega})}\Big) \leq C (\epsilon + M \epsilon^2 + M^2\epsilon^2).
\end{eqnarray}
Then we set $\tilde{E} = \hat{E} + E_b,~ \widetilde{A(s)}=\widehat{A(s)}+A(s_b)$, and they are both in the space $C^{3, \alpha}(\bar{\Omega})$.

\subsubsection{Solving $\tilde{v}$}

The updated tangential velocity is specified by solving the following family of two-point boundary-value problems depending on the parameter $x\in[0,l]$:
\begin{eqnarray}\label{4605}
\begin{cases}
\p_y\tilde{v}=c_1(\lambda,x)\p_x\hat{p}+c_2(\lambda,x)\hat{p}+c_3(\lambda,x)\hat{E}+c_4(\lambda,x)\widehat{A(s)}\\
\quad \quad-\frac{c_1(\lambda,x)}{\pi}\int_{0}^{\pi}\p_x\hat{p}\,\dd y -\frac{c_2(\lambda,x)}{\pi}\int_{0}^{\pi}\hat{p}\,\dd y -\frac{c_3(\lambda,x)}{\pi}\int_{0}^{\pi}\hat{E}\,\dd y\\
\quad \quad\qquad\quad-\frac{c_4(\lambda,x)}{\pi}\int_{0}^{\pi}\widehat{A(s)}\,\dd y+F_{v}'&\quad \text{in}\quad \Omega,\\
\tilde{v}=0&\quad \text{on}\quad W_-\cup W_+.
\end{cases}
\end{eqnarray}
For the $U \in X_{M\epsilon}$, the nonhomogeneous term $F_{v}'$ do not contain any derivative of $u,\,v$ with respect to $x$, and $F_{v}'\in C^{2, \alpha}(\bar{\Omega})$ (see $\eqref{4425}$). By the assumption of $m(\mathbf{x}) - m_b(x)$ in $\eqref{4206}$, there holds \begin{align}\label{eq67add}\|F'_{v}\|_{C^{2, \alpha}(\bar{\Omega})}\leq C(\epsilon + M^2\epsilon^2).\end{align}

Integrating the right-hand side of $\eqref{4605}_1$ with respect to $y$ in $[0,\pi]$, we get
\begin{eqnarray}\label{4606}
&&\int_{0}^{\pi} \Big[c_1(\lambda,x)\p_x\hat{p}+c_2(\lambda,x)\hat{p}+c_3(\lambda,x)\hat{E}+c_4(\lambda,x)\widehat{A(s)}-\frac{c_1(\lambda,x)}{\pi}\int_{0}^{\pi}\p_x\hat{p}\,\dd y\nonumber\\
&&\quad \quad -\frac{c_2(\lambda,x)}{\pi}\int_{0}^{\pi}\hat{p}\,\dd y-\frac{c_3(\lambda,x)}{\pi}\int_{0}^{\pi}\hat{E}\,\dd y-\frac{c_4(\lambda,x)}{\pi}\int_{0}^{\pi}\widehat{A(s)}\,\dd y+F_{v}'\Big]\,\dd y\equiv 0,
\end{eqnarray}
thanks to the definition $\eqref{4425}$. Thus,  problem $\eqref{4605}$ satisfies the compatibility condition $\eqref{4520}$ in Theorem \ref{Thm44}. So it has a unique solution $\tilde{v}$ defined in $\bar{\Omega}$, which enjoys the following estimate
\begin{eqnarray}\label{4607}
&&\|\tilde{v}\|_{C_*^{3, \alpha}(\bar{\Omega})} \leq C\left(\|\hat{E}\|_{C^{3, \alpha}(\bar{\Omega})} + \|\widehat{A(s)}\|_{C^{3, \alpha}(\bar{\Omega})} + \|\hat{p}\|_{C^{3, \alpha}(\bar{\Omega})} + \|F'_{v}\|_{C^{2, \alpha}(\bar{\Omega})}\right)\nonumber\\
&&\qquad\qquad\quad \leq C (\epsilon + M \epsilon^2 + M^2\epsilon^2).
\end{eqnarray}
The last inequality follows from \eqref{4602}, \eqref{4604}, and \eqref{eq67add}.

\subsubsection{Symmetry}
In the above we have obtained $\tilde{U}$. We now show that it satisfies the symmetry condition \eqref{4208} required in the definition of the iteration set $X_{M\epsilon}.$

Firstly we show that $\p_y\tilde{E}=\p_y\widetilde{A(s)}\equiv 0$ on $W_-\cup W_+$. Taking the derivative with respect to $y$ for the transport equations $\eqref{4603}_1$, $\eqref{4603}_2$,  we get
\begin{eqnarray}\label{4608}
\begin{cases}
{\mathrm{D}}'_{\mathbf{u}}(\p_y\hat{E}) = - \p_y\Big[\frac{v}{u}\Big]\p_y\hat{E} + \lambda a_1(\lambda,x)\p_y\hat{E} + \lambda a_2(\lambda,x)\p_y\widehat{A(s)}\\
\quad \quad \quad \quad + \lambda a_3(\lambda,x)\p_y\hat{p} + \lambda \p_yF_{E}'&\quad \text{in} \quad \bar{\Omega},\\
{\mathrm{D}}'_{\mathbf{u}}(\p_y\widehat{A(s)}) = - \p_y\Big[\frac{v}{u}\Big]\p_y\widehat{A(s)} + \lambda b_1(\lambda,x) \p_y\hat{E} + \lambda b_2(\lambda,x) \p_y\widehat{A(s)}\\
\quad \quad \quad \quad + \lambda b_3(\lambda,x) \p_y\hat{p} + \lambda \p_yF_{s}'&\quad \text{in} \quad \bar{\Omega}.
\end{cases}
\end{eqnarray}
The symmetry $\eqref{4207}$ guarantees that $\p_y\hat{E}=0,\,\p_y\widehat{A(s)}=0$ at the inlet $\Sigma_0\cap( W_-\cup W_+)$, while boundary conditions in \eqref{4601} and \eqref{4605} yield $\p_y \tilde{p}=\tilde{v}\equiv0$ on $W_-\cup W_+$. For the $U\in X_{M\epsilon}$, 
one  checks that $\p_yF_{E}' = \p_yF_{s}' \equiv 0$  on $W_-\cup W_+$  (cf. \eqref{4415}\eqref{4419}). Therefore, restricting $\eqref{4608}$ to $W_-\cup W_+$, we have
\begin{eqnarray}\label{4609}
\begin{cases}
{\mathrm{D}}'_{\mathbf{u}}(\p_y\hat{E}|_{W_-\cup W_+}) =  - \p_y\Big[\frac{v}{u}\Big]\p_y\hat{E}|_{W_-\cup W_+}+\lambda a_1(\lambda,x)\p_y\hat{E}|_{W_-\cup W_+} \\
\qquad\qquad+ \lambda a_2(\lambda,x)\p_y\widehat{A(s)}|_{W_-\cup W_+}&\quad x \in (0, l),\\
{\mathrm{D}}'_{\mathbf{u}}(\p_y\widehat{A(s)}|_{W_-\cup W_+}) =- \p_y\Big[\frac{v}{u}\Big]\p_y\widehat{A(s)}|_{W_-\cup W_+} + \lambda b_1(\lambda,x) \p_y\hat{E}|_{W_-\cup W_+} \\
\qquad\qquad+ \lambda b_2(\lambda,x) \p_y\widehat{A(s)}|_{W_-\cup W_+}&\quad x \in (0, l),\\
\p_y\hat{E}|_{W_-\cup W_+}=0,\quad\p_y\widehat{A(s)}|_{W_-\cup W_+}=0&\quad x=0.
\end{cases}
\end{eqnarray}
This shows that $\p_y\hat{E} = \p_y\widehat{A(s)}\equiv 0$, that is, $\p_y\tilde{E} = \p_y\widetilde{A(s)}\equiv 0$, on $W_-\cup W_+$. It follows that $\p_y\tilde{\rho}=\p_y\widetilde{c}=\p_y\tilde{u}\equiv0$ on $W_-\cup W_+$.

Next, we prove that $\p^2_y\tilde{v}\equiv 0$ on $W_-\cup W_+$. Acting $\p_y$ on $\eqref{4605}_1$,
\begin{eqnarray}\label{4610}
&&\p^2_y\tilde{v}=c_1(\lambda,x)\p_x(\p_y\hat{p})+c_2(\lambda,x)\p_y\hat{p}+c_3(\lambda,x)\p_y\hat{E}+c_4(\lambda,x)\p_y\widehat{A(s)}\nonumber\\
&&\quad \quad +\p_y\Big[-\frac{c_1(\lambda,x)}{\pi}\int_{0}^{\pi}\p_x\hat{p}\,\dd y-\frac{c_2(\lambda,x)}{\pi}\int_{0}^{\pi}\hat{p}\,\dd y -\frac{c_3(\lambda,x)}{\pi}\int_{0}^{\pi}\hat{E}\,\dd y\nonumber\\
&&\quad \quad -\frac{c_4(\lambda,x)}{\pi}\int_{0}^{\pi}\widehat{A(s)}\,\dd y\Big]+\p_y F_{v}'=c_1(\lambda,x)\p_x(\p_y\hat{p})+c_2(\lambda,x)\p_y\hat{p}\nonumber\\
&&\quad \quad +c_3(\lambda,x)\p_y\hat{E}+c_4(\lambda,x)\p_y\widehat{A(s)}+\p_y F_{v},
\end{eqnarray}
where by $\eqref{4425}$, $\p_y F_{v}'=\p_y F_{v}-\frac{1}{\pi}\p_y \int_{0}^{\pi}F_{v}\,\dd y=\p_y F_{v}$.
Therefore, by restricting $\eqref{4610}$ on $W_-\cup W_+$, get
\begin{eqnarray*}
&&\p^2_y\tilde{v}|_{W_-\cup W_+}=c_1(\lambda,x)\p_x(\p_y\hat{p})|_{W_-\cup W_+}+c_2(\lambda,x)(\p_y\hat{p})|_{W_-\cup W_+}+c_3(\lambda,x)(\p_y\hat{E})|_{W_-\cup W_+}\nonumber\\
&&\quad \quad + c_4(\lambda,x)(\p_y\widehat{A(s)})|_{W_-\cup W_+}+(\p_y F_{v})|_{W_-\cup W_+}\equiv0.
\end{eqnarray*}
Notice that from $\eqref{4422}$, there holds $(\p_y F_{v})|_{W_-\cup W_+}=0$.

Finally, we prove that $\p^3_y\tilde{p}\equiv 0$ on $W_-\cup W_+$. By $\eqref{4446}_1$, 
there holds
\begin{eqnarray*}
&&e_1(\lambda,x)\p_{x}^2(\p_y\hat{p})-\p^3_{y}\hat{p}+
e_2(\lambda,x)\p_{x}(\p_y\hat{p})+e_3(\lambda,x)(\p_y\hat{p})+\p_y\Big[e_6(\lambda,x)\nonumber\\
&&\quad \quad \times\int_{0}^{\pi}\p_x\hat{p}\,\dd y+e_7(\lambda,x)\int_{0}^{\pi}\hat{p}\,\dd y\Big]+\lambda e_{10}(\lambda, x)\int^{x}_{0}a_4(\lambda, t)\p_y\hat{p}(t, y)\,\dd t\nonumber\\
&&\quad \quad +\lambda e_{11}(\lambda, x)\int^{x}_{0}b_4(\lambda, t)\p_y\hat{p}(t, y)\,\dd t+\p_y\Big[\lambda e_{12}(\lambda,x)\int_{0}^{\pi}\int^{x}_0a_4(\lambda, t)\hat{p}(t, y)\,\dd t\dd y\nonumber\\
&&\quad \quad +\lambda e_{13}(\lambda,x)\int_{0}^{\pi}\int^{x}_0b_4(\lambda, t)\hat{p}(t, y)\,\dd t\dd y\Big]+\frac{1}{c_b^2}\Big[2\p_y\hat{E}- \frac{(\gamma + 1)\p_y\tilde{c}^2}{\gamma - 1}\Big]\p_{x}^2\hat{p}\nonumber\\
&&\quad \quad +\frac{1}{c_b^2}\Big[2\hat{E}- \frac{\gamma + 1}{\gamma - 1}(\tilde{c}^2 - c^2_b)\Big]\p_{x}^2(\p_y\hat{p})-\frac{\p_y\tilde{c}^2}{c_b^2}\p^2_{y}\hat{p}-\frac{\tilde{c}^2-c_b^2}{c_b^2}\p^3_{y}\hat{p}\nonumber\\
&&\quad \quad +\frac{\p_y\tilde{v}^2}{c_b^2}\p^2_{y}\hat{p}+\frac{\tilde{v}^2}{c_b^2}\p^3_{y}\hat{p}+\frac{2\tilde{v}\p_y\tilde{u}+2\tilde{u}\p_y\tilde{v}}{c_b^2}\p^2_{xy}\hat{p}+\frac{2\tilde{u}\tilde{v}}{c_b^2}\p^2_{xy}(\p_y\hat{p})\nonumber\\
&&\quad =\p_y F_0+\p_y F_p,
\end{eqnarray*}
that is,
\begin{eqnarray}\label{4611}
&&e_1(\lambda,x)\p_{x}^2(\p_y\hat{p})-\p^3_{y}\hat{p}+
e_2(\lambda,x)\p_{x}(\p_y\hat{p})+e_3(\lambda,x)(\p_y\hat{p})\nonumber\\
&&\quad +\lambda e_{10}(\lambda, x)\int^{x}_{0}a_4(\lambda, t)\p_y\hat{p}(t, y)\,\dd t+\lambda e_{11}(\lambda, x)\int^{x}_{0}b_4(\lambda, t)\p_y\hat{p}(t, y)\,\dd t\nonumber\\
&&\quad +\frac{1}{c_b^2}\Big[2\p_y\hat{E}- \frac{(\gamma + 1)\p_y\tilde{c}^2}{\gamma - 1}\Big]\p_{x}^2\hat{p}+\frac{1}{c_b^2}\Big[2\hat{E}- \frac{\gamma + 1}{\gamma - 1}(\tilde{c}^2 - c^2_b)\Big]\p_{x}^2(\p_y\hat{p})\nonumber\\
&&\quad +\frac{\p_y\tilde{v}^2}{c_b^2}\p^2_{y}\hat{p}+\frac{\tilde{v}^2}{c_b^2}\p^3_{y}\hat{p}+\frac{2\tilde{v}\p_y\tilde{u}+2\tilde{u}\p_y\tilde{v}}{c_b^2}\p^2_{xy}\hat{p}+\frac{2\tilde{u}\tilde{v}}{c_b^2}\p^2_{xy}(\p_y\hat{p})\nonumber\\
&&=\p_y F_0+\p_y F_p,
\end{eqnarray}
Since $\p_y\tilde{p}\equiv0$ on $W_-\cup W_+$, we see $\p_{x}(\p_y\hat{p})=\p_{x}^2(\p_y\hat{p})=0$. With the aid of the symmetry condition $\eqref{4207}$ at the inlet and the symmetry condition of $U$, we also have $\p_y F_0+\p_y F_p=0$ on $W_-\cup W_+$. Therefore, restricting $\eqref{4611}$ on $W_-\cup W_+$ yields $\p^3_{y}\hat{p}\equiv0$. Hence $\p^3_{y}\tilde{p}\equiv0$  on $W_-\cup W_+$.

From the above analysis, it is seen that for any $U\in X_{M\epsilon}$, there exists a unique $\tilde{U} = (\tilde{p},\,\tilde{s},\,\tilde{E},\,\tilde{v})$, where $\tilde{p}, \tilde{s}, \tilde{E}\in C^{3, \alpha}(\bar{\Omega})$ and $\tilde{v}\in C^{3, \alpha}_*(\bar{\Omega})$. By the estimates of $\eqref{4602}$, $\eqref{4604}$, and $\eqref{4607}$, we also have
\begin{eqnarray*}
\|\tilde{U} - U_b\|_3\leq C(\epsilon + M\epsilon^2 + M^2\epsilon^2).
\end{eqnarray*}
Selecting
\begin{align}\label{eq614add}
M = \max\{2C, 1\}\quad\text{and} \quad \epsilon_0\leq \min\{1/M(M + 1),h_0/M,1\},
\end{align}
one has
\begin{eqnarray*}
\|\tilde{U} - U_b\|_3\leq M\epsilon < h_0.
\end{eqnarray*}
Thus, we construct the desired mapping $\mathcal {T}:~ X_{M\epsilon} \rightarrow X_{M\epsilon}, U\mapsto \tilde{U}.$

\subsection{Contraction of mapping $\mathcal {T}$}\label{secysT}
In this subsection, it is proved that the mapping $\mathcal {T}$ is contractive with respect to the norm $\|\cdot\|_2$; that is, for any $U^{(1)}, ~U^{(2)}\in X_{M\epsilon}$,  there holds
\begin{eqnarray*}
\|\mathcal {T}(U^{(1)}) - \mathcal {T}(U^{(2)})\|_2 = \|\tilde{U}^{(1)} - \tilde{U}^{(2)}\|_2\leq \frac{1}{2}\|U^{(1)} - U^{(2)}\|_2,
\end{eqnarray*}
if $\epsilon$ is sufficiently small. Here, for convenience, we have set $\tilde{U}^{(i)} = \mathcal {T}(U^{(i)}),\,i = 1,\,2.$ Then $$\bar{U} \triangleq \tilde{U}^{(1)} - \tilde{U}^{(2)} = (\hat{U}^{(1)} + U_b) - (\hat{U}^{(2)} + U_b) = \hat{U}^{(1)} - \hat{U}^{(2)}.$$

\subsubsection{Estimate of $\bar{p}$}
From \eqref{4601}, we see that $\bar{p}$ is a solution to the following problem:
\begin{equation*}
\begin{cases}
\mathcal {L}^{(1)}(\bar{p}) + (\mathcal {L}^{(1)}_2 - \mathcal {L}^{(2)}_2)(\hat{p}^{(2)}) = F_p(U^{(1)}) - F_p(U^{(2)})&\quad \text{in} \quad \Omega, \\
\p_y\bar{p} = 0&\quad \text{on} \quad  W_-\cup W_+,\\
\p_x\bar{p} + \gamma_0\bar{p} = G_p(U^{(1)}) - G_p(U^{(2)})&\quad \text{on} \quad \Sigma_0,\\
\bar{p} = 0&\quad \text{on} \quad \Sigma_1,
\end{cases}
\end{equation*}
where, for $i,\,j=1,\,2$,
\begin{eqnarray*}
&&\mathcal {L}^{(i)}_2(\hat{p}^{(j)}) \triangleq \frac{1}{c_b^2}\left[2({E}^{(i)}-E_b)- \frac{\gamma + 1}{\gamma - 1}((c^{(i)})^2 - c^2_b)\right]\p_x^2\hat{p}^{(j)}-\frac{(c^{(i)})^2-c_b^2}{c_b^2}\p^2_y\hat{p}^{(j)}\nonumber\\
&&\quad \qquad\quad + \frac{(v^{(i)})^2}{c_b^2}\p^2_y\hat{p}^{(j)}+\frac{2u^{(i)}v^{(i)}}{c_b^2}\p^2_{xy}\hat{p}^{(j)}.
\end{eqnarray*}
By direct computation, we have
\begin{eqnarray}\label{46012}
&&\|F_p(U^{(1)}) - F_p(U^{(2)})\|_{C^{0, \alpha}(\bar{\Omega})} \leq \|F_5(U^{(1)}) - F_5(U^{(2)})\|_{C^{0, \alpha}(\bar{\Omega})}\nonumber\\&&\qquad\qquad\qquad+\|\lambda  F_6(U^{(1)}) - \lambda F_6(U^{(2)})\|_{C^{0, \alpha}(\bar{\Omega})}
\leq C\epsilon\|U^{(1)} - U^{(2)}\|_2,
\end{eqnarray}
and
\begin{eqnarray}\label{46112}
\|G_p(U^{(1)}) - G_p(U^{(2)})\|_{C^{1, \alpha}(\Sigma_0)}
\leq C\epsilon\|U^{(1)} - U^{(2)}\|_2.
\end{eqnarray}
Thanks to Theorem \ref{Thm42} and \eqref{46012},\,\eqref{46112}, we get
\begin{eqnarray}\label{4612}
&&\|\bar{p}\|_{C^{2, \alpha}(\bar{\Omega})} \leq C\Big(\|F_p(U^{(1)}) - F_p(U^{(2)})\|_{C^{0, \alpha}(\bar{\Omega})} + \|G_p(U^{(1)}) - G_p(U^{(2)})\|_{C^{1, \alpha}(\Sigma_0)}\nonumber\\
&&\quad \quad + (\|U^{(1)} - U^{(2)}\|_{2}\|\hat{p}^{(2)}\|_{C^{3, \alpha}(\bar{\Omega})})\Big)
\leq C\epsilon\|U^{(1)} - U^{(2)}\|_2.
\end{eqnarray}

\subsubsection{Estimate of $\bar{E}$ and $\overline{A(s)}$}
Similarly, we verify that $\bar{E}$ and $\overline{A(s)}$ are solutions to the following problem:
\begin{eqnarray}\label{46113}
\begin{cases}
{\mathrm{D}}'_{\mathbf{u}^{(1)}}\bar{E} + {\mathrm{D}}'_{\mathbf{u}^{(1)} - \mathbf{u}^{(2)}}\hat{E}^{(2)} = \lambda a_1(\lambda, x)\bar{E} + \lambda a_2(\lambda, x)\overline{A(s)}\\
\quad + \lambda a_3(\lambda, x)\bar{p} + F_{E}'(U^{(1)}) - F_{E}'(U^{(2)})&\quad \text{in} \quad \Omega,\\
{\mathrm{D}}'_{\mathbf{u}^{(1)}}\overline{A(s)} + {\mathrm{D}}'_{\mathbf{u}^{(1)} - \mathbf{u}^{(2)}}\widehat{A(s)}^{(2)} = \lambda b_1(\lambda, x)\bar{E} + \lambda b_2(\lambda, x)\overline{A(s)}\\
\quad  + \lambda b_3(\lambda, x)\bar{p} + F_{s}'(U^{(1)}) - F_{s}'(U^{(2)})&\quad \text{in} \quad \Omega,\\
\bar{E} = 0; \overline{A(s)} = 0&\quad \text{on} \quad \Sigma_0.
\end{cases}
\end{eqnarray}
By Theorem \ref{Thm41}, the following estimation can be obtained:
\begin{align}\label{4613}
&\|\bar{E}\|_{C^{2, \alpha}(\bar{\Omega})} + \|\overline{A(s)}\|_{C^{2, \alpha}(\bar{\Omega})}\nonumber\\
&  \leq C\left(\|\mathbf{u}^{(1)} - \mathbf{u}^{(2)}\|_{C^{2, \alpha}(\bar{\Omega})}\Big(\|\hat{E}^{(2)}\|_{C^{3, \alpha}(\bar{\Omega})}+ \|\widehat{A(s)}^{(2)}\|_{C^{3, \alpha}(\bar{\Omega})}\Big)\right.\nonumber\\
\phantom{=\;\;}&\quad\left. + \|\bar{p}\|_{C^{2, \alpha}(\bar{\Omega})}+ \|F_{E}'(U^{(1)}) - F_{E}'(U^{(2)})\|_{C^{2, \alpha}_*(\bar{\Omega})} + \|F_{s}'(U^{(1)}) - F_{s}'(U^{(2)})\|_{C^{2, \alpha}_*(\bar{\Omega})}\right)\nonumber\\
&\leq C\epsilon\|U^{(1)} - U^{(2)}\|_2.
\end{align}
\begin{remark}\label{RemEs}
Notice that the terms ${\mathrm{D}}'_{\mathbf{u}^{(1)} - \mathbf{u}^{(2)}}\hat{E}^{(2)}$ and ${\mathrm{D}}'_{\mathbf{u}^{(1)} - \mathbf{u}^{(2)}}\widehat{A(s)}^{(2)}$ in $\eqref{46113}$ lead to the decreasing of regularity of $\bar{E}$ and $\overline{A(s)}$. This always occurs in the Picard iteration scheme of quasi-linear hyperbolic equations. It is the reason why one can only obtain contraction in a weaker topology.
\end{remark}

\subsubsection{Estimate of $\bar{v}$}
By \eqref{4605}, $\bar{v}$ solves
\begin{eqnarray}\label{4614}
\begin{cases}
\p_y\bar{v}=c_1(\lambda,x)\p_x\bar{p}+c_2(\lambda,x)\bar{p}+c_3(\lambda,x)\bar{E}+c_4(\lambda,x)\overline{A(s)}\\
\quad \quad -\frac{c_1(\lambda,x)}{\pi}\int_{0}^{\pi}\p_x\bar{p}\,\dd y-\frac{c_2(\lambda,x)}{\pi}\int_{0}^{\pi}\bar{p}\,\dd y -\frac{c_3(\lambda,x)}{\pi}\int_{0}^{\pi}\bar{E}\,\dd y\\
\quad \quad -\frac{c_4(\lambda,x)}{\pi}\int_{0}^{\pi}\overline{A(s)}\,\dd y+F_{v}'(U^{(1)})-F_{v}'(U^{(2)})&\quad \text{in}\quad \Omega,\\
\bar{v}=0&\quad \text{on}\quad W_-\cup W_+.
\end{cases}
\end{eqnarray}
According to Theorem \ref{Thm44}, there holds
\begin{eqnarray}\label{4615}
&&\|\bar{v}\|_{C_*^{2, \alpha}(\bar{\Omega})} \leq C\left(\|\bar{E}\|_{C^{2, \alpha}(\bar{\Omega})} + \|\overline{A(s)}\|_{C^{2, \alpha}(\bar{\Omega})} + \|\bar{p}\|_{C^{2, \alpha}(\bar{\Omega})}+ \|F'_{v}(U^{(1)})-F'_{v}(U^{(2)})\|_{C^{1, \alpha}(\bar{\Omega})}\right)\nonumber\\
&&\quad\qquad\quad \leq C\epsilon\|U^{(1)} - U^{(2)}\|_2.
\end{eqnarray}
In summary, \eqref{4612}, \eqref{4613}, and \eqref{4615} give the estimate
\begin{eqnarray}\label{4616}
\|\tilde{U}^{(1)} - \tilde{U}^{(2)}\|_2\leq  C\epsilon\|U^{(1)} - U^{(2)}\|_2.
\end{eqnarray}
Taking $\epsilon_0 <\frac{1}{2C}$, then the mapping $\mathcal {T}$ contracts under the norm  $\|\cdot\|_2$.

\subsection{Proof of main theorem}
We now prove Theorem \ref{ThmSY}. It is separated into two steps. Firstly by the following theorem, we infer that there is a unique solution $U$ to Problem \Rmnum{2} in $X_{M\epsilon}$. Next, we improve the regularity of $v$ in the normal direction to show it solves Problem \Rmnum{1}.
\begin{theorem}\label{Thm47}
The mapping $\mathcal {T}$ has a unique fixed-point in $X_{M\epsilon}$.
\end{theorem}
\begin{proof}
Let $X \triangleq \{U~|~ \|U\|_2 < + \infty\}$, $Y \triangleq \{U~|~ \|U\|_3 < + \infty\}$. Then $(X, \|\cdot\|_2)$ and $(Y, \|\cdot\|_3)$ both are Banach spaces. From the embedding theorem of H\"{o}lder spaces,  $(Y, \|\cdot\|_3)$ is compactly embedded in $(X, \|\cdot\|_2)$. Therefore, $X_{M \epsilon}$ is a relatively compact set in $(X, \|\cdot\|_2)$, that is, $\bar{X}_{M \epsilon}$ is a compact set in $X_{M \epsilon}$ in $(X, \|\cdot\|_2)$, where $\bar{X}_{M \epsilon}$ is the closure of $X_{M \epsilon}$ in $(X, \|\cdot\|_2)$. From Lemma \ref{Lem52} and Lemma \ref{Lem53}, $X_{M \epsilon} = \bar{X}_{M \epsilon}$. Therefore, $X_{M \epsilon}$ is a convex compact set in $(X, \|\cdot\|_2)$.

In Section \ref{secgzT}, we have defined the mapping $\mathcal {T}$ on $X_{M\epsilon}$. In Section \ref{secysT}, we actually showed that $\mathcal {T}$ is compressed in $(X_{M\epsilon}, \|\cdot\|_2)$, thus continuous (under the norm $\|\cdot\|_2$). By Schauder fixed-point theorem \cite[Theorem 11.1]{Gilbarg-Tudinger-1998}, there is a fixed-point $U$ in $X_{M\epsilon}$ for the mapping $\mathcal {T}$. Notice that the uniqueness of fixed-point follows directly from $\eqref{4616}$.
\end{proof}

For $U\in X_{M\epsilon}$, we know that $v\in C^{3,\alpha}_*(\bar{\Omega})$, thus it losses regularity for one derivative with respect to the normal direction $x$, comparing to the tangent variable $y$. However, we could apply Lemma \ref{Lem41} to infer that $U$ is a solution to Problem \Rmnum{1}. Therefore, the equation of conservation of momentum in the $y$-direction holds, which could be written as
\begin{eqnarray*}
\p_xv =  - \frac{v}{u}\p_y v- \frac{1}{u \rho}\p_y p -\frac{\lambda m(\mathbf{x})v}{u}.
\end{eqnarray*}
From this, we easily have
\begin{eqnarray*}
\|v\|_{C^{3, \alpha}(\bar{\Omega})} \leq  C\left(\|v\|_{C_*^{3, \alpha}(\bar{\Omega})} + \| \hat{p}\|_{C^{3, \alpha}(\bar{\Omega})}\right)\leq C\epsilon.
\end{eqnarray*}
Therefore, we get $E,\,s,\,p,\,v\in C^{3,\alpha}(\bar{\Omega})$ and they satisfy
\begin{eqnarray*}
\|E - E_b\|_{C^{3,\alpha}(\bar{\Omega})} + \|s - s_b\|_{C^{3,\alpha}(\bar{\Omega})} + \|p - p_b\|_{C^{3,\alpha}(\bar{\Omega})} + \|v\|_{C^{3,\alpha}(\bar{\Omega})} \leq C\epsilon.
\end{eqnarray*}
In this way, we completed the proof of Theorem \ref{ThmSY}, thus answered Problem \Rmnum{1} affirmatively.

\section*{Acknowledgments}
This work is supported by the National Natural Science Foundation of China under Grants
No.11871218,  No.12071298, and by the Science and Technology Commission of Shanghai Municipality  under Grants No.18DZ2271000, No.21JC1402500, and by the independent design project of Zhejiang Normal University No.2021ZS08.


\end{document}